\documentclass{amsart}

\usepackage{amsmath,amssymb,amsfonts,amssymb,amsthm}

\usepackage{verbatim}
\usepackage[usenames]{color}
\usepackage{hyperref}
\usepackage{url}
\usepackage{tikz,tikz-qtree,ifthen,cancel}
\usepackage{array,tikz-qtree,ifthen,cancel}
\usepackage{graphicx}
\usepackage{adjustbox}
\usepackage{amsthm,graphicx,tikz,appendix,tikz-qtree,ifthen,cancel}
\usetikzlibrary{calc,shapes,patterns,positioning}

\newtheorem{thm}{Theorem}[section]
\newtheorem{prop}[thm]{Proposition}
\newtheorem{mainthm}[thm]{Main Theorem}
\newtheorem{lem}[thm]{Lemma}

\newtheorem*{thmR1}{Theorem \ref{thm.MillikenIPOC}}
\newtheorem*{thmR2}{Theorem \ref{thm.mainRamsey}}
\newtheorem*{thmfinalthm}{Theorem \ref{finalthm}}
\newtheorem{claim}{Claim}
\newtheorem{fact}[thm]{Fact}

\theoremstyle{remark}
\newtheorem{rem}[thm]{Remark}

\theoremstyle{definition}
\newtheorem{defn}[thm]{Definition}
\newtheorem{conv}[thm]{Convention}

\newtheorem{notation}[thm]{Notation}
\newtheorem{assumption}[thm]{Assumption}
\newtheorem{example}[thm]{Example}
\newtheorem{question}[thm]{Question}

\theoremstyle{remark}
\newtheorem*{ack}{Acknowledgments}

\newcommand{\al}{\alpha}
\newcommand{\om}{\omega}

\newcommand{\sse}{\subseteq}
\newcommand{\contains}{\supseteq}
\newcommand{\forces}{\Vdash}

\DeclareMathOperator{\ran}{ran}

\DeclareMathOperator{\stem}{stem}
\DeclareMathOperator{\Spl}{Spl}
\DeclareMathOperator{\spl}{spl}
\DeclareMathOperator{\Lev}{Lev}

\DeclareMathOperator{\Sims}{Sim}
\DeclareMathOperator{\Sim}{Sim}
\DeclareMathOperator{\Ext}{Ext}

\DeclareMathOperator{\splitpred}{splitpred}
\DeclareMathOperator{\MPE}{MPE}

\newcommand{\re}{\restriction}
\newcommand{\bP}{\mathbb{P}}

\newcommand{\bD}{\mathbb{D}}

\newcommand{\bQ}{\mathbb{Q}}

\newcommand{\bT}{\mathbb{T}}
\newcommand{\bS}{\mathbb{S}}

\newcommand{\G}{\mathrm{G}}
\newcommand{\HH}{\mathrm{H}}

\newcommand{\A}{\mathrm{A}}
\newcommand{\B}{\mathrm{B}}
\newcommand{\C}{\mathrm{C}}
\newcommand{\D}{\mathrm{D}}
\newcommand{\E}{\mathrm{E}}

\newcommand{\ssim}{\stackrel{s}{\sim}}
\newcommand{\sssim}{\stackrel{ss}{\sim}}

\newcommand{\ra}{\rightarrow}

\newcommand{\lgl}{\langle}
\newcommand{\rgl}{\rangle}

\newcommand{\POC}{Parallel $1$'s Criterion}
\newcommand{\SPOC}{Strict Parallel $1$'s Criterion}
\newcommand{\STROC}{Strong Parallel $1$'s Criterion}
\newcommand{\IPOC}{Incremental Parallel $1$'s Criterion}

\newcommand{\Nesetril}{Ne{\v{s}}et{\v{r}}il}

\newcommand{\Rodl}{R{\"{o}}dl}
\newcommand{\Erdos}{Erd{\H{o}}s}

\newcommand{\Fraisse}{Fra{\"{i}}ss{\'{e}}}
\newcommand{\Lauchli}{L{\"{a}}uchli}

\newcommand{\noprint}[1]{\relax}

\title[Ramsey theory of the   universal homogeneous triangle-free graph]{The Ramsey theory of\\   the   universal homogeneous triangle-free graph}

\author{N. Dobrinen}
\address{University of Denver\\
Department of Mathematics, 2390 S. York St., Denver, CO 80208, USA}
\email{natasha.dobrinen@du.edu}
  \urladdr{\url{http://cs.du.edu/~ndobrine}}
\thanks{This research was commenced  whilst the author was a visiting fellow at the Isaac Newton Institute for Mathematical Sciences in the programme `Mathematical, Foundational and Computational Aspects of the Higher Infinite' (HIF).
It was continued  while the author was a visitor at the Centre de Recerca Matem\`atica in the `Intensive Research Program on Large Cardinals and Strong Logics.'
The author gratefully acknowledges support from  the Isaac Newton Institute,
the Centre de Recerca Matem\`atica,
 and   National Science Foundation Grants DMS-1301665 and DMS-1600781}



\begin{document}

\maketitle

\begin{abstract}
The  universal homogeneous triangle-free graph, constructed by Henson \cite{Henson71} and  denoted $\mathcal{H}_3$,
is the triangle-free analogue of the Rado graph.
While the Ramsey theory of the Rado graph has been completely established,
beginning with \Erdos-Hajnal-Pos\'{a} \cite{Erdos/Hajnal/Posa75} and culminating in work of
  Sauer \cite{Sauer06} and Laflamme-Sauer-Vuksanovic  \cite{Laflamme/Sauer/Vuksanovic06},
 the Ramsey theory of $\mathcal{H}_3$ had only progressed to bounds for vertex colorings \cite{Komjath/Rodl86}
 and edge colorings \cite{Sauer98}.
 This was due to  a lack of broadscale techniques.

We solve this problem in general:
 For each finite triangle-free graph $\G$, there is a finite number $T(\G)$ such that for
 any coloring of all copies of $\G$ in $\mathcal{H}_3$ into finitely many colors,
there is a subgraph  of $\mathcal{H}_3$ which is again universal homogeneous triangle-free
 in which the coloring takes  no more than $T(\G)$  colors.
This is the first such result  for a homogeneous structure omitting copies of some non-trivial finite structure.
The proof entails developments of new broadscale techniques,
including
 a  flexible  method for
constructing    trees which code $\mathcal{H}_3$
and the development of their Ramsey theory.
\end{abstract}







\begin{center}
{\sc Overview}
\end{center}

 Ramsey theory  of finite structures is a well-established field with robust current activity.
Seminal examples include the classes of  finite linear orders \cite{Ramsey30},
finite Boolean algebras \cite{Graham/Rothschild71},
finite vector spaces over a finite field \cite{Graham/Leeb/Rothschild72},
finite ordered graphs \cite{Abramson/Harringon78}, \cite{Nesetril/Rodl77} and \cite{Nesetril/Rodl83},
finite ordered $k$-clique-free graphs  \cite{Nesetril/Rodl77} and \cite{Nesetril/Rodl83}, as well as many more recent advances.
{\em Homogeneous structures} are infinite structures  in which  any isomorphism between two finitely generated substructures can be extended to an automorphism of the whole structure.
A class of finite structures may have the Ramsey property, while the  homogeneous structure obtained by taking its limit may not.
The most basic example of this is linear orders.
The rational numbers are the \Fraisse\ limit of the class of all finite linear orders.
The latter has the Ramsey property, while the rationals do not:
There is a coloring of pairs of rational numbers into two colors such that every subset of the rationals forming another dense linear order without endpoints
has pairs taking each of the colors \cite{DevlinThesis}.

A  central question in the theory of  homogeneous relational structures asks the following:
Given a homogeneous structure $\bf{S}$ and
a finite substructure $\mathrm{A}$,
is there a number bound $T(\mathrm{A})$ such that for any coloring of all copies of $\mathrm{A}$ in $\bf{S}$ into finitely many colors,
there is a substructure $\bf{S}'$ of $\bf{S}$, isomorphic to  $\bf{S}$,
in which all copies of $\mathrm{A}$  take no more than $T(\mathrm{A})$ colors?
This question,  of interest for several decades since Laver's and Devlin's work on the rational numbers,  has gained recent momentum as it was
 brought into focus
 by Kechris, Pestov, and Todorcevic in \cite{Kechris/Pestov/Todorcevic05}.
This is interesting not only as a Ramsey-type property for infinite structures, but also because of its implications for topological dynamics, as shown in \cite{Zucker17}.

Prior to work in this paper, this problem  had been solved for only a few types of  homogeneous structures: the rationals (\cite{DevlinThesis}), the Rado graph and similar binary relational simple structures such as the random tournament (\cite{Sauer06}), ultrametric spaces (\cite{NVT08}), and
 enriched versions of the rationals and related circular directed graphs (\cite{Laflamme/NVT/Sauer10}).
Each of these do not omit any non-trivial substructures.
According to  \cite{NVTHabil},
``so far, the lack of tools to represent ultrahomogeneous structures is the major obstacle towards a better understanding of their infinite partition properties."
This paper
addresses this obstacle by
providing new tools to represent the universal homogeneous triangle-free graph and developing  the necessary Ramsey theory to prove upper bounds for the Ramsey degrees $T(\mathrm{A})$ for colorings of copies of a given finite triangle-free graph $\mathrm{A}$
within $\mathcal{H}_3$.
The methods  developed are robust
enough that  modifications
should likely apply to
 a large class of homogeneous structures omitting some finite substructure; particularly, in a forthcoming paper, the author is extending these methods to all $k$-clique free homogeneous graphs.

\section{Introduction}

The premise of Ramsey theory is that complete disorder is nearly impossible.
By beginning with a large enough structure, it is often possible to find  substructures in which order emerges
and persists among all smaller structures within it.
Although Ramsey-theoretic  statements are often simple,
they can be powerful tools:
in recent decades, the heart of  many  problems
in mathematics
have turned out to have  at their core some Ramsey-theoretic content.
This has been seen clearly in  Banach spaces and topological dynamics.

The field of Ramsey theory opened with the following celebrated result.
\begin{thm}[Ramsey, \cite{Ramsey30}]\label{thm.RamseyInfinite}
Let $k$ and $r$ be positive integers, and suppose $P_i$,  $i<r$, is a partition of all $k$-element subsets of $\mathbb{N}$.
Then there is an infinite subset $M$ of natural numbers
and some $i<r$
such that all $k$-element subsets of $M$ lie in $P_i$.
\end{thm}

The finite version of Ramsey's Theorem states that given positive integers $k,m,r$, there is a number $n$ large enough so that
given any partition of the $k$-element subsets of $\{0,\dots, n-1\}$ into $r$ pieces,
there is a subset $X$ of $\{0,\dots,n-1\}$ of size $m$ such that all $k$-element subsets of $X$ lie in one piece of the partition.
This follows from the infinite version using a  compactness argument.
The set $X$ is called {\em homogeneous} for the given partition.

The idea of partitioning  certain subsets of a given finite  set and looking for a large homogeneous subset has been extended to structures.
A \Fraisse\ class $\mathcal{K}$ of finite structures is said to have the {\em Ramsey property}
if for any $\A,\B\in\mathcal{K}$ with $\A$ embedding into $\B$,
 (written $A\le B$), and for any finite number $k$,
there is a finite ordered graph $\C$ such that
for any coloring of the copies of $\A$ in $\C$ into $k$ colors, there is a copy $\B'\le C$ of $\B$  such that all copies of $\A$ in $\B'$ have the same color.
We use the standard notation
\begin{equation}\C\ra (\B)^{\A}_k
\end{equation}
to denote that for any coloring of the copies of $A$ in $C$, there is a copy $B'$ of $B$ inside $C$ such that all copies of $A$ in $C$ have the same color.
Examples of \Fraisse\ classes of finite structures
 with the Ramsey property, having no extra relations, include
finite Boolean algebras (Graham and Rothschild, \cite{Graham/Rothschild71}) and
finite vector spaces over a finite field (Graham, Leeb, and Rothschild, \cite{Graham/Leeb/Rothschild72} and
\cite{Graham/Leeb/Rothschild73}).
Examples  of \Fraisse\ classes with extra structure satisfying the Ramsey  property include
 finite  ordered relational structures  (independently, Abramson and Harrington, \cite{Abramson/Harringon78} and \Nesetril\ and \Rodl, \cite{Nesetril/Rodl77}, \cite{Nesetril/Rodl83}).
In particular, this includes the class of finite ordered graphs, denoted $\mathcal{G}^{<}$.
The papers \cite{Nesetril/Rodl77} and \cite{Nesetril/Rodl83} further proved  the quite general result that  all set-systems of  finite ordered relational structures omitting some irreducible substructure have the Ramsey property.
This includes
 the \Fraisse\ class of finite ordered graphs omitting $n$-cliques, denoted $\mathcal{K}_n^{<}$.

In contrast, the \Fraisse\ class of unordered finite graphs  does not have the Ramsey property.
However, it does posses a non-trivial remnant of the Ramsey property, called
  finite Ramsey degrees.
Given any \Fraisse\ class $\mathcal{K}$ of finite structures,
for each $A\in\mathcal{K}$, let
$t(A,\mathcal{K})$  be the smallest number $t$, if it exists, such that
for each $B\in \mathcal{K}$ with $A\le B$ and for each $k\ge 2$,
there is some $C\in\mathcal{K}$, into which $B$ embeds, such that
\begin{equation}
C\ra (B)^A_{k,t},
\end{equation}
where this means that for each coloring of the copies of $A$ in $C$ into $k$ colors,
there is a copy $B'$ of $B$ in $C$ such that all copies of $A$ in $B'$ take no more than $t$ colors.
Then $\mathcal{K}$ has  {\em finite (small) Ramsey degrees} if
for each $\A\in\mathcal{K}$ the number
 $t(\A,\mathcal{K})$  exists.
The number $t(\A,\mathcal{K})$ is called the {\em Ramsey degree of $A$} in $\mathcal{K}$ (\cite{Fouche98}).
Note that $\mathcal{K}$ has the Ramsey property if and only if $t(A,\mathcal{K})=1$ for each $A\in\mathcal{K}$.
A strong connection between \Fraisse\ classes with finite Ramsey degrees and ordered expansions is made explicit in Section 10 of \cite{Kechris/Pestov/Todorcevic05},
where it is shown that if an ordered expansion $\mathcal{K}^{<}$ of a \Fraisse\ class $\mathcal{K}$ has the Ramsey property,
then $\mathcal{K}$ has finite small Ramsey degrees, and the degree of $\A\in\mathcal{K}$  can be computed from the number of non-isomorphic order expansions it has in $\mathcal{K}^{<}$.
A similar result holds for pre-compact expansions (see \cite{NVTHabil}).
It follows from the results stated above  that the classes of finite graphs and finite graphs omitting $n$-cliques have finite small Ramsey degrees.

At this point, it is pertinent to  mention recent
advances connecting Ramsey theory with topological dynamics.
A new connection
was established in \cite{Kechris/Pestov/Todorcevic05} which accounts for previously known phenomena regarding universal minimal flows.
In that paper,
Kechris, Pestov, and Todorcevic  proved several strong correspondences between Ramsey theory and topological dynamics.
A \Fraisse\ order class is a \Fraisse\  class which
has at least one relation which is a linear order.
One of their main theorems  (Theorem 4.7)
shows that  the extremely amenable (fixed point property on compacta) closed subgroups of the infinite symmetric group $S_{\infty}$ are exactly those of the form Aut$(\mathbf{F}^*)$, where $\mathbf{F}^*$
is the \Fraisse\ limit of some \Fraisse\ order class satisfying the Ramsey property.
Another main theorem (Theorem 10.8)
provides a way to compute the universal minimal flow of
topological groups which arise as the automorphism groups of \Fraisse\ limits of \Fraisse\ classes with the Ramsey property and the ordering property.
That the ordering property can be relaxed to the expansion property was proved by Nguyen Van Th\'{e} in \cite{NVT13}.

We now turn to Ramsey theory on infinite structures.
One may  ask whether analogues of Theorem \ref{thm.RamseyInfinite}
 can hold on more complex
 infinite relational structures,
 in particular, for  \Fraisse\ limits of
\Fraisse\ classes.
The
 \Fraisse\ limit
 $\mathbf{F}$ of  a \Fraisse\ class $\mathcal{K}$  of finite relational structures is said to have  {\em finite big Ramsey degrees} if for each  member $\A$  in $\mathcal{K}$,
there is a finite number $T(\A,\mathcal{K})$ such that for any coloring $c$ of all the substructures of $\mathbf{F}$ which are isomorphic to $\A$ into finitely many colors,
there is a substructure $\mathbf{F}'$ of $\mathbf{F}$ which is isomorphic to $\mathbf{F}$ and in which $c$ takes  no more than $T(\A,\mathcal{K})$ colors.
When this is the case, we write
\begin{equation}
\mathbf{F}\ra(\mathbf{F})^{\A}_{k,T(\A,\mathcal{K})}.
\end{equation}
This notion has been around  for several decades,
but the terminology was initiated in
 \cite{Kechris/Pestov/Todorcevic05}.

The first homogeneous structure shown to have finite big Ramsey degrees is the rationals, which are the \Fraisse\ limit of the class of finite linear orders $\mathcal{LO}$.
That the upper bounds exist was known by Laver,
 following from  applications of Milliken's Theorem (see Theorem \ref{thm.Milliken}).
The lower bounds were proved by Devlin  in 1979 in his thesis \cite{DevlinThesis}, where he showed that the numbers $T(k,\mathcal{LO})$ are actually tangent numbers, coefficients of the Talyor series expansion of the tangent function.
In particular, $T(1,\mathbb{Q})=1$, as any coloring of the rationals into finitely many colors contains a copy of the rationals in one color;
thus, the rationals are {\em indivisible}.
On the other hand,
$T(2,\mathbb{Q})=2$, so immediately for colorings of pairsets of rationals, one sees that there is no Ramsey property for the rationals when one requires that the substructure $\mathbf{Q}'$ of $\mathbb{Q}$ be ``big'', meaning isomorphic to the original infinite one.

The next homogeneous structure for which big Ramsey degrees have been proved is the
the Rado graph, denoted
$\mathcal{R}$.
Also known as the random graph, $\mathcal{R}$ is the countable graph which
is universal for all countable graphs, meaning each countable graph embeds into $\mathcal{R}$ as an induced substructure.
Equivalently, the Rado graph is the \Fraisse\ limit of the  class of  finite graphs, denoted $\mathcal{G}$.
It is an easy exercise from the defining property of the Rado graph to show that  the Rado graph  is  indivisible, meaning that  the big Ramsey degree  for vertices in the Rado graph is $1$.
The first non-trivial lower bound result for big Ramsey degrees was proved by
\Erdos, Hajnal and P\'{o}sa in \cite{Erdos/Hajnal/Posa75} in 1975, where they
showed
there is a coloring of the edges in $\mathcal{R}$ into two colors such that for any subgraph $\mathcal{R}'$ of the Rado graph such that $\mathcal{R}'$ is also universal for countable graphs,
the  edges in $\mathcal{R}'$ take on both colors.
That this  upper bound is sharp was proved over two decades later in 1996 by Pouzet and Sauer
in \cite{Pouzet/Sauer96},
and thus,
 the big Ramsey degree for edges  in the Rado graph is $2$.
The problem
of whether every finite graph has a finite big Ramsey degree in the Rado graph
took another decade to solve.
In \cite{Sauer06},
Sauer
proved that the Rado graph, and in fact a  general class  of binary relational homogeneous structures,
have finite big Ramsey degrees.
As in Laver's result, Milliken's Theorem plays a  central role in  obtaining the upper bounds.
The sharp lower bounds were proved the same year  by Laflamme, Sauer, and Vuksanovic in \cite{Laflamme/Sauer/Vuksanovic06}.

Sauer's result on the Rado graph  in conjunction with the
attention called to big Ramsey degrees  in \cite{Kechris/Pestov/Todorcevic05}
sparked new interest in the field.
In 2008,
Nguyen Van Th\'{e}
investigated big Ramsey degrees for homogeneous ultrametric spaces.
Given $S$ a set of positive real numbers,
$\mathcal{U}_S$ denotes the class of all finite ultrametric spaces  with strictly positive distances in $S$.
Its \Fraisse\ limit, denoted
$\mathbf{Q}_S$, is called the {\em Urysohn space associated with} $\mathcal{U}_S$ and is a homogeneous ultrametric space.
In  \cite{NVT08},
Nguyen Van Th\'{e} proved  that
$\mathbf{Q}_S$ has finite big Ramsey degrees whenever $S$ is finite.
Moreover, if $S$ is infinite, then any member of $\mathcal{U}_S$ of size greater than or equal to $2$ does not have a big Ramsey degree.
Soon after,
 Laflamme, Nguyen Van Th\'{e}, and Sauer proved
in
\cite{Laflamme/NVT/Sauer10}
that enriched structures of the rationals, and two related directed graphs, have  finite big Ramsey degrees.
For each $n\ge 1$,
 $\mathbb{Q}_n$ denotes the structure $(\mathbb{Q},  Q_1,\dots, Q_n,<)$, where  $Q_1,\dots, Q_n$ are disjoint dense subsets of $\mathbb{Q}$  whose union is $\mathbb{Q}$.
This is the \Fraisse\ limit of the class $\mathcal{P}_n$ of all finite linear orders equipped with an equivalence relation with  $n$ many equivalence classes.
Laflamme, Nguyen Van Th\'{e}, and Sauer proved  that
each member of $\mathcal{P}_n$ has a finite big Ramsey degree in $\mathbb{Q}_n$.
Further,
using the bi-definability between $\mathbb{Q}_n$ and the circular directed graphs $\mathbf{S}(n)$, for $n=2,3$,
they proved that
 $\mathbf{S}(2)$ and  $\mathbf{S}(3)$
have finite big Ramsey degrees.
Central to these results is a colored verision of  Milliken's theorem which they  proved in order to deduce the big Ramsey degrees.
For a more detailed overview of these results, the reader is referred to
\cite{NVTHabil}.

A common theme emerges when one looks at the proofs in \cite{DevlinThesis}, \cite{Sauer06}, and
\cite{Laflamme/NVT/Sauer10}.
The first two rely  in an essential way  on Milliken's Theorem,
Theorem \ref{thm.Milliken}  in Section \ref{sec.2}.
The third proves a new colored version of Milliken's Theorem and uses it to deduce the results.
The results in \cite{NVT08} use Ramsey's theorem.
This would lead one to conclude or at least conjecture that, aside from Ramsey's Theorem itself,  Milliken's Theorem contains the core combinatorial content of  big Ramsey degree results.
The lack of such a result applicable to  homogeneous structures omitting non-trivial substructures
posed  the main obstacle to the investigation of  their big Ramsey degrees.
This is addressed in the present paper.

This article is concerned with the question of big Ramsey degrees for the  homogeneous countable  triangle-free graph, denoted $\mathcal{H}_3$.
A graph $\G$ is {\em triangle-free} if for any three vertices  in $\G$, there is at least one pair with no edge between them;
in other words, no triangle embeds into $\G$ as an induced subgraph.
A triangle-free graph $\mathcal{H}$ on countably many vertices is a {\em homogeneous}
if each isomorphism between two finite (triangle-free) subgraphs can be extended to an automorphism of $\mathcal{H}$.
It is {\em universal} if every triangle-free graph on countably many vertices embeds into it. 
Universal homogeneous triangle-free graphs were first constructed by Henson in  \cite{Henson71}.
Such graphs are also seen to be
the \Fraisse\ limit of $\mathcal{K}_3$, the \Fraisse\ class of all countable triangle-free graphs,
 and
any two universal  homogeneous triangle-free graphs are isomorphic.

As mentioned above,
 \Nesetril\ and \Rodl\
proved that  the \Fraisse\ class of finite ordered triangle-free graphs, denoted $\mathcal{K}^{<}_3$, has the Ramsey property.
It follows that the \Fraisse\ class of  unordered finite triangle-free graphs, denoted
$\mathcal{K}_3$, has finite small Ramsey degrees.
In contrast, whether or not
every finite triangle-free graph has  a
finite  big Ramsey degree in  $\mathcal{H}_3$
had been open until now.
The first result on colorings of vertices of $\mathcal{H}_3$
 was obtained by Henson in \cite{Henson71} in 1971.
In that paper,
 he proved that  $\mathcal{H}_3$ is weakly indivisible:
Given  any coloring of the vertices  of $\mathcal{H}_3$ into  two colors,
either
there is a copy  of $\mathcal{H}_3$ in which all vertices have the first  color,
or else a copy of each member of $\mathcal{K}_3$ can be found with all vertices having the second color.
From this follows a prior result of Folkman in \cite{Folkman70}, that  for any finite triangle-free graph $\G$ and any number $k\ge 2$,
there is a finite triangle-free  graph $\HH$ such that for any partition of the vertices of $\mathrm{H}$ into $k$ pieces,
there is a copy of $\G$ in having all its vertices in one of the pieces of the partition.
In 1986,
Komj\'{a}th and \Rodl\  proved
that $\mathcal{H}_3$ is indivisible;
thus, the  big Ramsey degree for vertex colorings is $1$.
It then became of interest whether this result would extend to colorings of copies of a fixed finite triangle-free graph, rather than just colorings of vertices.

In 1998, Sauer proved in \cite{Sauer98} that edges  have  finite big Ramsey degree of $2$ in $\mathcal{H}_3$, leaving
 open the general question:

\begin{question}\label{q.fRd}
Does every finite triangle-free graph have finite big Ramsey degree in $\mathcal{H}_3$?
\end{question}

This paper answers this question in the affirmative.

Ideas from  Sauer's proof   in \cite{Sauer06}
that the Rado
 graph has finite big Ramsey degrees
provided  a strategy for our proof in this paper.
A rough outline of Sauer's proof is as follows:
Graphs can be coded by nodes on trees.
Given such codings, the graph coded by the nodes in
the tree consisting of all finite length sequences of $0$'s and $1$'s, denoted as
 $2^{<\om}$,
 is bi-embeddable with the Rado graph.
Only certain subsets, called strongly diagonal,  need to be considered when handling tree codings of a given finite graph $\G$.
Any finite strongly diagonal set can be enveloped into a strong tree, which is a tree isomorphic to $2^{\le k}$ for some $k$.
The coloring on the copies of $\G$ can be extended to color the strong tree envelopes.
Applying Milliken's Theorem  for strong trees  finitely many times, one obtains an infinite  strong subtree $S$  of $2^{<\om}$ in which for
all diagonal sets coding $\G$ with the same strong similarity type have the same color.
To finish,  take a strongly diagonal $\D$ subset of $S$ which codes the Rado graph, so that
all codings of $\G$ in $\D$ must be strongly diagonal.
Since there are only finitely many similarity types of strongly diagonal sets coding $\G$, this yields the finite big Ramsey degrees for the Rado graph.
See Section \ref{sec.2} for more details.

This outline seemed to the author the most likely to succeed if
indeed the universal triangle-free graph were to have finite  big Ramsey degrees.
However,
there were  difficulties involved  in each step  of trying to adapt Sauer's proof
 to the setting of $\mathcal{H}_3$,
largely because $\mathcal{H}_3$  omits a substructure, namely triangles.
First,
unlike the
 bi-embeddability between the Rado graph and the graph coded by the nodes in  $2^{<\om}$,
there is no bi-embeddability relationship between $\mathcal{H}_3$ and some
triangle-free graph coded by some
 tree with a  very regular  structure.
To handle this,
rather than letting certain nodes in a tree code vertices at the very end of the whole proof scheme as Sauer does in \cite{Sauer06},
we
 introduce a new notion of {\em strong triangle-free tree}
in which  we
 distinguish certain  nodes {\em in} the tree   (called {\em coding nodes})
to code the vertices of a given graph, and in which the branching is maximal subject to the constraint of  these distinguished nodes not coding any triangles.
We further develop a flexible construction method for creating   strong triangle-free trees in which the distinguished nodes code $\mathcal{H}_3$.
These are found in Section \ref{sec.3}.

Next,
we wanted  an analogue of Milliken's Theorem for strong triangle-free trees.
While  we were able to prove
such a theorem
 for any  configuration extending some fixed stem,
the result simply does not hold  for colorings of stems, as  can be
seen by an example of a bad coloring defined using interference between splitting nodes and coding nodes on the same level  (Example \ref{ex.bc}).
The means around this  was to introduce the new notion of {\em  strong coding tree}, which  is a skew tree that stretches  a strong triangle-free tree while preserving all important aspects of its coding structure.
Strong coding trees are defined and constructed in Section \ref{sec.4}.
There, the fundamentals of the collection of strong coding trees are charted, including sufficient conditions
guaranteeing when a finite subtree $A$  of a strong coding tree $T$ may be end-extended into $T$ to form another strong coding tree.

Having formulated the correct kind of trees to code $\mathcal{H}_3$, the next task is to prove an analogue of Milliken's Theorem for strong coding trees.
This is accomplished in Sections \ref{sec.5} and  \ref{sec.1SPOC}.
First, we prove analogues of
 the Halpern-\Lauchli\ Theorem  (Theorem \ref{thm.HL})
 for strong coding trees.
There are two cases, depending on whether the level sets being colored contain a splitting node or a coding node.
In  Case (a) of
Theorem \ref{thm.matrixHL}, we obtain
the direct analogue of the Halpern-\Lauchli\ Theorem when the level set being colored has a splitting node.
A similar result is proved in Case (b) of Theorem \ref{thm.matrixHL} for level sets containing a coding node, but some restrictions apply, and these are   taken care of in Section \ref{sec.1SPOC}.
The proof of Theorem \ref{thm.matrixHL}
 Section  \ref{sec.5} uses the set-theoretic method of forcing, using some forcing posets created specifically for strong coding trees.
However, one never moves into a generic extension; rather the forcing mechanism is used to do an unbounded search for a finite object.
Once found, it is used to build the next finite level of the tree homogeneous for a given coloring.
Thus, the result is a ZFC proof.
This builds on
 ideas from Harrington's forcing proof of the Halpern-\Lauchli\ Theorem.

In Section \ref{sec.1SPOC},
after an initial lemma obtaining end-homogeneity,
we
achieve the analogue of the Halpern-\Lauchli\ Theorem
for Case (b) in Lemma \ref{lem.Case(c)}.
The proof
introduces a third forcing
which homogenizes over the possibly different end-homogeneous colorings, but again achieves a ZFC result.
Then, using much induction and fusion,
we obtain the first of our two Milliken-style theorems.
\begin{thmR1}
Let $T$ be a strong coding tree and let $A$ be a finite subtree of $T$ satisfying the \STROC.
Then for any coloring of all strictly similar copies of $A$ in $T$ into finitely many colors,
there is a strong coding tree $S\le T$
such that all strictly similar copies of $A$ in $S$  have the same color.
\end{thmR1}
The \STROC\ is  made clear in Definition \ref{defn.strPOC}.
Initial segments of strong coding trees automatically satisfy the \STROC.
Essentially, it is a strong condition which guarantees that the finite subtree can be extended to a tree coding $\mathcal{H}_3$.

Developing the correct notion of strong subtree envelope for the setting of triangle-free graphs  presented a further  obstacle.
The idea of extending a subset $X$ of a strong coding tree $T$  to an envelope which is a  finite strong triangle-free  tree  and applying Theorem \ref{thm.MillikenIPOC}
(which would be the direct analogue of Sauer's method)
simply does not work, as it can lead to an infinite regression of adding coding nodes in order to
 make an envelope of that form.
That is, there is no upper bound on the number of
 similarity  types  of finite strong triangle-free  subtrees of $T$
which are minimal containing
copies of $X$ in $T$.
To overcome this,
in Sections
\ref{sec.squiggle}
   and
 \ref{sec.1color}
we develop  the   notions of  incremental new parallel $1$'s and  strict similarity type for finite diagonal sets of coding nodes
as well as a new notion of  envelope.
Given any   finite triangle-free graph $\G$, there are only finitely many strict similarity types of diagonal trees coding $\G$.
Letting
$c$ be any coloring of all copies of $\G$ in $\mathcal{H}_3$ into finitely many colors,
we
transfer the coloring to the envelopes and
 apply the results in previous sections to obtain a strong coding tree $T'\le T$ in which all  envelopes  encompassing the same strict similarity type  have the same color.
The next new idea is to thin $T'$
to an incremental strong subtree $S\le T'$ while simultaneously choosing a set $W\sse T'$ of {\em witnessing coding nodes}.
These have the property that  each  finite
subset $X$ of $S$  is incremental, and furthermore, one can add to $X$ coding nodes from $W$ to form an envelop satisfying the \STROC.
Then we arrive at our second Milliken-style theorem for strong coding trees, extending the first one.

\begin{thmR2}
[Ramsey Theorem for Strict Similarity Types]
Let $Z$ be a finite antichain of coding nodes in a strong coding tree $T$,
and let $h$ be a coloring of all subsets of $T$ which are strictly similar to $Z$ into finitely many colors.
Then  there is an incremental strong coding tree $S\le T$ such that all  subsets of $S$
strictly similar to  $Z$  have the same $h$ color.
\end{thmR2}

After thinning to a strongly diagonal subset $D\sse S$  still coding $\mathcal{H}_3$,
the only sets of coding nodes in $D$ coding a given finite triangle-free graph $\G$ are automatically antichains which are incremental and strongly diagonal.
Applying  Theorem \ref{thm.mainRamsey} to the finitely many strict similarity types of incremental strongly diagonal sets  coding $\G$, we arrive at the main theorem.
\begin{thmfinalthm}
The universal triangle-free  homogeneous graph has finite big Ramsey degrees.
\end{thmfinalthm}

For each $\G\in\mathcal{K}_3$,
the number $T(\G,\mathcal{K}_3)$
is bounded by  the number   of strict similarity types of diagonal sets of coding nodes coding $\G$, which we denote as StrSim$(\G,\bT)$, $\bT$ referring to any strong coding tree (see Section \ref{sec.4}).
It is  presently open to see if StrSim$(\G,\bT)$ is in fact the lower bound.
If  it is, then recent work of Zucker  would provide an interesting connection with topological dynamics.
In \cite{Zucker19},
Zucker proved that
if  a  \Fraisse\ structure  $\mathbf{F}$ has finite big Ramsey degrees and moreover, $\mathbf{F}$  admits a big Ramsey structure,
then
any big Ramsey flow of Aut$(\mathbf{F})$ is a universal completion flow, and further,  any two universal completion flows are isomorphic.
His proof of  existence  of a big Ramsey structure  a \Fraisse\ structure presently
relies on the existence of  colorings for an increasing sequence of finite objects whose union is $\mathbf{F}$ exhibiting all color classes which cannot be removed and which cohere in a  natural  way.
In particular, the lower bounds for the big Ramsey numbers are necessary to Zucker's analysis.
His work already applies to the rationals,  the Rado graph,  lower bounds being obtained by Laflamme, Sauer, and Vuksanovic in \cite{Laflamme/Sauer/Vuksanovic06} and calculated for each class of graphs of fixed finite size by Larson in \cite{Larson08},
finite ultrametric spaces with distances  from a fixed finite set,
 $\bQ_n$ for each $n\ge 2$, $\mathbf{S}(2)$, and $\mathbf{S}(3)$.
As the strict similarity types found in this paper satisfy Zucker's  coherence condition,
the precise lower bounds for the big Ramsey degrees of $\mathcal{H}_3$ would provide another such example of a universal completion flow.

\begin{ack}
Much gratitude goes to
 Dana Barto\v{s}ov\'a for listening to and making helpful comments on early and later stages of these results and for her continued encouragement of my work on this problem;
 Jean Larson for listening to early stages of this work and her encouragement;
Norbert Sauer for discussing key aspects of the homogeneous triangle-free graph with me during a research visit in Calgary in 2014;    Stevo Todorcevic for pointing out to me in 2012 that any proof of finite Ramsey degrees for $\mathcal{H}_3$ would likely involve a new type of Halpern-\Lauchli\ Theorem;
 and to the organizers and participants of the
Ramsey Theory Doccourse at Charles University, Prague, 2016, for their encouragement.
Most of all,  I am grateful for  and much indebted to
  Richard Laver  for
providing  for me in 2011 the main  points of
 Harrington's forcing proof of the Halpern-\Lauchli\ Theorem, from which I could reconstruct the proof,
setting the stage for the possibility of accomplishing this work.  His spirit lives on.
\end{ack}


\section{Background: Trees coding graphs and the Halpern-\Lauchli\ and  Milliken Theorems}\label{sec.2}

This section
 provides background and context for  the developments in this paper.
It contains the method of using trees to code graphs, the Halpern-\Lauchli\ and Milliken Theorems,
and a discussion of  their applications to
previously known results on
 big Ramsey degrees
for homogeneous structures.


\subsection{Trees coding graphs}\label{subsection.treescodinggraphs}

In \cite{Erdos/Hajnal/Posa75},
\Erdos, Hajnal and P\'{o}sa
gave the vertices in a graph a natural lexicographic order
 and used it to solve problems regarding strong embeddings of graphs.
The set of  vertices of a graph ordered by this lexicographic order can be
viewed as nodes in
the  binary tree of finite sequences of $0$'s and $1$'s
with
 the usual tree ordering.
This   was made explicit  in \cite{Sauer98} and is described  below.

The following notation is standard in mathematical logic and shall be used throughout.
The set of all natural numbers $\{0,1,2,\dots\}$ is denoted by $\om$.
Each natural number  $k\in\om$
is equated with  the set  of all natural numbers strictly less than $k$.
Thus, $0$ denotes the emptyset, $1=\{0\}$,  $2=\{0,1\}$, etc.
For each natural number $k$, $2^k$  denotes
 the set of all functions from $\{0,\dots, k-1\}$ into $\{0,1\}$.
A finite {\em binary sequence} is a function
 $s:k\ra 2$ for some $k\in \om$.
We may  write $s$ as $\lgl s(0), \dots, s(k-1)\rgl$;
 for each $i<k$, $s(i)$ denotes the $i$-th {\em value} or {\em entry}  of the sequence $s$.
We shall  use  $2^{<\om}$  to denote the collection $\bigcup_{k\in \om}2^k$
 of all finite binary sequences.
For $s\in 2^{<\om}$, we let $|s|$ denote the {\em length} of $s$; this is exactly the domain of $s$.
For nodes $s,t\in 2^{<\om}$, we write
$s\sse t$ if and only if $|s|\le |t|$ and for each $i<|s|$, $s(i)=t(i)$.
In this case, we say that  $s$ is an {\em initial segment} of $t$, or that $t$ {\em extends} $s$.
If $t$ extends $s$ and $|t|>|s|$, then we write $s\subset t$, and we say that $s$ is a proper initial segment of $t$.
For $i<\om$, we let $s\re i$ denote
the function $s$ restricted to domain $i$.
Thus, if $i< |s|$, then $s\re i$ is
 the proper initial segment of $s$ of length $i$, $s\re i=\lgl s(0),\dots, s(i-1)\rgl$;
if $i\ge |s|$, then $s\re i$  equals $s$.
The set $2^{<\om}$ forms a tree when partially ordered by inclusion.

Let $v,w$ be vertices in some graph.
Two nodes
 $s,t\in 2^{<\om}$  are said to   {\em represent} $v$ and $w$,  respectively, if and only if,  without loss of generality assuming that $|s|<|t|$,
then
$v$ and $w$ have an edge between them
if and only if
 $t(|s|)=1$.
The number $t(|s|)$ is called the {\em passing number} of $t$ at $s$.
Thus, if $t$ has  passing number  $1$  at $s$, then
$s$ and $t$ code an edge  between $v$ and $w$;
and if $t$ has passing number $0$ at $s$, then $s$ and $t$ code a non-edge between $v$ and $w$.

Using this idea, any graph can be coded by nodes in a binary tree as follows.
Let $\G$ be a  graph with $N$ vertices,
where $N\le \om$,
and
let  $\lgl v_n:n<N\rgl$ be any enumeration of
the vertices of  $\G$.
Choose any node  $t_0\in 2^{<\om}$ to  represent the vertex $v_0$.
For  $n>0$,
given nodes $t_0,\dots,t_{n-1}$  in $2^{<\om}$ coding the vertices $v_0,\dots,v_{n-1}$,
take  $t_n$  to be any node in $2^{<\om}$  such that
 $|t_n|>|t_{n-1}|$ and
 for all $i< n$,
$v_n$ and $ v_i$ have an edge between them if and only if $t_n(|t_i|)=1$.
Then the set of nodes $\{t_n:n<N\}$  codes the graph $\G$.
Note that any finite graph of size $k$ can be coded by a collection of  nodes in $\bigcup_{i<k}{}^i 2$.
Throughout this paper we shall hold to the convention that the nodes in a tree used to code a
 graph will have different lengths.
Figure 1.\  shows a  set of nodes $\{t_0,t_1,t_2,t_3\}$  from $2^{<\om}$  coding the four-cycle $\{v_0,v_1,v_2,v_3\}$.

\begin{figure}
\begin{tikzpicture}[grow'=up,scale=.5]
\tikzstyle{level 1}=[sibling distance=4in]
\tikzstyle{level 2}=[sibling distance=2in]
\tikzstyle{level 3}=[sibling distance=1in]
\tikzstyle{level 4}=[sibling distance=0.5in]
\node {$\left< \ \right >$}
 child{
	child{
		child{edge from parent[draw=none]
			child{edge from parent[draw=none]}
			child{edge from parent[draw=none]}}
		child{ coordinate (t2)
			child{edge from parent[draw=none]}
			child {edge from parent[draw=none]}}
		}
	child{ coordinate (t1)
		child{
			child{edge from parent[draw=none]}
			child{coordinate (t3)}}
		child {edge from parent[draw=none]} }}
 child{coordinate (t0)
	child{edge from parent[draw=none]
		child{edge from parent[draw=none]
			child{edge from parent[draw=none]}
			child{edge from parent[draw=none]}
		}
		child{edge from parent[draw=none]
			child{edge from parent[draw=none]}
			child {edge from parent[draw=none]}}
	}
	child {edge from parent[draw=none]}};
		
\node[right] at (t0) {$t_{0}$};
\node[right] at (t1) {$t_{1}$};
\node[right] at (t2) {$t_{2}$};
\node[right] at (t3) {$t_{3}$};

\node[circle, fill=black,inner sep=0pt, minimum size=7pt] at (t0) {};
\node[circle, fill=black,inner sep=0pt, minimum size=7pt] at (t1) {};
\node[circle, fill=black,inner sep=0pt, minimum size=7pt] at (t2) {};
\node[circle, fill=black,inner sep=0pt, minimum size=7pt] at (t3) {};

\draw[thick, dotted] let \p1=(t1) in (-12,\y1) node (v1) {$\bullet$} -- (7,\y1);
\draw[thick, dotted] let \p1=(t2) in (-12,\y1) node (v2) {$\bullet$} -- (7,\y1);
\draw[thick, dotted] let \p1=(t3) in (-12,\y1) node (v3) {$\bullet$} -- (7,\y1);
\draw[thick, dotted] let \p1=(t0) in (-12,\y1) node (v0) {$\bullet$} -- (7,\y1);

\node[left] at (v1) {$v_1$};
\node[left] at (v2) {$v_2$};
\node[left] at (v3) {$v_3$};
\node[left] at (v0) {$v_0$};

\draw[thick] (v0.center) to (v1.center) to (v2.center) to (v3.center) to [bend left] (v0.center);

\end{tikzpicture}
\caption{A tree with nodes $\{t_0,t_1,t_2,t_3\}$ coding the 4-cycle $\{v_0,v_1,v_2,v_3\}$}
\end{figure}


\subsection{The Halpern-\Lauchli\ and Milliken Theorems}\label{subsection.HLM}

The theorem of Halpern and \Lauchli\ below
was  established
as
 a technical lemma  containing core combinatorial content  of the
proof that
 the Boolean Prime Ideal Theorem, the statement that any filter can be extended to an ultrafilter, is strictly weaker than the Axiom of Choice, assuming  the Zermelo-Fraenkel axioms of set theory.
(See  \cite{Halpern/Levy71}.)
The Halpern-\Lauchli\ Theorem forms the basis for a Ramsey theorem on  strong trees due to Milliken, which in turn
 forms the backbone of all previously found
finite big Ramsey degrees, except where Ramsey's Theorem itself suffices.
An in-depth presentation of the various versions of the Halpern-\Lauchli\ Theorem as well as Milliken's Theorem can be found in  \cite{TodorcevicBK10}.
An account focused solely on
the aspects  relevant to  the present work can be found in \cite{DobrinenRIMS17}.
Here, we merely give an overview sufficient for this article, and shall restrict to subtrees of $2^{<\om}$, though the results hold more generally for  finitely branching trees.

In  this paper,
we use the  definition of tree which is standard for  Ramsey theory on trees.
The meet of two nodes $s$ and $t$ in $2^{<\om}$, denoted $s\wedge t$,
is the longest member $u\in 2^{<\om}$ which is an initial segment of both $s$ and $t$.
Thus, $u=s\wedge t$ if and only if $u=s\re|u|=t\re|u|$ and $s \re (|u|+1)\ne t\re (|u|+1)$.
In particular, if $s\sse t$ then $s\wedge t=s$.
A set  of nodes $A\sse 2^{<\om}$ is {\em closed under meets}
if $s\wedge t$ is in $A$, for each pair $s,t\in A$.

\begin{defn}\label{defn.tree}
A  subset $T\sse 2^{<\om}$  is a {\em tree}
if $T$
 is closed under meets and  for each pair $s,t\in T$
with $|s|\le |t|$,
$t\re |s|$ is also in $T$.
\end{defn}

Given $n<\om$ and a set of nodes $A\sse 2^{<\om}$, define
 \begin{equation}
A(n)=\{t\in A:|t|=n\}.
\end{equation}
A set $X\sse A$ is a {\em level set}
 if $X\sse A(n)$ for some  $n<\om$.
Note that a  tree $T$ does not have to contain all   initial segments  of its members, but for each $s\in T$, the level set $T(|s|)$ must equal $\{t\re |s|:t\in T$ and $|t|\ge |s|\}$.

Let
 $T\sse 2^{<\om}$ be a tree and  let
$L=\{|s|:s\in T\}$.
If $L$ is infinite, then
$T$
is a
{\em strong tree} if every node in $T$ {\em splits} in $T$; that is, for each $t\in T$, there are $u,v\in T$ such that  $u$ and $v$ properly extend $t$,
and $u(|t|)=0$ and $v(|t|)=1$.
If $L$ is finite,
then $T$ is a {\em strong tree}
if for each node $t\in T$ with $|t|<\max(L)$,
$t$ splits in $T$.
A finite strong tree subtree of $2^{<\om}$ with $k$ many levels is called
a {\em strong tree of height $k$}.
Note that each finite strong subtree of $2^{<\om}$
is isomorphic as a tree to some binary tree of height $k$.
In particular, a strong tree  of height $1$ is simply a node in $2^{<\om}$.
See Figure 2. for an example of a strong tree of height $3$.


\begin{figure}
\begin{tikzpicture}[grow'=up,scale=.6]
\tikzstyle{level 1}=[sibling distance=4in]
\tikzstyle{level 2}=[sibling distance=2in]
\tikzstyle{level 3}=[sibling distance=1in]
\tikzstyle{level 4}=[sibling distance=0.5in]
\tikzstyle{level 5}=[sibling distance=0.2in]
\node {} coordinate (t9)
child{coordinate (t0) edge from parent[thick]
			child{coordinate (t00) edge from parent[thin]
child{coordinate (t000)
child {coordinate(t0000)
child{coordinate(t00000) edge from parent[color=black] }
child{coordinate(t00001)}}
child {coordinate(t0001) edge from parent[color=black]
child {coordinate(t00010)}
child{coordinate(t00011)}}}
child{ coordinate(t001)
child{ coordinate(t0010)
child{ coordinate(t00100)}
child{ coordinate(t00101) edge from parent[color=black] }}
child{ coordinate(t0011) edge from parent[color=black]
child{ coordinate(t00110)}
child{ coordinate(t00111)}}}}
			child{ coordinate(t01)  edge from parent[color=black]
child{ coordinate(t010)
child{ coordinate(t0100) edge from parent[thin]
child{ coordinate(t01000)}
child{ coordinate(t01001)}}
child{ coordinate(t0101)
child{ coordinate(t01010) edge from parent[thin]}
child{ coordinate(t01011)}}}
child{ coordinate(t011)
child{ coordinate(t0110)
child{ coordinate(t01100) edge from parent[thin]}
child{ coordinate(t01101)}}
child{ coordinate(t0111) edge from parent[thin]
child { coordinate(t01110)}
child{ coordinate(t01111)}}}}}
		child{ coordinate(t1)  edge from parent[thick]
			child{ coordinate(t10)
child{ coordinate(t100)
child{ coordinate(t1000) edge from parent[color=black, thin]
child{ coordinate(t10000)}
child{ coordinate(t10001)}}
child{ coordinate(t1001)
child{ coordinate(t10010)}
child{ coordinate(t10011) edge from parent[color=black, thin] }}}
child{ coordinate(t101)
child{ coordinate(t1010) edge from parent[color=black, thin]
child{ coordinate(t10100) }
child{ coordinate(t10101)}}
child{ coordinate(t1011)
child{ coordinate(t10110) edge from parent[color=black, thin] }
child{ coordinate(t10111)}}}}
			child{  coordinate(t11)  edge from parent[color=black, thin]
child{ coordinate(t110)
child{ coordinate(t1100)
child{ coordinate(t11000)}
child{ coordinate(t11001)}}
child{ coordinate(t1101)
child{ coordinate(t11010)}
child{ coordinate(t11011)}}}
child{  coordinate(t111)
child{  coordinate(t1110)
child{  coordinate(t11100)}
child{  coordinate(t11101)}}
child{  coordinate(t1111)
child{  coordinate(t11110)}
child{  coordinate(t11111)}}}} };

\node[left] at (t0) {$0$};
\node[left] at (t00) {$00$};
\node[left] at (t000) {$000$};
\node[left] at (t001) {$001$};
\node[left] at (t01) {$01$};
\node[left] at (t010) {$010$};
\node[left] at (t011) {$011$};
\node[right] at (t1) {$1$};
\node[right] at (t10) {$10$};
\node[right] at (t100) {$100$};
\node[right] at (t101) {$101$};
\node[right] at (t11) {$11$};
\node[right] at (t110) {$110$};
\node[right] at (t111) {$111$};

\node[circle, fill=black,inner sep=0pt, minimum size=6pt] at (t9) {};

\node[circle, fill=black,inner sep=0pt, minimum size=6pt] at (t01) {};
\node[circle, fill=black,inner sep=0pt, minimum size=6pt] at (t01011) {};
\node[circle, fill=black,inner sep=0pt, minimum size=6pt] at (t01101) {};
\node[circle, fill=black,inner sep=0pt, minimum size=6pt] at (t10010) {};
\node[circle, fill=black,inner sep=0pt, minimum size=6pt] at (t10) {};
\node[circle, fill=black,inner sep=0pt, minimum size=6pt] at (t10111) {};

\end{tikzpicture}
\caption{A strong subtree of $2^{<\om}$ of height $3$}
\end{figure}


The following is  the strong tree version of the Halpern-\Lauchli\ Theorem.
It is a Ramsey theorem for colorings of products of level sets  of finitely many trees.
Here, we  restrict to the case of binary trees, since that is sufficient for the exposition in this paper.

\begin{thm}[Halpern-\Lauchli, \cite{Halpern/Lauchli66}]\label{thm.HL}
Let  $T_i=2^{<\om}$ for each $i<d$, where $d$ is any positive integer, and
let
\begin{equation}
c:\bigcup_{n<\om}\prod_{i<d} T_i(n)\ra k
\end{equation}
 be a given coloring, where $k$ is any positive integer.
Then there is an infinite set of levels $L\sse \om$ and infinite  strong subtrees $S_i\sse T_i$,  each with  nodes exactly at the levels in $L$,
such that $c$ is monochromatic on
\begin{equation}
\bigcup_{n\in L}\prod_{i<d} S_i(n).
\end{equation}
\end{thm}

This theorem of Halpern and \Lauchli\ was applied by Laver in
 \cite{Laver84}
to prove that
given $k\ge 2$ and given
any coloring of the product of $k$ many copies of the rationals  $\mathbb{Q}^k$
into finitely many colors,
there are subsets $X_i$ of the rationals which again are  dense linear orders without endpoints such that
$X_0\times\dots\times X_{k-1}$ has at most $k!$ colors.
Laver further proved that  $k!$ is  the lower bound.
Thus, the big Ramsey degree for the simplest object (single $k$-length sequences) in the \Fraisse\ class of products of finite linear orders
has been found.
The full  result for all big Ramsey degrees for  Age($\mathbb{Q}^k)$ would involve applications of  the extension of  Milliken's  theorem to  products of finitely many copies of $2^{<\om}$;
such an extension
 has been proved by Vlitas in \cite{Vlitas14}.

Harrington   produced   an interesting method for proving the Halpern-\Lauchli\ Theorem which uses the set-theoretic technique of forcing, but which takes place entirely  in the standard axioms of set theory, and most of mathematics, ZFC.
No  new external  model is actually built, but rather, finite  bits of information,  guaranteed by the existence of a  generic filter for the forcing, are
used to build the subtrees satisfying the Halpern-\Lauchli\ Theorem.
This proof is said to provide the clearest  intuition into the theorem (see \cite{TodorcevicBK10}).
Harrington did not publish his proof,  though the ideas were well-known
in certain circles.
A version  close to Harrington's original proof appeared in  \cite{DobrinenRIMS17}, where a   proof  was reconstructed based on an outline provided to the author by Laver in 2011.
This  proof
 formed the starting point for our proofs in
 Sections \ref{sec.5}  and
\ref{sec.1SPOC} of Halpern-\Lauchli\ style theorems for strong coding trees.
An earlier proof appeared in \cite{Farah/TodorcevicBK}.
That proof uses the  weaker assumption $\kappa\ra (\aleph_0)^d_2$ instead of Harrington's original $\kappa\ra (\aleph_1)^{2d}_{\aleph_0}$ (see Definition \ref{defn.arrownotation}), necessitating more involved arguments.

Harrington's  proof for $d$ many trees uses the forcing which adds $\kappa$ many Cohen subsets of the product of level sets of $d$ many copies of $2^{<\om}$,
where $\kappa$ satisfies a certain partition relation, depending on $d$.
For any set $X$ and cardinal $\mu$, $[X]^{\mu}$ denotes the collection of all subsets of $X$ of cardinality $\mu$.

\begin{defn}\label{defn.arrownotation}
Given cardinals $r,\sigma,\kappa,\lambda$,
\begin{equation}
\lambda\ra(\kappa)^r_{\sigma}
\end{equation}
means that for each coloring of $[\lambda]^r$ into $\sigma$ many colors,
there is a subset $X$ of $\lambda$ such that $|X|=\kappa$ and all members of $[X]^r$ have the same color.
\end{defn}

The following  ZFC result guarantees cardinals large enough to have the Ramsey property for colorings into infinitely many colors.

\begin{thm}[\Erdos-Rado, \cite{Erdos/Rado56}]\label{thm.ER}
For $r<\om$ and $\mu$ an infinite cardinal,
$$
\beth_r(\mu)^+\ra(\mu^+)_{\mu}^{r+1}.
$$
\end{thm}

For $d$ many trees,  letting  $\kappa=\beth_{2d-1}(\aleph_0)^+$ suffices for Harrington's proof.
A  modified version of Harrington's proof appears in
\cite{Farah/TodorcevicBK}, where the assumption on $\kappa$ is weaker, only $\beth_{d-1}(\aleph_0)^+$,  but the construction is more complex.
This proof informed the approach in \cite{Dobrinen/Hathaway16}
to reduce the large cardinal assumption for obtaining the consistency of the Halpern-\Lauchli\ Theorem at  a measurable cardinal.
Building on this and  ideas from \cite{Shelah91} and \cite{Dzamonja/Larson/MitchellRado09}, Zhang proved the consistency of Laver's result for the $\kappa$-rationals, for $\kappa$ measurable, in \cite{Zhang17}.

The Halpern-\Lauchli\ Theorem forms the essence of the next Theorem; the proof   follows  by a several step  induction
 applying Theorem \ref{thm.HL}.

\begin{thm}[Milliken, \cite{Milliken79}]\label{thm.Milliken}
Let $k\ge 1$ be given and let  all  strong subtrees of $2^{<\om}$  of height $k$ be colored by finitely many colors.
Then there is an infinite strong subtree $T$ of $2^{<\om}$ such that all strong subtrees of $T$ of height $k$ have the same color.
\end{thm}

In the Introduction,
an outline of Sauer's proof that the Rado graph has finite big Ramsey degrees was presented.
Knowledge of his proof is not a pre-requisite for reading this paper,
but the reader with knowledge of that paper will have better context for and understanding of  the present article.
A more detailed outline of the work in \cite{Sauer06} appears in Section 3 of \cite{DobrinenRIMS17}, which surveys some recent work regarding Halpern-\Lauchli\ and Milliken Theorems and variants.
Chapter 6 of \cite{TodorcevicBK10} provides a solid  foundation  for understanding how  Milliken's theorem is used to attain big Ramsey degrees for both Devlin's result on the rationals and Sauer's result on the Rado graph.
Of course,  we recommend foremost Sauer's original article \cite{Sauer06}.

We point out that
 Milliken's Theorem
 has been shown to consistently hold at a measurable cardinal by Shelah in \cite{Shelah91}, using ideas from Harrington's proof.
An enriched version was proved by D\v{z}amonja, Larson, and Mitchell in  \cite{Dzamonja/Larson/MitchellRado09} and applied  to obtain the consistency of finite big Ramsey degrees for colorings of finite subsets of the $\kappa$-rationals,
 where $\kappa$ is a measurable cardinal.
They
 obtained the consistency of finite big Ramsey degrees for colorings of finite subgraphs of the $\kappa$-Rado graph for $\kappa$ measurable
in \cite{Dzamonja/Larson/MitchellRado09}.
The uncountable height of the tree $2^{<\kappa}$ coding the $\kappa$-rationals and the $\kappa$-Rado graph   renders
 the notion of strong similarity type  more complex than for the countable cases.

There is another theorem  stronger than Theorem \ref{thm.Milliken},  also due to Milliken in \cite{Milliken81},
which shows that the collection of all infinite strong subtrees of $2^{<\om}$ forms a topological Ramsey space, meaning that it satisfies an infinite-dimensional Ramsey theorem for Baire sets when equipped with its version of the  Ellentuck  topology (see \cite{TodorcevicBK10}).
Though not outrightly used, this  fact informed some of our intuition when  approaching  the present work.


\section{Strong triangle-free trees coding $\mathcal{H}_3$}\label{sec.3}

In the previous section, it was shown how nodes in binary trees may be used to code graphs, and strong trees and Milliken's Theorem were presented.
In this section, we introduce {\em strong triangle-free trees}, which
 seem to be the correct analogue of  Milliken's  strong trees suitable for  coding triangle-free graphs.

Sauer's  proof in \cite{Sauer06} that the Rado graph has finite big Ramsey degrees
uses the fact that the Rado graph is bi-embeddable with the graph coded by the collection of  all nodes in $2^{<\om}$, where  nodes with the same length code vertices with no edges between them.
Colorings on the Rado graph are transfered to the graph represented by  the nodes in $2^{<\om}$,  Milliken's Theorem for strong trees is applied, and then the homogeneity is transfered back to the Rado graph.
In the case of the universal triangle-free graph, there is no known   bi-embeddability between $\mathcal{H}_3$ and some triangle-free graph coded by nodes in a tree with some kind of uniform structure.
Indeed, this may be fundamentally impossible  precisely because the absence of  triangles disrupts any uniformity of a coding structure.
Thus, instead of looking for a uniform sort of  structure
which codes some triangle-free graph bi-embeddable with $\mathcal{H}_3$
and trying to prove a Milliken-style theorem for them,
 we define a new kind of tree in which certain nodes are distinguished to
code the vertices of a given triangle-free graph.
Moreover,
 nodes in the tree
   branch as much as possible, subject to the constraint that at each level of the tree, no  node
is
extendible to another distinguished node  which  would code a triangle with previous distinguished  nodes.
 The precise formulation of  strong triangle-free tree appears in  Definition \ref{defn.stft}.

Some conventions and notation are now set up.
Given a triangle-free graph $\G$, finite or infinite,
let  $\lgl v_n:n<N\rgl$ be any enumeration of
the vertices of $\G$, where $N\le \om$ is the number of  vertices in $\G$.
We may construct a tree $T$ with  certain  nodes $\lgl c_n:n<N\rgl$  in $T$ coding the graph $\G$ as follows.
Let  $c_0$  be any node  in $2^{<\om}$
and declare $c_0$ to
code the vertex $v_0$.
For  $n>0$,
given  nodes $c_0,\dots,c_{n-1}$  in $2^{<\om}$ coding the vertices $v_0,\dots,v_{n-1}$,
let $c_n$ be any node in $2^{<\om}$  such that
 the length of $c_n$, denoted
$|c_n|$,
 is strictly greater than the length of $c_{n-1}$
 and
 for all $i< n$,
$c_n(|c_i|)=1$
 if and only if
$v_n$ and $ v_i$ have an edge between them.
The set  of nodes $\{c_n:n<N\}$ codes the graph $\G$.

\begin{defn}[Tree with coding nodes]\label{defn.treewcodingnodes}
A {\em tree with coding nodes}
is a structure $(T,N;\sse,<,c)$ in the language of
$\mathcal{L}=\{\sse,<,c\}$,
 where $\sse$ and $<$ are binary relation symbols  and $c$ is a unary function symbol,
 satisfying the following:
 $T$ is a subset of  $2^{<\om}$ satisfying that   $(T,\sse)$ is a tree (recall Definition \ref{defn.tree}), $N\le \om$ and $<$ is the usual linear order on $N$,  and $c:N\ra T$ is an injective  function such that  $m<n<N$ implies $|c(m)|<|c(n)|$.
\end{defn}

\begin{conv}\label{convention.l_n}
We shall use $c_n$ to denote $c(n)$ and call it the {\em $n$-th coding node} in $T$.
The length of $c_n$ shall be denoted by $l_n$.
When necessary to avoid confusion between more than one tree, the $n$-th coding node of a tree $T$ will be
denoted as $c^T_n$,  and its length as $l^T_n=|c^T_n|$.
\end{conv}

\begin{defn}\label{def.rep}
A graph $\G$ with vertices enumerated as $\lgl v_n:n<N\rgl$ is {\em  represented}  by a tree $T$ with  coding nodes $\lgl c_n:n<N\rgl$
if and only if
for each pair $i<n<N$,
 $v_n\, \E\, v_i\longleftrightarrow  c_n(l_i)=1$.
We will often simply say that $T$ {\em codes} $\G$.
\end{defn}

The next step is to determine which tree configurations code triangles, for those
 are the configurations that must  be omitted from any tree coding  a triangle-free graph.
Notice  that if $v_i,v_j,v_k$ are the vertices of some triangle, $c_i,c_j,c_k$ are coding nodes
coding these vertices, respectively, and the edge relationships between them,
and  $|c_i|<|c_j|<|c_k|$,
then  it must be the case that $c_j(|c_i|)=c_k(|c_i|)=c_k(|c_j|)=1$.
Moreover,  this is the only way a triangle can be coded by coding nodes.

Now we present a criterion which, when satisfied, guarantees that any node $t$  in the tree may be extended to a coding node without coding a triangle with any coding nodes of length less than $|t|$.

\begin{defn}[Triangle-Free   Criterion]\label{defn.trianglefreeextcrit}
Let $T\sse 2^{<\om}$ be a  tree with coding nodes $\lgl c_n:n<N\rgl$,
where  $N\le\om$.
$T$
{\em satisfies the Triangle-Free Criterion (TFC)}
if the following holds:
For each $t\in T$,
if  $l_n<|t|$ and
 $t(l_i)=c_n(l_i)=1$  for  some $i<n$,
then $t(l_n)=0$.
\end{defn}

In words, a tree  $T$ with coding  nodes $\lgl c_n:n<N\rgl$ satisfies the  TFC  if  for each $n<N$,  whenever a node $u$  in $T$ has the same length as coding node $c_n$,  and $u$ and $c_n$ both have passing number $1$ at the level of  a coding node $c_i$ for some $i<n$,
then $u^{\frown}1$ must not be in $T$.
In particular, the TFC implies  that if $c_n$ has passing number $1$ at $c_i$ for any $i<n$,
then $c_n$ cannot split; that is, ${c_n}^{\frown}1$ must not be in $T$.

\begin{rem}
The point of the TFC is as follows:
  Whenever  a finite tree  $T$ satisfies the TFC, then any maximal node of $T$ may be extended to a new  coding node without coding a triangle with the coding nodes in $T$.
\end{rem}

The next proposition provides a characterization of tree representations of triangle-free graphs.

\begin{prop}[Triangle-Free  Tree Representation]\label{prop.trianglefreerep}
Let $T\sse 2^{<\om}$ be a tree
with coding nodes $\lgl c_n:n<N\rgl$
 coding  a countable   graph $\G$  with  vertices $\lgl v_n:n<N\rgl$, where $N\le\om$.
Assume  that  the coding nodes in $T$ are dense in $T$, meaning that
for each $t\in T$, there is some coding node $c_n\in T$ such that $t\sse c_n$.
Then the  following are equivalent:
\begin{enumerate}
\item
 $\G$ is triangle-free.
\item
$T$ satisfies the Triangle-Free Criterion.
\end{enumerate}
\end{prop}

\begin{proof}
Note that if  $N$ is finite,
then the coding nodes in $T$ being dense in $T$ implies  that every maximal node in $T$ is a coding node; in this case, the maximal nodes in $T$ have different lengths.

Suppose (2) fails.
Then there are $i<j<N$
 and  $t\in T$ with length greater than $l_j$ such that
$t(l_i)=c_j(l_i)=1$ and $t(l_j)=1$.
Since every node in $T$ extends to a coding  node, there is a $k>j$ such that $c_k\contains t$.
Then   $c_k$ has passing number $1$ at both $c_i$ and $c_j$.
Thus, the coding nodes $c_i,c_j,c_k$ code that
the vertices $\{v_i,v_j,v_k\}$ have edges between each pair, implying  $\G$ contains a triangle.
Therefore, (1) fails.

Conversely, suppose that  (1) fails.  Then
$\G$ contains a triangle, so
 there are
$i<j<k<N$  such that the vertices  $v_i,v_j,v_k$ have edges between each pair.
Since the coding nodes $c_i,c_j,c_k$ code these edges,
it is the case that
$c_j(l_i)= c_k(l_i)=c_k(l_j)=1$.
Hence, the nodes $c_i,c_j,c_k$ witness the failure of the TFC.
\end{proof}

\begin{defn}[Parallel $1$'s]\label{defn.parallel1s}
For two nodes $s,t\in 2^{<\om}$, we  say that $s$ and $t$  {\em have  parallel $1$'s} if there is some $l<\min(|s|,|t|)$ such that $s(l)=t(l)=1$.
\end{defn}

\begin{defn}\label{defn.p1cstft}
Let $T$ be a tree with coding nodes $\lgl c_n:n<N\rgl$
such that, above the stem of $T$,  splitting in $T$ occurs only at the levels of coding nodes.
Then $T$
satisfies the {\em Splitting Criterion}
if for each  $n<N$ and each non-maximal $t$ in $T$ with $|t|=|c_n|$,
$t$ splits in $T$ if and only if   $t$ and $c_n$ have no parallel $1$'s.
\end{defn}

Notice that  whenever a  tree $T$ with coding nodes satisfies the Splitting Criterion,
each coding node which is not solely a sequence of $0$'s
will not split in $T$.
Thus, the Splitting Criterion
produces maximal splitting subject to
ensuring that no nodes can be extended to code a triangle,
while simultaneously  reducing the number of similarity types of trees under consideration later for the big Ramsey degrees, if we require each coding node to have at least one passing number of $1$.

Next, strong triangle-free trees are defined.
These trees provide the intuition and the main structural properties of their skewed variant defined in Section  \ref{sec.4}.

\begin{defn}[Strong triangle-free tree]\label{defn.stft}
A {\em strong triangle-free tree} is  a tree with coding nodes, $(T,N;\sse,<,c)$ such that
for each $n<N$, the length of the $n$-th coding node $c_n$ is $l_n=n+1$ and
\begin{enumerate}
\item
If $N=\om$, then $T$ has no maximal nodes.
 If $N<\om$, then all maximal nodes of $T$ have the same length,  which is  $l_{N-1}$.
\item
$\stem(T)$ is the empty sequence $\lgl\rgl$.
\item
$c_0 = \lgl 1\rgl$, and
for each $0<n<N$,   $c_n(l_{n-1})=1$.
\item
For each $n<N$, the sequence of length $l_n$ consisting of all $0$'s, denoted $0^{l_n}$, is a node in $T$.
\item
$T$ satisfies the Splitting Criterion.
\end{enumerate}
$T$ is a {\em strong triangle-free tree densely coding $\mathcal{H}_3$} if $T$ is an infinite  strong triangle-free tree and
the set of coding nodes  is dense in $T$.
\end{defn}

Strong triangle-free trees  can be defined more generally than  we choose to  present here,  for instance,
 by relaxing conditions (2) and  (3),  leaving off the restriction  that $l_n=n+1$, and letting $c_0$ be any node.
The notion of strong  subtree of a given strong triangle-free tree  can also  be made precise, and the collection of such trees  end up forming  a space somewhat similar
 to the Milliken space of strong trees.
However, as
Milliken-style theorems are  impossible to prove for
strong triangle-free trees, as will be shown in Example \ref{ex.bc},
we restrict here to  a
 simpler   presentation
with the aim of
 building  the reader's understanding of the essential structure of strong triangle-free trees,
as the
 strong coding trees defined  in the next section  are skewed and slightly relaxed  versions of trees in Definition \ref{defn.stft}.

We now set  up to present  a  method for constructing strong triangle-free trees densely coding $\mathcal{H}_3$.
Let  $\mathcal{K}_3$ denote the \Fraisse\ class of all triangle-free  countable graphs.
Given a graph $\HH$ and a subset $V_0$ of the vertices of $\HH$,
the notation $\HH|V_0$ denotes the induced subgraph  of $\HH$ on the vertices in $V_0$.
In  \cite{Henson71},
Henson proved that a countable graph $\HH$ is universal for  $\mathcal{K}_3$ if and only if $\HH$ satisfies the following property.
\begin{enumerate}
\item[($A_3$)]
\begin{enumerate}
\item[(i)]
$\HH$ does not admit any triangles,
\item[(ii)]
If $V_0,V_1$ are disjoint finite sets of vertices of $\HH$ and $\HH|V_0$ does not admit an edge,
then there is another vertex which is connected in $\HH$ to every member of $V_0$ and to no member of $V_1$.
\end{enumerate}
\end{enumerate}

Henson used this property to construct a universal triangle-free graph $\mathcal{H}_3$ in \cite{Henson71}, as well as universal graphs for each \Fraisse\ class of countable graphs omitting $k$-cliques,
 as the analogues of the Rado graph for countable $k$-clique free graphs.
The following property $(A_3)'$ is a reformulation of
Henson's property $(A_3)$.

\begin{enumerate}
\item[$(A_3)'$]
\begin{enumerate}
\item[(i)]
$\HH$ does not admit any triangles.
\item[(ii)]
Let $\lgl v_n:n<\om\rgl$ enumerate the vertices of $\HH$, and
let $\lgl F_i:i<\om\rgl$ be any enumeration of the finite subsets of $\om$ such that for each $i<\om$, $\max(F_i)<i$ and each finite set  appears  infinitely many times in the enumeration.
Then there is a strictly increasing sequence $\lgl n_i: i<\om\rgl$
such that for each $i<\om$, if  $\HH|\{v_m:m\in F_i\}$
has no edges,
then  for all $m<i$, $v_{n_i} \, E\, v_m\longleftrightarrow m\in F_i$.
\end{enumerate}
\end{enumerate}

It is straightforward to check the following fact.

\begin{fact}\label{fact.A_3'}
Let $\HH$ be a  countably infinite graph. Then
$\HH$
is universal for $\mathcal{K}_3$ if and only if
$(A_3)'$ holds.
\end{fact}

The following   re-formulation of  property $(A_3)'$  will be used  to build trees with coding nodes which code $\mathcal{H}_3$.
Let
$T\sse 2^{<\om}$  be a tree with coding  nodes $\lgl c_n:n<\om\rgl$.
We say that $T$ {\em satisfies property $(A_3)^{\tt tree}$} if the following holds:

\begin{enumerate}
\item[$(A_3)^{\tt{tree}}$]
\begin{enumerate}
\item[(i)]
$T$ satisfies the Triangle-Free Criterion,
\item[(ii)]
Let  $\lgl F_i:i<\om\rgl$   be any enumeration  of finite subsets of $\om$ such that
for each $i<\om$, $\max(F_i)<i$, and
 each finite subset of $\om$ appears as $F_i$ for infinitely many indices $i$.
For each $i<\om$,
if for all pairs $j<k$ in $F_i$
it is the case that
 $c_k(l_j)=0$ ,
then
there is some $n\ge i$ such that
for all $m<i$,
$c_n(l_m)=1$ if and only if $ m\in F_i$.
\end{enumerate}
\end{enumerate}

\begin{fact}\label{fact.A_3treeimpliestrianglefreegraph}
A tree  $T$  with coding nodes $\lgl c_n:n<\om\rgl$ codes $\mathcal{H}_3$  if and only if
$T$ satisfies $(A_3)^{\tt tree}$.
\end{fact}

\begin{rem}\label{rem.dense}
Any strong triangle-free tree in which the coding nodes are dense automatically satisfies $(A_3)^{\tt tree}$, and hence codes $\mathcal{H}_3$.
\end{rem}

The next lemma shows that any finite strong triangle-free tree can  be extended to a tree satisfying $(A_3)^{\tt tree}$.

\begin{lem}\label{lem.stftextension}
Let $T$ be a finite strong triangle-free tree with coding nodes $\lgl c_n:n<N\rgl$, where $N<\om$.
Given any $F\sse N-1$ for which the set $\{c_n:n\in F\}$ codes no edges,
there is a maximal node $t\in T$ such that
for all $n<N-1$,
\begin{equation}
t(l_n)=1\ \ \longleftrightarrow \ \ n\in F.
\end{equation}
\end{lem}

\begin{proof}
The proof is by induction on $N$ over all strong triangle-free trees with $N$ coding nodes.
For $N\le 1$,
the lemma trivially holds but is not very instructive,
so we  shall start with  the case $N=2$.
Let $T$ be a strong triangle-free tree with coding nodes $\{ c_0,c_1 \}$.
By (2) of Definition \ref{defn.stft},
the stem of $T$ is the empty sequence,
so both $\lgl 0\rgl$ and $\lgl 1\rgl$ are in $T$.
By (3) of Definition \ref{defn.stft},
$c_0=\lgl 1\rgl$,
and $c_1(l_0)=1$.
By the Splitting Criterion,  $c_0$ does not split in $T$ but $\lgl 0\rgl$ does,
so $\lgl 0,0\rgl$, $\lgl 0,1\rgl$,  and $\lgl 1,0\rgl$ are in $T$  while $\lgl 1,1\rgl$ is not in $T$.
Note that $c_1=\lgl 0,1\rgl$, since  it must be that
$l_1=2$ and
 $c_1(l_0)=1$,
and $\lgl 1,1\rgl$ is not in $T$.
The only non-empty $F\sse 1$ is $F=\{0\}$.
The coding node $c_1$ satisfies that $c_1(l_n)=1$ if and only if $n\in \{0\}$.
For $F=\emptyset$,
both the nodes $t=\lgl 0,0 \rgl$ and $t=\lgl 1,0\rgl$ satisfy
that for all $n<1$, $t(l_n)=1$  if and only if $n\in F$.

Now assume that the lemma holds for all $N'<N$, where $N\ge 3$.
Let $T$ be a strong triangle-free tree with $N$ coding nodes.
Let $F$ be a subset of $N-1$ such that $\{c_n:n\in F\}$ codes no edges.
By the induction hypothesis, there is a node $t$ in $T$ of length $l_{N-2}$ such that
for all $n<N-2$, $t(l_n)=1$ if and only if $n\in F$.
If  $N-2\not\in F$,
then
as $t^{\frown}0$ is guaranteed to be in $T$
by the Splitting Criterion,
the  node $t'=t^{\frown}0$ in $T$  satisfies that
 for all $n<N-1$,
$t'(l_n)=1$ if and only if $n\in F$.
Now  suppose $N-2\in F$.
We claim that $t^{\frown}1$ is in $T$.
By the Splitting Criterion,
if  $t^{\frown}1$ is not in $T$,
then it must be the case that $t$ and $c_{N-2}$ have a parallel $1$.
So there is some $i<N-2$ such that $t(l_i)=c_{N-2}(l_i)=1$.
As $t$ codes edges only with those vertices with indexes  $n<N-2$ which are in $F\setminus\{N-2\}$,
it follows that $i\in F$.
But then $\{c_i,c_{N-2}\}$ codes an edge, contradicting the assumption on $F$.
Therefore, $t$ and $c_{N-2}$ do not have any parallel $1$'s,
and hence $t^{\frown}1$ is in $T$.
Letting $t'=t^{\frown}1$,
we see that for all $n<N-1$,
$t(l_n)=1$ if and only if $n\in F$.
\end{proof}

We now present a  method for constructing   strong triangle-free trees densely coding $\mathcal{H}_3$.
Here and throughout the paper, $0^n$ denotes the sequence of length $n$ consisting of all $0$'s.

\begin{thm}[Strong Triangle-Free Tree $\bS$  Densely Coding $\mathcal{H}_3$]\label{thm.stftree}
Let  $\lgl F_i:i<\om\rgl$  be any sequence enumerating the finite subsets of $\om$  so  that each finite set appears infinitely often.
Assume  that
  for each $i<\om$,
 $F_i\sse i-1$ and  $F_{3i}=F_{3i+2}=\emptyset$.
Then there is a strong triangle-free tree $\bS$ which
 satisfies
property $(A_3)^{\tt tree}$
and
 densely codes $\mathcal{H}_3$.
Moreover, this property is satisfied specifically by
the coding node $c_{4i+j}$  meeting  requirement $F_{3i+j}$,
 for each $i<\om$ and $j\le 2$.
\end{thm}

\begin{proof}
Let  $\lgl F_i:i<\om \rgl$  satisfy the hypotheses.
Enumerate the nodes in $2^{<\om}$ as
 $\lgl u_i: i<\om\rgl$
 in such a manner that $i<k$ implies $|u_i|\le |u_k|$.
Then $u_0=\emptyset$, $|u_1|=1$, and for all $i\ge 2$, $|u_i|<i$.
We will build a strong triangle-free tree $\bS\sse 2^{<\om}$ with  coding nodes $c_n\in\bS\cap 2^{n+1}$ densely coding $\mathcal{H}_3$ satisfying the following properties:
\begin{enumerate}

\item[(i)]
$c_0=\lgl 1\rgl$, and
for each $n<\om$,  $l_n:=|c_n|=n+1$ and   $c_{n+1}(l_n)=1$.
\item[(ii)]
For $n=4i+j$, where  $j\le 2$, $c_{n}$ {\em satisfies requirement $F_{3i+j}$},
meaning that if $\{c_k:k\in F_{3i+j}\}$ codes no edges,
then for all $k< n-1$,
$c_n(l_k)=1$ if and only if $k\in F_{3i+j}$.
\item[(iii)]
For $n=4i+3$, if $u_i$ is in  $\bS\cap 2^{\le n}$,
then $c_n$ is a coding node extending $u_i$.
If $u_i$ is not in
$\bS$, then
 $c_n={0^n}^{\frown}1$.
\end{enumerate}

As in Lemma \ref{lem.stftextension},
the first two coding nodes of $\bS$ are completely determined by the definition of strong triangle-free tree.
Thus, $c_0=\lgl 1\rgl$, $c_1=\lgl 0,1\rgl$,
and the tree $\bS$ up to height $2$ consists of the nodes
$\{\emptyset,\lgl 0\rgl, \lgl 1\rgl, \lgl 0,0\rgl,\lgl 0,1\rgl,\lgl 1,0\rgl\}$.
Denote this tree as $\bS_2$.
Since  $F_0=F_1=\emptyset$,  $c_0$ and $c_1$ trivially satisfy requirements $F_0$ and $F_1$, respectively.
It is simple to check that $\bS_2$  is a strong triangle-free tree,  and that (i) - (iii) are satisfied.

For the general construction step, suppose $n\ge 2$,
 $\bS_n \sse  2^{\le n}$ has been constructed,
 and  coding nodes $\lgl c_i :i<n\rgl$ have been chosen so that $\bS_n$  is a strong triangle-free tree satisfying (i) - (iii).
Extend each maximal node in $\bS_n$  to length $n+1$ according to the Splitting Criterion.
Thus,
for each $s\in \bS_n\cap {2}^n $,
 $s^{\frown}0$ is  in $\bS_{n+1}$,
and
 $s^{\frown}1$ is in $\bS_{n+1}$ if and only if
$s$  has no parallel $1$'s with $c_{n-1}$.
Now we choose $c_n$ so that  (i) - (iii) hold.
There are three cases.
\vskip.1in

\it Case 1. \rm
Either
$n=4i$  and $i\ge 1$, or $n=4i+2$ and  $i<\om$.
Let $n'$ denote $3i$ if $n=4i$, and let $n'$ denote $3i+2$ if $n=4i+2$.
Then $F_{n'}=\emptyset$, so
let $c_n={0^n}^{\frown}1$.

\it Case 2. \rm
$n=4i+1$ and $1\le i<\om$.
If
 for all pairs of integers  $k<m$  in $F_{3i+1}$ it is the case that
 $c_m(l_k)=0$, then
take $c_n$ to be a maximal node in $\bS_{n+1}$ such that  for all $k<n-1$,
$c_n(l_k)=1$  if and only if $k\in F_{3i+1}$,
and $c_n(l_{n-1})=1$.
Otherwise, let $c_n={0^n}^{\frown}1$.

\it Case 3.  \rm
$n=4i+3$ and $i<\om$.
Recall that  $|u_i|\le i$, so $|u_i|\le n-3$.
If $u_i$ is in  $\bS_i$,
then take $c_n$ to be the maximal node in $\bS_{n+1}$ which is  $u_i$ extended by all $0$'s until  its last entry, which  is $1$.
Precisely, letting $q=n-|u_i|$,
set $c_n={u_i}^{\frown}{0^q}^{\frown}1$.
If $u_i$ is not in $ \bS_i$, let  $c_n={0^n}^{\frown}1$.
\vskip.1in

(i) - (iii) hold automatically by the choices of $c_n$ in  Cases 1 - 3.
What is left  is to check is that such nodes in Cases 1 - 3 actually exist in $\bS_{n+1}$.
The node ${0^n}^{\frown}1$ is in $\bS_{n+1}$, as  it  has no parallel $1$'s with  $c_{n-1}$.
Thus, in Case 1 and the second halves of Cases 2 and 3, the node we declared to be $c_n$ is indeed  in $\bS_{n+1}$.

In  Case 2 where $n=4i+1$ with $i\ge 1$, suppose that  $F_{3i+1}\ne\emptyset$
and  for all pairs $k<m$ of integers in $F_{3i+1}$,
 $c_m(l_k)=0$.
Since $\max(F_{3i+1})\le 3i-1\le n-3$ and since  by the induction hypothesis,
$\bS_{n-1}$ is a strong triangle-free tree,
Lemma \ref{lem.stftextension}
implies that there is a node $t\in \bS_{n-1}$ such that
for each $k<n-1$,
$s(l_k)=1$ if and only if  $k\in F_{3i+1}$.
Note that
$t^{\frown}0$ and $c_{n-1}$ have no parallel $1$'s,
since $c_{n-1}={0^{n-1}}^{\frown}1$.
Thus, by the Splitting Criterion,
$t^{\frown}0^{\frown} 1$ is in $\bS_{n+1}$, and this node satisfies our choice of $c_n$.

In
 Case 3 when $n=4i+3$, if  $u_i\in \bS_i$,
then by the Splitting Criterion, also
${u_i}^{\frown}0^q$ is in $\bS_n$,
where $q=n-|u_i|$.
Since $n-1=4i+2$,  $c_{n-1}={0^{n-1}}^{\frown}1$;
so ${u_i}^{\frown}0^q$ has no parallel $1$'s with $c_{n-1}$.
Thus, by the Splitting Criterion,
${u_i}^{\frown}{0^q}^{\frown}1$  is  in $\bS_{n+1}$.

Let $\bS=\bigcup_{n<\om}\bS_n$.
By  the construction, $\bS$  is an infinite  strong triangle-free tree with coding nodes $\lgl c_n:n<\om\rgl$.
(ii) implies that $\bS$ satisfies $(A_3)^{\tt tree}$ and hence codes $\mathcal{H}_3$.
By (iii), the coding nodes are dense in $\bS$.
\end{proof}

\begin{example}[A Strong Triangle-Free Tree]\label{ex.stft}
Presented here is a concrete example of the first six steps of constructing a strong triangle-free tree densely coding $\mathcal{H}_3$.
In the  construction of Theorem \ref{thm.stftree}, $F_0=F_1=F_2=\emptyset$.
 The coding nodes  $c_0=\lgl 1\rgl$ and $c_1=\lgl 0,1\rgl$ are determined by the definition of strong triange-free tree.
The coding node $c_2$ we choose to be $\lgl 0,0,1\rgl$.
(It could also have been  chosen to be $\lgl 1,0,1\rgl$.)
Since $u_0$ is the empty sequence,  $c_3$ can be any  sequence which has last entry $1$; in this example we let $c_3=\lgl 1,0,0,1\rgl$.
$F_3=\emptyset$, so $c_4=\lgl 0,0,0,0,1\rgl$.
Suppose $F_4=\{0,2\}$.
Then we may take $c_5=\lgl 0,1,0,1,0,1\rgl $ to code edges between the vertex $v_5$ and the vertices $v_0$ and $v_2$;
we also make $v_5$ have an edge with $v_4$.
Notice that having
chosen the coding node $c_n$,
each  maximal node $s\in\bS_{n+1}$
splits in $\bS_{n+2}$ if and only if
   $s(i)+c_n(i)\le 1$ for all $i\le n$.
See Figure 3.
The graph on the left with vertices $\{v_0,\dots,v_5\}$ is being coded by the coding nodes $\{c_0,\dots,c_5\}$.
The tree and the graph are intended to continue growing upwards to the infinite tree $\bS$ coding the graph $\mathcal{H}_3$.


\begin{figure}\label{fig.bS}
\begin{tikzpicture}[grow'=up,scale=.58]

\tikzstyle{level 1}=[sibling distance=5in]
\tikzstyle{level 2}=[sibling distance=3in]
\tikzstyle{level 3}=[sibling distance=1.8in]
\tikzstyle{level 4}=[sibling distance=1.1in]
\tikzstyle{level 5}=[sibling distance=.7in]
\tikzstyle{level 6}=[sibling distance=0.4in]
\tikzstyle{level 7}=[sibling distance=0.2in]

\node {$\lgl \rgl$}
child{coordinate (t0)
			child{coordinate (t00)
child{coordinate (t000)
child {coordinate(t0000)
child{coordinate(t00000)
child{coordinate(t000000)
child{coordinate(t0000000)}
child{coordinate(t0000001)}
}
child{coordinate(t000001)
child{coordinate(t0000010)}
child{ edge from parent[draw=none]  coordinate(t0000011)}
}
}
child{coordinate(t00001)
child{coordinate(t000010)
child{coordinate(t0000100)}
child{coordinate(t0000101)}
}
child{ edge from parent[draw=none]  coordinate(t000011)
}
}}
child {coordinate(t0001)
child {coordinate(t00010)
child {coordinate(t000100)
child {coordinate(t0001000)}
child { edge from parent[draw=none] coordinate(t0001001)}
}
child {coordinate(t000101)
child {coordinate(t0001010)}
child { edge from parent[draw=none]  coordinate(t0001011)}
}
}
child{coordinate(t00011) edge from parent[draw=none] }}}
child{ coordinate(t001)
child{ coordinate(t0010)
child{ coordinate(t00100)
child{ coordinate(t001000)
child{ coordinate(t0010000)}
child{ coordinate(t0010001)}
}
child{ coordinate(t001001)
child{ coordinate(t0010010)}
child{ edge from parent[draw=none] coordinate(t0010011)}
}
}
child{ coordinate(t00101)
child{ coordinate(t001010)
child{ coordinate(t0010100)}
child{ coordinate(t0010101)}
}
child{   edge from parent[draw=none]coordinate(t001011)
}
}}
child{  edge from parent[draw=none] coordinate(t0011)}}}
			child{ coordinate(t01)
child{ coordinate(t010)
child{ coordinate(t0100)
child{ coordinate(t01000)
child{ coordinate(t010000)
child{ coordinate(t0100000)}
child{ edge from parent[draw=none]  coordinate(t0100001)}
}
child{ coordinate(t010001)
child{ coordinate(t0100010)}
child{edge from parent[draw=none]  coordinate(t0100011)}
}
}
child{ coordinate(t01001)
child{ coordinate(t010010)
child{ coordinate(t0100100)}
child{ edge from parent[draw=none]  coordinate(t0100101)}
}
child{edge from parent[draw=none]  coordinate(t010011)}
}}
child{ coordinate(t0101)
child{ coordinate(t01010)
child{ coordinate(t010100)
child{ coordinate(t0101000)}
child{edge from parent[draw=none]  coordinate(t0101001)}
}
child{ coordinate(t010101)
child{ coordinate(t0101010)}
child{edge from parent[draw=none]  coordinate(t0101011)}
}
}
child{  edge from parent[draw=none]  coordinate(t01011)
}}}
child{ edge from parent[draw=none]  coordinate(t011)}}}
		child{ coordinate(t1)
			child{ coordinate(t10)
child{ coordinate(t100)
child{ coordinate(t1000)
child{ coordinate(t10000)
child{ coordinate(t100000)
child{ coordinate(t1000000)}
child{ coordinate(t1000001)}
}
child{ coordinate(t100001)
child{ coordinate(t1000010)}
child{ edge from parent[draw=none] coordinate(t1000011)}
}
}
child{ edge from parent[draw=none] coordinate(t10001)
}}
child{ coordinate(t1001)
child{ coordinate(t10010)
child{ coordinate(t100100)
child{ coordinate(t1001000)}
child{edge from parent[draw=none]  coordinate(t1001001)}
}
child{ coordinate(t100101)
child{ coordinate(t1001010)}
child{edge from parent[draw=none]   coordinate(t1001011)}
}
}
child{  edge from parent[draw=none] coordinate(t10011)
}}}
child{ coordinate(t101)
child{ coordinate(t1010)
child{ coordinate(t10100)
child{ coordinate(t101000)
child{ coordinate(t1010000)}
child{ coordinate(t1010001)}
}
child{ coordinate(t101001)
child{ coordinate(t1010010)}
child{   edge from parent[draw=none]   coordinate(t1010011)}
}
}
child{   edge from parent[draw=none]  coordinate(t10101)
}}
child{ edge from parent[draw=none] coordinate(t1011)}}}
child{  edge from parent[draw=none] coordinate(t11)} };

\node[right] at (t1) {$c_0$};
\node[left] at (t01) {$c_1$};
\node[left] at (t001) {$c_2$};
\node[right] at (t1001) {$c_3$};
\node[left] at (t00001) {$c_4$};
\node[right] at (t010101) {$c_5$};

\node[circle, fill=black,inner sep=0pt, minimum size=5pt] at (t1) {};
\node[circle, fill=black,inner sep=0pt, minimum size=5pt] at (t01) {};
\node[circle, fill=black,inner sep=0pt, minimum size=5pt] at (t001) {};
\node[circle, fill=black,inner sep=0pt, minimum size=5pt] at (t1001) {};
\node[circle, fill=black,inner sep=0pt, minimum size=5pt] at (t00001) {};
\node[circle, fill=black,inner sep=0pt, minimum size=5pt] at (t010101) {};

\draw[thick, dotted] let \p1=(t1) in (-18,\y1) node (v0) {$\bullet$} -- (7,\y1);
\draw[thick, dotted] let \p1=(t01) in (-18,\y1) node (v1) {$\bullet$} -- (7,\y1);
\draw[thick, dotted] let \p1=(t001) in (-18,\y1) node (v2) {$\bullet$} -- (7,\y1);
\draw[thick, dotted] let \p1=(t1001) in (-18,\y1) node (v3) {$\bullet$} -- (7,\y1);
\draw[thick, dotted] let \p1=(t00001) in (-18,\y1) node (v4) {$\bullet$} -- (7,\y1);
\draw[thick, dotted] let \p1=(t010101) in (-18,\y1) node (v5) {$\bullet$} -- (7,\y1);

\node[left] at (v0) {$v_0$};
\node[left] at (v1) {$v_1$};
\node[left] at (v2) {$v_2$};
\node[left] at (v3) {$v_3$};
\node[left] at (v4) {$v_4$};
\node[left] at (v5) {$v_5$};

\draw[thick] (v0.center) to (v1.center) to (v2.center) to (v3.center) to [bend left] (v0.center);
\draw[thick] (v3.center) to (v4.center) to (v5.center);
\draw[thick] (v0.center) to [bend right] (v5.center);
\draw[thick] (v5.center) to [bend left] (v2.center);

\end{tikzpicture}
\caption{A strong triangle-free tree $\bS$ densely coding $\mathcal{H}_3$}
\end{figure}
\end{example}


\begin{rem}\label{rem.properties}
We have set up the definition of strong triangle-free tree so that no coding node in a strong triangle-free tree splits.
The purpose  of this  is to simplify later work by reducing the number of different isomorphism types of trees coding  a given finite triangle-free graph.
The purposes of the density of the coding  nodes  and the Splitting Criterion are to saturate the trees with as many extensions as possible coding vertices without coding any triangles, so as to
allow for  thinning to subtrees which  still can code $\mathcal{H}_3$, setting the stage for later Ramsey-theoretic results.
\end{rem}

\begin{rem}\label{rem.badcoloring}
Given a strong triangle-free tree $T$ densely coding $\mathcal{H}_3$,
the collection of all
 strong triangle-free subtrees  $S$ of $T$ densely coding $\mathcal{H}_3$  forms an interesting space of trees.
The author has  proved
 Halpern-\Lauchli-style  theorems for such trees, provided that the stem is fixed.
This was the author's first approach toward the main theorem of this paper, and these proofs formed the strategy for the proofs in later sections.
However, the introduction of coding nodes hinders a full development of Ramsey theory for trees which have splitting nodes and coding nodes of the same length, as shown in the next example.
Such a bad coloring on coding nodes prevents the transition from cone-homogeneity to homogeneity on a strong triangle-free subtree with dense coding nodes.
\end{rem}

\begin{example}[A bad coloring]\label{ex.bc}
Given a strong triangle-free tree $\bS$ with coding nodes $\lgl c_n:n<\om\rgl$ dense in $\bS$,
let
$s_i=0^{i}$, for each $i<\om$.
Note that each $s_i$ splits in $\bS$ and  that
 $|c_n|=|s_{n+1}|$, for each $n<\om$.
Color all coding nodes $c_n$ extending  ${s_0}^{\frown}1$, which is exactly $\lgl 1\rgl$,   blue.
Let $k$ be given and suppose
for each $i\le k$,
we have colored all coding nodes extending ${s_i}^{\frown}1$.
 The coding node $c_k$  extends ${s_i}^{\frown}1$ for some $i\le k$, so it has already been assigned a color.
If $c_k$ is blue, color every coding node in $\bS$ extending ${s_{k+1}}^{\frown}1$ red;
if $c_k$ is red, color every coding node in $\bS$ extending ${s_{k+1}}^{\frown}1$ blue.
This produces a red-blue coloring of the coding nodes  such that
any  subtree $S$  of $\bS$ with coding nodes dense in $S$  and satisfying the Splitting Criterion (which would be the natural definition  of  infinite strong triangle-free subtree) has coding nodes of both colors:
For given a coding node $c$  of $S$, the node $0^{|c|}$ is a splitting node in $S$,
and all coding nodes in $S$ extending  ${0^{|c|}}^{\frown}1$ have color different from the color of $c$.
\end{example}

Since this example precludes a satisfactory  Ramsey theory of strong triangle-free trees coding $\mathcal{H}_3$,
 instead of  presenting those  Ramsey-theoretic results on strong triangle-free trees which were obtained,
we immediately move on to
  the skew version of strong triangle-free trees.
Their full Ramsey theory will be developed in the rest of the article.


\section{Strong coding trees}\label{sec.4}

This section introduces the main tool for  our investigation of the big Ramsey degrees for the universal triangle-free graph, namely {\em strong coding trees}.
Essentially,  strong coding trees are simply stretched versions of strong triangle-free trees, with
all the coding structure  preserved while
removing  any
 entanglements  between coding nodes and splitting nodes
 which could prevent Ramsey theorems, as in  Example \ref{ex.bc}.
The collection of all subtrees of a strong coding tree $T$ which are isomorphic to $T$, partially ordered by a relation
defined later in this section,
will be seen, by the end of Section \ref{sec.1SPOC},
to
  form a space of trees coding $\mathcal{H}_3$
with many  similarities to  the Milliken space of strong trees \cite{Milliken79}.


\subsection{Definitions and notation}\label{subsection.4.1}

The following terminology and notation will be used throughout.
Recall that
by a  {\em tree}, we mean exactly a subset $T\sse 2^{<\om}$ which is closed under meets and is a union of level sets; that is, $s,t\in T$ and $|t|\ge |s|$ imply that $t\re |s|$ is also a member of $T$.
Further, recall Definition
\ref{defn.treewcodingnodes} of a
 {\em tree with coding nodes}.
Let  $T\sse 2^{<\om}$  be a tree with coding nodes $\lgl c^T_n:n<N\rgl$, where $N\le \om$,
and let
  $l^T_n$ denote $|c^T_n|$.
$\widehat{T}$ denotes the collection of all initial segments of nodes in $T$; thus,
$\widehat{T}=\{t \re n:t\in T$ and $n\le |t|\}$.
A node $s\in T$  is called  a {\em splitting node}  if both $s^{\frown}0$ and $s^{\frown}1$  are in $\widehat{T}$;
equivalently,  $s$ is a splitting node in $T$ if there are nodes $s_0,s_1\in T$ such that
$s_0\contains s^{\frown}0$ and $s_1\contains s^{\frown}1$.
Given $t$ in a tree  $T$, the  {\em level} of  $T$ of length  $|t|$  is  the set of all $s\in T$ such that $|s|=|t|$.
By our definition of tree,
this is exactly the set of $s\re|t|$ such that $s\in T$ and $|s|\ge |t|$.
 $T$   is {\em skew} if
each level of $T$ has exactly one of either a  coding node or a splitting node.
A skew tree
$T$ is {\em strongly skew} if  additionally
for each splitting node $s\in T$,
every
 $t\in T$ such that $|t|>|s|$ and  $t\not\supset s$ also
satisfies
$t(|s|)=0$;
that is, the passing number  of any node  passing by, but not extending, a splitting node
is $0$.
The set of {\em levels} of a skew tree $T\sse 2^{<\om}$, denoted $L^T$, is the set  of those $l<\om$ such that $T$ has either a splitting or a coding node
of length $l$.
Let $\lgl d^T_m:m<M\rgl$  enumerate the collection of all  coding  and splitting nodes of $T$ in increasing order of length.
The  nodes $d^T_m$ will be called the {\em critical nodes} of $T$.
Note that $N\le M$, and $M=\om$ if and only if $N=\om$.
For each $m<M$,
the {\em $m$-th level of $T$} is
\begin{equation}
\Lev_T(m)=\{s\in \widehat{T}:|s|=|d^T_m|\}.
\end{equation}
Then for any strongly skew tree $T$,
\begin{equation}
T=\bigcup_{m<M}\Lev_T(m).
 \end{equation}
Let $m_n$ denote the integer
such that $c^T_n\in \Lev_T(m_n)$.
Then $d^T_{m_n}=c^T_n$,
and the critical node
$d^T_m$ is a splitting node if and only if
 $m\ne m_n$ for any $n$.
For each $0<n<N$,
the {\em $n$-th interval} of $T$
is $\bigcup\{\Lev_T(m):  m_{n-1}<m\le m_n\}$.
 The {\em $0$-th interval} of $T$  is  defined to be $\bigcup_{m\le m_0}\Lev_T(m)$.
Thus, the $0$-th interval of $T$ is the set of those nodes in $T$ with lengths in $[0,l^T_0]$, and for $0<n<N$, the $n$-th interval of $T$ is the set of those nodes in  $T$ with lengths in $(l^T_{n-1},l^T_n]$.

The next definition  provides   notation for the set of  exactly those nodes just above the $(n-1)$-st coding node which will split  in the $n$-th interval of $T$.
Define
\begin{equation}
\Spl(T,0)=\{t\in \widehat{T}:
|t|=|\stem(T)|+1 \mathrm{\ and\ } \exists m<m_0
\mathrm{\ such\ that\ }d^T_m\contains t\}.
\end{equation}
For $n\ge 1$,
define
\begin{equation}
\Spl(T,n)=\{t\in \widehat{T}:
|t|=l_{n-1}+1\mathrm{\ and\ }
\exists
m\in(m_{n-1},m_n) \mathrm{\ such\ that\ } d^T_m\contains  t\}.
\end{equation}
Thus, $\Spl(T,n)$ is the set of nodes in $\widehat{T}$ of length
just one
 above the length of $c_{n-1}$ (or the stem of $T$ if $n=0$)
which extend to a splitting node  in the $n$-th interval of  $T$.
The lengths of the nodes in $\Spl(T,n)$ were chosen to so that they provide  information about passing numbers at $c_{n-1}^T$.
For  $t\in\Spl(T,n)$, let $\spl_T(t)$ denote the minimal extension of $t$ which splits in $T$.

Given a node $s$  in $T$ for which there is an $i< |s|$ such that $s\re i$ is a splitting node in $T$,
 the {\em splitting predecessor of $t$ in $T$}, denoted
$\splitpred_T(s)$,
is  the proper initial segment $u\subset s$ of maximum length such that both $u^{\frown}0$ and $u^{\frown}1$ are  in $\widehat{T}$.
Thus, $\splitpred_T(s)$ is the longest  splitting node in $T$ which is a   proper initial segment   of $s$.
When the tree  $T$ is clear from the context,  the  subscripts  and superscripts of  $T$ will be dropped.


\subsection{Definition and construction of strong coding trees}\label{defconsct}
Now we present a new tool for representing the universal triangle-free graph, namely strong coding trees.
The following \POC\   is a  central concept,
 ensuring that  a  finite subtree of a strong coding tree $T$ can be extended inside $T$ so that the criterion $(A_3)^{\tt tree}$ can be met.

\begin{defn}[\POC]\label{defn.parallel1sProp}
Let $T\sse 2^{<\om}$ be a  strongly skew tree with coding nodes $\lgl c_n:n<N\rgl$.
We say that $T$ satisfies the {\em \POC} if
the following hold:
Given any
 set of two or more nodes $\{t_i:i<\tilde{i}\}$ in $T$
and some $l$ such that $t_i\re(l+1)$, $i<\tilde{i}$, are all distinct, and
 $t_i(l)=1$ for all $i<\tilde{i}$,
\begin{enumerate}
\item
There is a coding node $c_n$ in $T$ such that
 for all $i<\tilde{i}$,
$l_n<|t_i|$ and
$t_i(l_n)=1$;
we say that  $c_n$ {\em witnesses} the parallel $1$'s of $\{t_i:i<\tilde{i}\}$.
\item
Letting $l'$ be  least such that
$t_i(l')=1$ for all $i<\tilde{i}$,
and  letting
$n$ be
 least  such that $c_{n}$ witnesses the parallel $1$'s of the set of nodes $\{t_i:i<\tilde{i}\}$,
then $T$ has no splitting nodes and no coding nodes
 of lengths strictly  between $l'$ and $l_{n}$.
\end{enumerate}
\end{defn}

We say that a set of nodes $\{t_i:i<\tilde{i}\}$ has a new set of parallel $1$'s at $l$ if $t_i(l)=1$ for all $i<\tilde{i}$, and $l$ is least such that this occurs.
Thus, the \POC\ says that any new set of parallel $1$'s must occur at a level $l$  which is above the last splitting node in $T$ in the interval $(l_{n-1},l_n]$ containing $l$,
and that  $c_n$ must witness this set of parallel $1$'s.

\begin{defn}[Splitting Criterion for Skew Trees]\label{defn.splitcritst}
A  strongly skew tree
$T$ with coding nodes $\lgl c_n:n<N\rgl$  satisfies the {\em Splitting Criterion for Skew Trees} if  the following hold:
For each $1\le n<N$ and each  $s\in \widehat{T}$ of length $l_{n-1}+1$,
$s$ is in $\Spl(T,n)$ if and only if
 $s$ and $c_n\re (l_{n-1}+1)$ have no parallel $1$'s.
For each $s\in\widehat{T}$  of length $|\stem(T)|+1$,
$s$ is in $\Spl(T,0)$ if and  only if $s={\stem(T)}^{\frown}0$.
\end{defn}

Notice   that any tree with coding nodes satisfying the Splitting Criterion for Skew Trees also satisfies the Triangle-Free Criterion (Definition \ref{defn.trianglefreeextcrit}),
and hence will not code any triangles.

Now we arrive at the main structural  concept for coding copies of $\mathcal{H}_3$.
This  extends the idea of  Milliken's strong trees  - branching as much as possible whenever one split occurs -
to skew trees
with the additional property that they can code omissions of triangles.

\begin{defn}[Strong  coding tree]\label{def.diagtreeH_3}
A  tree $T\sse 2^{<\om}$ with coding nodes  $\lgl c_n:n<\om\rgl$
 is a {\em   strong coding tree} if
$T$ is strongly skew,
for each node $t\in T$,  the node $0^{|t|}$ is also in $T$, and
the following hold:
\begin{enumerate}
\item
The coding nodes of $T$ are dense in $T$.
\item
For each $n\ge 1$,  $c_n(l_{n-1})=1$.
\item
$T$ satisfies the \POC.
\item
$T$ satisfies the  Splitting Criterion for Skew Trees.
\item
 $c_0$ extends ${\stem(T)}^{\frown}1$ and does not split.
\item
Given $n<\om$, $s\in \Spl(T,n)$, and  $i<2$, there is exactly one extension $s_i\supseteq \spl(s)^{\frown}i$ of length $l_n$ in $T$, and
its unique immediate extension in $\widehat{T}$ is ${s_i}^{\frown}i$.
\item
For each $n<\om$,
each node $t$ in
$\widehat{T}$ of length $l_{n-1}+1$  which is  not in $\Spl(T,n)$
 has exactly one extension of length $l_n$ in $T$,
 say $t_*$,
and its unique immediate extension  in $\widehat{T}$ is ${t_*}^{\frown}0$.
Here, $l_{-1}$ denotes the length of $\stem(T)$.
\end{enumerate}
\end{defn}

An example of a strong coding tree is presented in Figure 4.
One should notice that upon ``zipping up"
 the splits occurring
in the intervals between coding nodes in $\bT$ to the next coding node level,
 one recovers the strong triangle-free tree $\bS$ from the previous section.
The existence of strong coding trees will be proved
 in Theorem \ref{thm.cool}.


\begin{figure}\label{fig.bT}
\begin{tikzpicture}[grow'=up,scale=.35]
\tikzstyle{level 1}=[sibling distance=7in]
\tikzstyle{level 2}=[sibling distance=2in]
\tikzstyle{level 3}=[sibling distance=2in]
\tikzstyle{level 4}=[sibling distance=2in]
\tikzstyle{level 5}=[sibling distance=1in]
\tikzstyle{level 6}=[sibling distance=1.8in]
\tikzstyle{level 7}=[sibling distance=.8in]
\tikzstyle{level 8}=[sibling distance=.8in]
\tikzstyle{level 9}=[sibling distance=.8in]
\tikzstyle{level 10}=[sibling distance=1.2in]
\tikzstyle{level 11}=[sibling distance=.3in]
\tikzstyle{level 12}=[sibling distance=.3in]
\tikzstyle{level 13}=[sibling distance=.3in]
\tikzstyle{level 14}=[sibling distance=.5in]
\node {$d_0$}
child{coordinate (t0)
			child{coordinate (t00)
child{coordinate (t000)
child{coordinate (t0000)
child{coordinate (t00000)
child{coordinate (t000000)
child{coordinate (t0000000)
child{coordinate (t00000000)
child{coordinate (t000000000)
child{coordinate (t0000000000)
child{coordinate (t00000000000)
child{coordinate (t000000000000)
child{coordinate (t0000000000000)
child{coordinate (t00000000000000)}
child{edge from parent[draw=none] coordinate (t00000000000001)}
}
child{coordinate (t0000000000001)
child{edge from parent[draw=none]  coordinate (t00000000000010)}
child{coordinate (t00000000000011)}
}
}
child{edge from parent[draw=none] coordinate (t000000000001)}
}
child{edge from parent[draw=none]  coordinate (t00000000001)}
}
child{edge from parent[draw=none]  coordinate (t0000000001)}
}
child{coordinate (t000000001)
child{edge from parent[draw=none]  coordinate (t0000000010)}
child{coordinate (t0000000011)
child{coordinate (t00000000110)
child{coordinate (t000000001100)
child{coordinate (t0000000011000)
child{coordinate (t00000000110000)}
child{edge from parent[draw=none] coordinate (t00000000110001)}
}
child{edge from parent[draw=none] coordinate (t0000000011001)}
}
child{edge from parent[draw=none]  coordinate (t000000001101)}
}
child{edge from parent[draw=none]   coordinate (t00000000111)}
}
}
}
child{edge from parent[draw=none] coordinate (t00000001)}
}
child{edge from parent[draw=none] coordinate (t0000001)}
}
child{edge from parent[draw=none]  coordinate (t000001)}
}
child{coordinate (t00001)
child{edge from parent[draw=none]  coordinate (t000010)}
child{coordinate (t000011)
child{coordinate (t0000110)
child{coordinate (t00001100)
child{coordinate (t000011000)
child{coordinate (t0000110000)
child{coordinate (t00001100000)
child{coordinate (t000011000000)
child{coordinate (t0000110000000)
child{coordinate (t00001100000000)}
child{edge from parent[draw=none] coordinate (t00001100000001)}
}
child{edge from parent[draw=none] coordinate (t0000110000001)}
}
child{coordinate (t000011000001)
child{coordinate (t0000110000010)
child{edge from parent[draw=none]  coordinate (t00001100000100)}
child{coordinate (t00001100000101)}
}
child{edge from parent[draw=none] coordinate (t0000110000011)}
}
}
child{edge from parent[draw=none] coordinate (t00001100001)}
}
child{edge from parent[draw=none]  coordinate (t0000110001)}
}
child{edge from parent[draw=none] coordinate (t000011001)}
}
child{edge from parent[draw=none]  coordinate (t00001101)}
}
child{edge from parent[draw=none]  coordinate (t0000111)}
}
}
}
child{edge from parent[draw=none] coordinate (t0001)}
}
child{ edge from parent[draw=none]coordinate(t001)}}
			child{ coordinate(t01)
child{ edge from parent[draw=none] coordinate(t010)}
child{coordinate(t011)
child{coordinate(t0110)
child{coordinate(t01100)
child{coordinate(t011000)
child{coordinate(t0110000)
child{coordinate(t01100000)
child{coordinate(t011000000)
child{coordinate(t0110000000)
child{coordinate(t01100000000)
child{coordinate(t011000000000)
child{coordinate(t0110000000000)
child{coordinate(t01100000000000)}
child{edge from parent[draw=none] coordinate(t01100000000001)}
}
child{ edge from parent[draw=none]coordinate(t0110000000001)}
}
child{ edge from parent[draw=none]  coordinate(t011000000001)}
}
child{coordinate(t01100000001)
child{coordinate(t011000000010)
child{coordinate(t0110000000100)
child{edge from parent[draw=none] coordinate(t01100000001000)}
child{coordinate(t01100000001001)}
}
child{edge from parent[draw=none] coordinate(t0110000000101)}
}
child{edge from parent[draw=none]   coordinate(t011000000011)}
}
}
child{ edge from parent[draw=none]  coordinate(t0110000001)}
}
child{ edge from parent[draw=none] coordinate(t011000001)}
}
child{coordinate(t01100001)
child{coordinate(t011000010)
child{edge from parent[draw=none] coordinate(t0110000100)}
child{coordinate(t0110000101)
child{coordinate(t01100001010)
child{coordinate(t011000010100)
child{coordinate(t0110000101000)
child{coordinate(t01100001010000)}
child{edge from parent[draw=none] coordinate(t01100001010001)}
}
child{edge from parent[draw=none]  coordinate(t0110000101001)}
}
child{edge from parent[draw=none]  coordinate(t011000010101)}
}
child{edge from parent[draw=none] coordinate(t01100001011)}
}
}
child{edge from parent[draw=none]  coordinate(t011000011)}
}
}
child{ edge from parent[draw=none] coordinate(t0110001)}
}
child{ edge from parent[draw=none]  coordinate(t011001)}
}
child{ edge from parent[draw=none]  coordinate(t01101)}
}
child{ edge from parent[draw=none]  coordinate(t0111)}
}}}
		child{ coordinate(t1)
			child{ coordinate(t10)
child{ coordinate(t100)
child{ coordinate(t1000)
child{ coordinate(t10000)
child{ coordinate(t100000)
child{ coordinate(t1000000)
child{ coordinate(t10000000)
child{ coordinate(t100000000)
child{ coordinate(t1000000000)
child{ coordinate(t10000000000)
child{ coordinate(t100000000000)
child{ coordinate(t1000000000000)
child{ coordinate(t10000000000000)}
child{ edge from parent[draw=none]   coordinate(t10000000000001)}
}
child{edge from parent[draw=none]   coordinate(t1000000000001)}
}
child{ edge from parent[draw=none]   coordinate(t100000000001)}
}
child{edge from parent[draw=none]  coordinate(t10000000001)}
}
child{ edge from parent[draw=none]  coordinate(t1000000001)}
}
child{edge from parent[draw=none]  coordinate(t100000001)}
}
child{   edge from parent[draw=none] coordinate(t10000001)}
}
child{ coordinate(t1000001)
child{ coordinate(t10000010)
child{ coordinate(t100000100)
child{ edge from parent[draw=none]  coordinate(t1000001000)}
child{ coordinate(t1000001001)
child{ coordinate(t10000010010)
child{ coordinate(t100000100100)
child{ coordinate(t1000001001000)
child{ coordinate(t10000010010000)}
child{ edge from parent[draw=none] coordinate(t10000010010001)}
}
child{edge from parent[draw=none]   coordinate(t1000001001001)}
}
child{edge from parent[draw=none]  coordinate(t100000100101)}
}
child{ edge from parent[draw=none] coordinate(t10000010011)}
}
}
child{  edge from parent[draw=none] coordinate(t100000101)}
}
child{   edge from parent[draw=none] coordinate(t10000011)}
}
}
child{  edge from parent[draw=none] coordinate(t100001)}
}
child{ edge from parent[draw=none]  coordinate(t10001)}
}
child{ coordinate(t1001)
child{ coordinate(t10010)
child{ edge from parent[draw=none]   coordinate(t100100)}
child{ coordinate(t100101)
child{ coordinate(t1001010)
child{ coordinate(t10010100)
child{ coordinate(t100101000)
child{ coordinate(t1001010000)
child{ coordinate(t10010100000)
child{ coordinate(t100101000000)
child{ coordinate(t1001010000000)
child{ coordinate(t10010100000000)}
child{edge from parent[draw=none] coordinate(t10010100000001)}
}
child{ edge from parent[draw=none]coordinate(t1001010000001)}
}
child{edge from parent[draw=none]  coordinate(t100101000001)}
}
child{edge from parent[draw=none]  coordinate(t10010100001)}
}
child{ edge from parent[draw=none]  coordinate(t1001010001)}
}
child{ edge from parent[draw=none] coordinate(t100101001)}
}
child{edge from parent[draw=none]   coordinate(t10010101)}
}
child{ edge from parent[draw=none]   coordinate(t1001011)}
}
}
child{  edge from parent[draw=none] coordinate(t10011)}
}
}
child{edge from parent[draw=none]  coordinate(t101)}}
child{edge from parent[draw=none]   coordinate(t11)
} };

\node[left] at (t0) {$d_1$};
\node[right] at (t10) {$c_0$};
\node[left] at (t10) {$d_2$};
\node[left] at (t100) {$d_3$};
\node[left] at (t0000) {$d_4$};
\node[left] at (t01100) {$d_5$};
\node[right] at (t01100) {$c_1$};
\node[left] at (t100000) {$d_6$};
\node[left] at (t0110000) {$d_7$};
\node[left] at (t00000000) {$d_8$};
\node[right] at (t000011000) {$c_2$};
\node[left] at (t000011000) {$d_9$};

\node[left] at (t0110000000) {$d_{10}$};
\node[left] at (t00001100000) {$d_{11}$};
\node[left] at (t000000000000) {$d_{12}$};
\node[left] at (t1000001001000) {$d_{13}$};
\node[right] at (t1000001001000)  {$c_3$};

\node[circle, fill=black,inner sep=0pt, minimum size=5pt] at (t10) {};
\node[circle, fill=black,inner sep=0pt, minimum size=5pt] at (t01100) {};
\node[circle, fill=black,inner sep=0pt, minimum size=5pt] at (t000011000) {};
\node[circle, fill=black,inner sep=0pt, minimum size=5pt] at (t1000001001000) {};

\draw[thick, dotted] let \p1=(t10) in (-30,\y1) node (v0) {$\bullet$} -- (10,\y1);
\draw[thick, dotted] let \p1=(t01100) in (-30,\y1) node (v1) {$\bullet$} -- (10,\y1);
\draw[thick, dotted] let \p1= (t000011000) in (-30,\y1) node (v2) {$\bullet$} -- (10,\y1);
\draw[thick, dotted] let \p1=  (t1000001001000) in (-30,\y1) node (v3) {$\bullet$} -- (10,\y1);

\node[left] at (v0) {$v_0$};
\node[left] at (v1) {$v_1$};
\node[left] at (v2) {$v_2$};
\node[left] at (v3) {$v_3$};

\draw[thick] (v0.center) to (v1.center) to (v2.center) to (v3.center) to [bend left] (v0.center);

\end{tikzpicture}
\caption{A strong coding tree $\bT$}
\end{figure}

Recall that $\lgl d_m:m<\om\rgl$ enumerates the set of all critical nodes (coding nodes and splitting nodes) in $T$ in order of strictly increasing length.

\begin{defn}[Finite strong coding  tree]
Given a strong coding tree $T$,
by an {\em initial segment} or {\em initial subtree of $T$}
we mean the first $m$ levels of $T$, for some $m<\om$.
We shall use the notation
\begin{equation}
r_m(T)=\bigcup_{k<m}\Lev_T(k).
\end{equation}
A tree with coding nodes is a  {\em finite strong coding  tree} if and only if it is
equal to some $r_{m+1}(T)$  where either $d_m$ is a coding node or else $m=0$.
\end{defn}

Thus, finite strong coding trees  are exactly the finite trees  with coding nodes $\lgl c_n:n<N\rgl$, where $N<\om$, which have all  maximal nodes of the  length of its  longest coding node  and satisfy  (2) - (7) of Definition \ref{def.diagtreeH_3} for all  $n<N$.

The next lemma extends the ideas of
 Lemma \ref{lem.stftextension} to the setting of finite strong coding trees.

\begin{lem}\label{lem.finitesctA_3}
Let $A$ be any finite strong coding tree with coding nodes $\lgl c_n:n<N\rgl$, where $N<\om$.
Let $A^+$ denote the  nodes of length $l_{N-1}+1$ extending the maximal nodes in $A$ as determined by (6) and (7) in Definition \ref{def.diagtreeH_3}.
Then given any $F\sse N$ such that $\{c_n:n\in F\}$ codes no edges, there is a  $t\in A^+$ such that
for all $n<N$,
\begin{equation}\label{eq.finitesctA_3}
t(l_n)=1\ \ \longleftrightarrow\ \ n\in F.
\end{equation}
\end{lem}

\begin{proof}
The proof is by induction on $N$ over all finite  strong  strong coding trees with $N$ coding nodes.
For $N=0$, $A=\emptyset$, the lemma vacuously holds.
For $N=1$,
it follows from the definition of finite strong coding tree that
$A$ has critical nodes $d_0=\stem(A)$, $d_1$ which is a splitting node extending ${d_0}^{\frown}0$,
and $d_2=c_0$ which extends ${d_0}^{\frown}1$.
Thus, $A^+$ has three nodes, $t_0\supset {d_0}^{\frown}0$ with passing number $0$ at $c_0$;
$t_1\supset {d_0}^{\frown}1$ with passing number $1$ at $c_0$;
and $t_2={c_0}^{\frown}0$ which of course has passing number $0$ at $c_0$.
Both of the nodes $t_0$ and $t_2$ satisfy equation (\ref{eq.finitesctA_3}) if $F=\emptyset$,
and $t_1$ satisfies (\ref{eq.finitesctA_3}) if $F=\{0\}$.

Now suppose that $N\ge 2$ and  the lemma holds for  $N-1$.
Let $A$ be a finite strong coding tree with coding nodes $\lgl c_n:n<N\rgl$.
Let $F$ be a subset of $N$ such that $\{c_n:n\in F\}$ codes no edges,
and let $m$ be the index such that $d_{m-1}=c_{N-2}$.
By the induction hypothesis, there is a node $u$ in
$(r_m(A))^+$ such that
for all $n<N-1$, $u(l_n)=1$ if and only if $n\in F$.
If  $N-1\not\in F$,
by (6) and (7) of the definition of strong coding tree
there is an extension $t\supset u$ in $A^+$ with passing number $0$ at $c_{N-1}$,
and this $t$ satisfies (\ref{eq.finitesctA_3}) for $F$.

If  $N-1\in F$,
it suffices to show that $u\in\Spl(A,N-1)$,
for then there will be a $t\supset u$ in $A^+$ with passing number $1$ at $c_{N-1}$, and this $t$ will satisfy
(\ref{eq.finitesctA_3}).
By the Splitting Criterion for Skew Trees,
if $u\not\in\Spl(A,N-1)$,
then $u$ and $c_{N-1}\re (l_{N-2}+1)$ must have a parallel $1$.
Then by the \POC,
 there is some $i\le N-2$ such that $u(l_i)=c_{N-1}(l_i)=1$.
Since  $u$ codes edges only with those vertices with indexes
less than $N-1$  in $F$,
it follows that $i$ must be in $F$.
But then $\{c_i,c_{N-1}\}$ is a subset of $F$ coding an edge, contradicting the assumption on $F$.
Therefore,  $u$ is in $\Spl(A,N-1)$.
\end{proof}

We now present a flexible method for constructing a strong coding tree $\bT$.
This should be thought of as a stretched and skewed version of the strong triangle-free tree $\bS$ which was constructed in Theorem \ref{thm.stftree}.
The passing numbers at the coding  nodes in $\bT$ code edges and non-edges  exactly as the passing numbers of the coding nodes in $\bS$.
The  strong coding tree  $\bT$ which  we  construct will be {\em regular}:
  For each $n$,
nodes in $\Spl(\bT,n)$ extend to splitting nodes in the $n$-th interval of $\bT$  from lexicographically  least to largest.
Regularity is not necessary for achieving the main theorems of this article.
 However,  as any strong coding tree  contains a subtree which is  a regular strong coding tree,
it does no harm
to  only work with regular trees.

\begin{thm}\label{thm.cool}
Let  $\lgl F_i:i<\om\rgl$  be any sequence enumerating  the finite subsets of $\om$ so that each finite set appears cofinally often.
Assume further that
for each $i<\om$,
 $F_i\sse  i-1$ and $F_{3i}=F_{3i+2}=\emptyset$.
Then there is a strong coding  tree $\bT$ which densely codes $\mathcal{H}_3$, where for each $i<\om$ and $j\le 2$, the coding node $c_{4i+j}$ meets requirement $F_{3i+j}$.
\end{thm}

\begin{proof}
Let  $\lgl F_i:i<\om\rgl$  satisfy the hypotheses, and
let $\lgl u_i: i<\om\rgl$ be an enumeration of all the nodes in $2^{<\om}$   in such a way that each  $|u_i|\le i$.
We  construct  a strong coding tree $\bT\sse 2^{<\om}$ with coding nodes $\lgl c_n:n<\om\rgl$  and lengths $l_n=|c_n|$
so that for each $n<\om$,
$r_{m_n+1}(\bT):=\bigcup\{\Lev_{\bT}(i):i\le m_ n\}$  is a finite strong coding tree
and $\Lev_{\bT}(m_n+1)$ satisfies (6) and (7) of the definition of strong coding tree,
where $m_n$ is the index such that the $m_n$-th critical node $d_{m_n}$ is equal to the $n$-th coding node $c_n$,
and the following properties are satisfied:

\begin{enumerate}
\item[(i)]
For $n=4i+j$,  $j\le 2$,  $c_{n}$ meets requirement $F_{3i+j}$.
\item[(ii)]
For $n=4i+3$, if $u_i$ is in
$r_{m_{n-3}+2}(\bT)$,
then $c_n$ is a coding node extending $u_i$.
Otherwise,
 $c_n={0^{l_{n-1}-1}}^{\frown}\lgl 1,1\rgl ^{\frown}{0^{q_n}}$
where $q_n=l_n-(l_{n-1}+1)$.
\end{enumerate}

To begin, define $\Lev_{\bT}(0)=\{\lgl\rgl\}$.
Then the minimum length splitting node in $\bT$ is $\lgl \rgl$, and we label it $d_0$.
Let $\Lev_{\bT}(1)=\{\lgl 0\rgl,\lgl 1\rgl\}$.
To satisfy  (5) of Definition \ref{def.diagtreeH_3},
 $c_0$ is going  to extend $\lgl 1\rgl$,
so  in order to satisfy (4), it must be the case that  $\Spl(\bT,0)=\{\lgl 0\rgl\}$.
Take the splitting node $d_1$ to be  $\lgl 0\rgl$.
Let  $\Lev_{\bT}(2)=\{\lgl 0,0\rgl,\lgl 0,1\rgl, \lgl 1,0\rgl\}$, and
define  $c_0=\lgl 1,0\rgl$.
Then  $l_0=2$,  $d_2=c_0$,
 and
\begin{equation}
r_{m_0+1}(\bT)=\bigcup\{\Lev_{\bT}(i):i\le 2\}
\end{equation}
is a finite strong coding tree satisfying (i) and (ii).
The  next level of $\bT$ must  satisfy
(6) and (7).
Extend
$\lgl 0,0\rgl$  to $\lgl 0,0,0\rgl$,  extend $\lgl 0,1\rgl$  to $\lgl 0,1,1\rgl$, and
extend $\lgl 1,0\rgl$ to $\lgl 1,0,0\rgl$, and let these compose $\Lev_{\bT}(3)$.

For the sake of clarity,
the next few levels of $\bT$ up to the level of $c_1$ will be constructed concretely.
To satisfy (2),
the next coding node $c_1$  must  extend $\lgl 0,1,1\rgl$,
 since this is the only node in $\Lev_{\bT}(3)$ which has passing number $1$ at $c_0$.
The knowledge that $c_1$ will extend  $\lgl 0,1,1\rgl$  along with the Splitting Criterion for Skew Trees
determine that
  $\Spl(\bT,1)=\{\lgl 0,0,0\rgl, \lgl 1,0,0\rgl\}$,
 since these are the nodes in
$\Lev_{\bT}(3)$
which have no parallel $1$'s with
$\lgl 0,1,1\rgl$.
As we are building $\bT$ to be regular,
 $\lgl 1,0,0\rgl$ is  first in  $\Spl(\bT,1)$
 to be extended to a splitting node.
Let $d_3=\lgl 1,0,0\rgl$,
and let  $\Lev_{\bT}(4)=\{\lgl 0,0,0,0\rgl, \lgl 0,1,1,0\rgl,\lgl 1,0,0,0\rgl,\lgl 1,0,0,1\rgl\}$, so that $\bT_4$ is strongly skew.
Next, let $d_4=\lgl 0,0,0,0\rgl$ as this node should split since it is the only extension of $\lgl 0,0,0\rgl$ in $\Lev_{\bT}(4)$.
Let
\begin{equation}
\Lev_{\bT}(5)=
\{\lgl 0,0,0,0,0\rgl, \lgl 0,0,0,0,1\rgl,\lgl 0,1,1,0,0\rgl,\lgl 1,0,0,0,0\rgl,\lgl 1,0,0,1,0\rgl\}.
\end{equation}
Let $c_1=\lgl 0,1,1,0,0\rgl$, as this is
the only extension of $\lgl 0,1,1\rgl$ in  $\Lev_{\bT}(5)$.
Thus, $d_5=c_1$,
$l_1=5$, $\spl_{\bT}(\lgl 1,0,0\rgl)=\lgl 1,0,0\rgl$
and $\spl_{\bT}(\lgl 0,0,0\rgl)=\lgl 0,0,0,0\rgl$.
Moreover,
$r_6(\bT)$
 is  a regular,  finite strong coding tree satisfying requirements  (i) - (ii).
The next level of $\bT$  is determined by (6) and (7), so
let
\begin{equation}
\Lev_{\bT}(6) =\{\lgl 0,0,0,0,0,0\rgl, \lgl 0,0,0,0,1,1\rgl,\lgl 0,1,1,0,0,0\rgl,\lgl 1,0,0,0,0,0\rgl,\lgl 1,0,0,1,0,1\rgl\}.
\end{equation}
This  constructs the tree  $r_7(\bT)$,
which is $\bT$ up to the level of $l_1+1=6$.
Notice that the second lexicographically least node in $\Lev_{\bT}(l_1+1)$ is
$\lgl 0,0,0,0,1,1\rgl={0^{(l_{1}-1)}}^{\frown}\lgl 1,1\rgl$.

Suppose $r_{m_n-1}+2(\bT)$ has been constructed
 so that
$r_{m_{n-1}+1}(\bT)$
 is a finite   strong coding tree
 satisfying (i)  and  (ii)
and such that $\Lev_{\bT}(m_{n-1}+1)$ satisfies (6) and (7)  of Definition \ref{def.diagtreeH_3},
where $m_{n-1}$ is the index such that $d_{m_{n-1}}=c_{n-1}$.
As part of the induction hypothesis, suppose also that the second lexicographically least node in $\Lev_{\bT}(m_{n-1}+1)$ is ${0^{(m_{n-1}-1)}}^{\frown}\lgl 1,1\rgl$,
this being true in the base case of $r_{m_1+2}(\bT)$.
Enumerate the members of $\Lev_{\bT}(m_{n-1}+1)$  in decreasing lexicographical order as $\lgl s_k:k<K\rgl$.
At this stage, we need to know which node $s_k$ will be extended to the next coding node $c_n$ as this determines the set $\Spl(\bT,n)$.
We will show how to choose $k_*$ in the three cases below, so that  extending $s_{k_*}$ to $c_n$   will meet
requirements (i)  and  (ii).
Once $k_*$ is chosen,
 $\Spl(\bT,n)$  is  the set  $\{s_k:k\in K_{sp}\}$,
where
 $K_{sp}$ is the set of those $k<K$ such that
for all $i<n$, $s_k(l_i)+s_{k_*}(l_i)\le 1$,
that is, $s_k$ and $s_{k_*}$ have no parallel $1$'s at or  below $l_{n-1}$.
Then let $c_n={s_{k_n^*}}^{\frown}0^{|K_{sp}|}$,
and extend all nodes in $\{s_k:k<K\}$ according to (6) and (7) in the definition of strong coding tree.
We point out that  $l_n$ will equal $l_{n-1}+|K_{sp}|+1$.

There are three cases  to consider regarding which $k<K$ should be $k_*$.

\it Case 1. \rm $n=4i$ or $n=4i+2$ for some $i<\om$.
Let $n'$ denote $3i$ if $n=4i$ and  $3i+2$ if $n=4i+2$.
In this case, $F_{n'}=\emptyset$.
Let $k_*=K-2$.
Since $s_{K-1}$ is the lexicographic least member of $\Lev_{\bT}(m_{n-1}+1)$, $s_{K-1}$ must be $0^{l_{n-1}+1}$.
Hence,
$s_{K-2}$ being  next lexicographic largest implies that $s_{K-2}={0^{(l_{n-1}-1)}}^{\frown}\lgl 1,1\rgl$.
Let $k_*=K-2$.
Then any extension of $s_{k_*}$ to a coding  node will have passing number 1 at  $c_{n-1}$ and  passing number $0$ at  $c_i$ for all  $i<n-1$.

\it Case 2. \rm $n=4i+1$ for some $1\le i  <\om$.
If there is a pair $k<m$ of integers in $F_{3i+1}$ such that $c_m(l_k)=1$,
then again take $k_*$ to be $K-2$.
Otherwise,
 $c_m(l_k)=0$
for all pairs $k<m$ in $F_{3i+1}$.
Note that $i\ge 1 $ implies that
$\max(F_{3i+1})\le 3i-1\le n-3$.
Since by the induction hypothesis $r_{m_{n-2}+1}(\bT)$ is a finite strong coding tree,
Lemma \ref{lem.finitesctA_3}
implies
there is some $t\in \Lev_{\bT}(m_{n-3}+2)$
such that $t(l_j)=1$ if and only if $j\in F_{3i+1}$.
Let $t'$ be the node in  $2^{<\om}$ of length
$l_{n-2}+1$ which extends $t$ by all $0$'s.
By our construction, this node is in $r_{m_{n-2}+2}(\bT)$.
Since,  by Case 1,
$c_{n-1}$ is the node of length $l_{n-1}$ extending ${0^{l_{n-2}-1}}^{\frown}\lgl 1,1\rgl$ by all $0$'s, one sees that
$t'\re (l_{n-2}+1)$ and $c_{n-1}\re(l_{n-2}+1)$ have no parallel $1$'s.
Thus, $t'\re (l_{n-2}+1)$  is in $\Spl(\bT, n-1)$.
Let $k_*$ be the index in $K$ such that
$s_{k_*}$ is
 the rightmost extension of $t'$
in $\Lev_{\bT}(m_{n-1}+1)$.

\it Case 3. \rm $n=4i+3$ for some $i<\om$.
If $u_i\not\in   r_{m_{n-3}+2}(\bT)$,
then let $k_*=K-2$.
Otherwise, $u_i\in   r_{m_{n-3}+2}(\bT)$.
Let $u'$ be the leftmost extension of
 $u_i$  in $r_{m_{n-2}+2}(\bT)$
of  length $l_{n-2}+1$.
In particular, $u'(l_{n-2}-1)=u'(l_{n-2})=0$.
As in Case 2,
$c_{n-1}$ is the node of length $l_{n-1}$
such that for all $l<l_{n-1}$,
$c_{n-1}(l)=1$ if and only if $l\in\{l_{n-2}-1,l_{n-2}\}$.
Thus, $u'$ and $c_{n-1}\re (l_{n-2}+1)$ have no parallel $1$'s,
so by the induction hypothesis,
$u'\in \Spl(\bT, n-1)$.
Hence, there is an extension $u''\contains u'$ in $r_{m_{n-1}+2}(\bT)$ such that
$u''(l_{n-1})=1$.
Let $k_*$ be the index of the  node $u''$.

To finish the construction of $\bT$ up to level $l_n+1$,
let   $l_n=l_{n-1}+|K_{sp}|+1$.
For each $k\not\in K_{sp}$,
extend $s_k$ via all $0$'s to length $l_n+1$.
Note in each of the three cases, $k_*$ is not in $ K_{sp}$,
since $s_{k_*}$ has passing number $1$ at $c_{n-1}$.
Thus, $c_n$ is  the extension of $s_{k_*}$ by all $0$'s to length $l_n$, and its immediate extension, or passing number by itself, is $0$.
Enumerate $K_{sp}$  as $\lgl k_i:i<|K_{sp}|\rgl$   so that
$s_{k_i}>_{\mathrm{lex}}s_{k_{i+1}}$ for each $i$.
Let $\spl(s_{k_i})={s_{k_i}}^{\frown}0^{i}$;  in particular,
$\spl(s_{k_0})=s_{k_0}$.
For each $i<|K_{sp}|$,
 letting $p_i=|K_{sp}|-i$,
 ${s_{k_i}}^{\frown}0^{|K_{sp}|}$  and
${\spl(s_{k_i})}^{\frown}1^{\frown} {0^{p_i-2}}^{\frown}1$
are the two extensions of $s_{k_i}$ in
$\Lev_{\bT}(l_n+1)$.
This constructs $\Lev_{\bT}(l_n+1)$.
Notice that for each $j<2$, the $t\in \Lev_{\bT}(l_n+1)$ extending $\spl(s_{k_i})^{\frown}j$ has passing number $t(l_n)=j$.

Let  $\bT=\bigcup_{i<\om}\Lev_{\bT}(i)$.
Then $\bT$ is a strong coding tree because each initial segment  $r_{m_n+1}(\bT)$, $n<\om$,  is a  finite strong coding tree,
and the coding nodes are dense in $\bT$.
\end{proof}


\begin{fact}\label{fact.perfect}
Any   strong coding tree
is a perfect tree.
\end{fact}

\begin{proof}
Let  $t$ be   any node in $T$, and
and let $j$ be minimal such that $l_j\ge |t|$.
Extend $t$ leftmost in $T$ to the node of length $l_j$, and label this $t'$.
Let $s=0^{l_j}$.
By density of coding nodes in $T$, there is a coding node  $c_k$ in  $T$ extending $s$, with $k\ge j+2$.
Extending $t'$ leftmost in $T$ to length $l_{k-1}+1$ produces a node $t''$ in $\widehat{T}$ which has no parallel $1$'s with $c_k\re (l_{k-1}+1)$.
Thus, $t''\in\Spl(T,k)$, so $t''$ extends to a splitting node in $T$ before reaching the level of $c_k$.
\end{proof}

In  particular,
it follows from the definition of strong coding tree that   in any strong coding tree $T$,
for any $n<\om$,
the node $0^{l_{n-1}}$
will split  in $T$ before the level $l_n$.


\subsection{The space $(\mathcal{T}(T),\le,r)$ of strong coding trees}\label{subsection4.2}

The space of  subtrees of a given strong coding tree, equipped with a strong partial ordering, will
form the fundamental structure allowing for the Ramsey theorems in latter sections.
It turns out that not every subtree of a given strong coding tree $T$ can be extended within $T$ to form another strong coding tree.
The notion of valid subtree provides conditions when a finite subtree can be extended in any desired manner within $T$.
Some lemmas guaranteeing that  finite
valid subtrees
of a given strong coding tree $T$ can be extended
 to any desired configuration within $T$
 are presented at the end of this subsection.
These lemmas will be very useful in  subsequent sections.
Those familiar with topological Ramsey spaces
will notice  the   influence of \cite{TodorcevicBK10}  in our chosen style of presentation,
the idea being that  the space of strong coding trees
has a similar character  to the topological Ramsey space of Milliken's infinite strong trees,
though
background in \cite{TodorcevicBK10} is not  necessary for understanding this article.

To begin, we define a strong notion of isomorphism between meet-closed sets  by  augmenting
Sauer's
 notion of strong similarity type
from   \cite{Sauer06} to fit the present setting.
Given a subset $S\sse 2^{<\om}$,
recall that
 the {\em meet closure of $S$}, denoted $S^{\wedge}$,
is the set of all meets of pairs of nodes in $S$.
In this definition $s$ and $t$ may be equal, so  $S^{\wedge}$ contains $S$.
We say that  $S$ is {\em meet-closed} if $S=S^{\wedge}$.
Note that each tree  is meet-closed, but there are meet-closed sets which are not trees,
as
 Definition \ref{defn.tree} of tree applies throughout this paper.

\begin{defn}[\cite{Sauer06}]\label{def.2.2.Sauer}
$S\sse 2^{<\om}$ is an {\em antichain} if $s\sse t$ implies $s=t$, for all $s,t\in S$.
A set $S\sse 2^{<\om}$ is {\em transversal} if $|s|=|t|$ implies $s=t$ for all $s,t\in S$.
A set $D\sse 2^{<\om}$ is {\em diagonal} if $D$  is an antichain with  $D^{\wedge}$ being transversal.
A diagonal set
$D$ is {\em strongly diagonal} if additionally  for any $s,t,u\in D$ with $s\ne t$, if $|s\wedge t|<|u|$ and $s\wedge t\not\subset u$, then $u(|s\wedge t|)=0$.
\end{defn}

It follows that the meet closure of any antichain of coding nodes in a strong coding tree is  strongly diagonal.
In fact, strong coding trees were designed with this property in mind.

The following augments Sauer's Definition 3.1 in \cite{Sauer06}  to the setting of  trees with coding  nodes.
The lexicographic order on $2^{<\om}$
between two nodes
$s,t\in 2^{<\om}$, with neither extending the other,
is defined by
$s<_{\mathrm{lex}} t$ if and only if $s\contains (s\wedge t)^{\frown}0$ and $t\contains (s\wedge t)^{\frown}1$.
It is important to note that in a given strong coding tree $T$, each node $s$ at the level of a coding node $c_n$ in $T$
has exactly one  immediate extension in $\widehat{T}$.
This is the unique node $s^+$ of length $l_n+1$ in $\widehat{T}$ such that $s^+\supset s$.
This fact is used in (7) of the following definition.

\begin{defn}\label{def.3.1.likeSauer}
Let $S,T\sse 2^{<\om}$ be meet-closed subsets of  a fixed strong coding tree $\bT$.
The function $f:S\ra T$ is a {\em strong similarity} of $S$ to $T$ if for all nodes $s,t,u,v\in S$, the following hold:
\begin{enumerate}
\item
$f$ is a bijection.
\item
$f$ preserves lexicographic order: $s<_{\mathrm{lex}}t$ if and only if $f(s)<_{\mathrm{lex}}f(t)$.
\item
$f$ preserves initial segments:
$s\wedge t\sse u\wedge v$ if and only if $f(s)\wedge f(t)\sse f(u)\wedge f(v)$.

\item
$f$ preserves meets:
$f(s\wedge t)=f(s)\wedge f(t)$.

\item
$f$ preserves relative lengths:
$|s\wedge t|<|u\wedge v|$ if and only if
$|f(s)\wedge f(t)|<|f(u)\wedge f(v)|$.

\item $f$ preserves coding  nodes:
$f$ maps the set of  coding nodes in $S$
onto the set of coding nodes in $T$.
\item

$f$ preserves passing numbers at coding nodes:
If $c$ is a coding node in $S$ and $u$ is a node in $S$ with $|u|\ge|c|$,
then $(f(u))^+(|f(c)|)=u^+(|c|)$;
in words, the passing number of the immediate successor of $f(u)$ at $f(c)$  equals the passing number of the immediate successor of $u$ at $c$.
\end{enumerate}
\end{defn}

In all cases above,   it may be that $s=t$ and $u=v$
 so that (3) implies $s\sse u$ if and only if $f(s)\sse f(u)$, etc.
It follows from  (4) that
 $s\in S$ is a splitting node in $S$ if and only if  $f(s)$ is a splitting node in $T$.
We say that $S$ and $T$ are {\em strongly similar} if there is a strong similarity of $S$ to $T$, and in this case write $S\ssim T$.
If $T'\sse T$ and $f$ is a strong similarity of $S$ to $T'$, then $f$ is a {\em strong similarity embedding} of $S$ into $T$, and $T'$ is a {\em strong similarity copy} of $S$ in $T$.
For $A\sse T$, let  $\Sims^s_T(A)$ denote the set of all subsets of $T$ which are strongly similar to $A$.
The notion of strong similarity is relevant for all meet-closed subsets of a strong coding tree, including subsets which form trees.
Note that if $A$ is a meet-closed set which is not a tree and $S=\{u\re|v|:u,v\in A$ and $|u|\ge|v|\}$ is its induced tree,
technically  $A$ and $S$ are not strongly similar.
This distinction will present no difficulties.

Not only are strong coding trees perfect, but the  ones constructed in the manner of Theorem \ref{thm.cool}, and hence any tree with the same strong similarity type,
also have the following useful property.

\begin{fact}\label{fact.everytwointervalssplit}
Let $\bT$ be constructed in the manner of Theorem \ref{thm.cool}, and let $T$ be a strong coding tree which is strongly similar to $T$.
Then for each even integer $n<\om$,
each node in $T$  of length $l_n$ splits in $T$ before the level of $c_{n+2}$.
\end{fact}

\begin{proof}
Given a node $t$ in $T$ at the level of $c_n$,
if $t$ does not already split
 before the level of $c_{n+1}$,
then  its only extension to length $l_{n+1}+1$  has passing number $0$ at $c_{n+1}$;
call this extension $t'$.
Now
since $n+2$ is even,
 the coding node $c_{n+2}$ has passing number $0$
at all $c_i$, $i<n+1$, and passing number $1$ at $c_{n+1}$.
Thus, $t'$ and $c_{n+2}\re (l_{n+1}+1)$ have no parallel $1$'s, so $t'$ splits before reaching the level of $c_{n+2}$.
\end{proof}

Depending on how a finite subtree $A$
of  a strong coding tree $T$
sits inside $T$, it may be impossible to extend $A$ inside of $T$ to another strong coding tree.
As a simple example, the set of nodes $A=\{\lgl \rgl,
\lgl 0,0,0,0\rgl,
\lgl 1,0,0,1\rgl\}$ in $\bT$
is strongly similar to  $r_2(\bT)$.
However
  $A$  cannot be  extended in $\bT$ to a strong coding tree strongly similar to $\bT$ with  $\lgl 0,0,0,0\rgl$ being a splitting node.
The reasons are as follows.
  Any  such  extension  $A'$
  in $\bT$ must  have nodes extending $\lgl 0,0,0,0,0,0\rgl$,
 $\lgl 0,0,0,0,1,1\rgl$, and $\lgl 1,0,0,1,0,1\rgl\}$.
The nodes
 $\lgl 0,0,0,0,1,1\rgl$ and $\lgl 1,0,0,1,0,1\rgl$ have  parallel $1$'s,
so the next coding node must witness them.
In order to be strongly similar to $r_3(\bT)$,
$\lgl 0,0,0,0,1,1\rgl$ must be extended to the next coding node in $A'$,
and by the Triangle-Free Criterion, any such node is immediately succeeded by a $0$, so it cannot witness the new parallel $1$'s, thus failing to satisfy the \POC.

Another potential problem is the following.
Let $T$ be a strong coding tree and take $m$ such that $d^T_m$ is a splitting node,
$d^T_{m+2}=c^T_n$ is a coding node,
and $|d^T_{m-2}|>l_{n-1}$,
 where $n\ge 3$.
So, $d^T_m$ is a splitting node with at least two splitting nodes  preceding it in $T$ and at least one splitting node proceeding it before the next coding node in $T$.
It follows by the structure of strong coding trees that
there are at least two maximal nodes in $r_{m+1}(T)$
which have no parallel $1$'s but
 which are pre-determined  to passing $c^T_n$ with passing number $1$,
as their only extensions of length $l_n+1$ in $\widehat{T}$ both have passing number $1$ at $c_n^T$.
It follows  that any  strong coding subtree
 $S$ of $T$
with the same initial segment as $T$ up to
level $m$,  i.e.\
 $r_{m+1}(S)=r_{m+1}(T)$,
 is  necessarily  going to
have $r_{m+2}(S)=r_{m+2}(T)$;
for if
the  splitting node $d^S_{m+1}$ is not equal to $d^T_{m+1}$, then
 the pre-determined new parallel $1$'s appear
 in  $r_{m+2}(S)$ before the splitting node $d^S_{m+1}$,  implying $S$ violates the \POC.
Thus,
if $r_{m+2}(S)$ is a finite strong coding tree end-extending $r_{m+1}(T)$ into $T$ and strongly similar to $r_{m+2}(T)$,
then $r_{m+2}(S)$ must actually equal $r_{m+2}(T)$.
Clearly this is not what we want.

\begin{defn}\label{defn.nopredetll1}
Let $X=\{x_i:i<\tilde{i}\}$ be a level set   of two or more nodes in $\widehat{T}$, and let $l$ be their length.
We say that $X$ has {\em no pre-determined new sets of  parallel $1$'s in $T$}
if  either $X$ contains a coding node, or else
for any $l_n>l$,
there are extensions $y_i\contains x_i$ of length   $l_n$
such that the following holds:
For each $I\sse \tilde{i}$ of size at least two,
if
 there is an $l'< l_n$ such that $y_i(l')=1$ for all $i\in I$,
then there is an $l''< l$ such that $y_i(l'')=1$ for all $i\in I$.
\end{defn}

It in order to determine whether  a level set of nodes $X=\{x_i:i<\tilde{i}\}$ of length $l$, not containing a coding node,  has pre-determined new sets of  parallel $1$'s in $T$,
it
suffices to extend the nodes in $X$
 leftmost in $\widehat{T}$ to nodes $y_i\contains x_i$ of   length $l_n+1$, where $c_n$ is the
minimal
 coding node in $T$ of length greater than $l$:
 $X$
 has no pre-determined new sets of  parallel $1$'s
 if and only if
 there is an $l'<l$ such that
$\{i<\tilde{i}:x_i(l')=1\}$ contains the set  $\{i<\tilde{i}:y_i(l_n)=1\}$.

\begin{defn}\label{defn.valid}
A subtree  $A$, finite or infinite, of a strong coding tree $T$ is {\em valid} in $T$ if
 each level set in $A$ has no pre-determined new sets of parallel $1$'s in $T$.
\end{defn}

The  point is that valid subtrees are safe to work with:
They can always be extended within the ambient strong coding tree to any desired strong similarity type.
This will be seen clearly in the lemmas at the end of the section.

We now come to the definition of the space of strong coding subtrees of a fixed strong coding tree.
Define the
 partial ordering  $\le$  on the collection of all strong coding trees  as follows:
For  strong coding trees $S$ and $T$,
\begin{equation}
 S\le T\ \  \Longleftrightarrow\ \  S \mathrm{\ is\ a\  valid\ subtree\ of\ }T\mathrm{\ and\ }  S\ssim T.
\end{equation}

\begin{defn}[The space $(\mathcal{T}(T),\le,r)$]\label{AT}
Let $T$ be any  strong coding  tree.
Define  $\mathcal{T}(T)$  to be the collection of all
strong coding trees $S$ such that $S\le T$.
As previously defined, for $m<\om$,
$r_m(S)$ denotes $\bigcup_{i<m}\Lev_S(m)$,
 the initial subtree of $S$ containing its first $m$ critical nodes.
The restriction map $r$ is formally a map from $\om\times \mathcal{T}(T)$ which on  input $(m,S)$ produces $r_m(S)$.
Let
$\mathcal{A}_m(T)$ denote $\{r_m(S):S\in\mathcal{T}(T)\}$, and let
$\mathcal{A}(T)=\bigcup_{m<\om}\mathcal{A}_m(T)$,
the collection of all {\em finite approximations} to members of $\mathcal{T}(T)$.

For $A\in\mathcal{A}_m(T)$ and
$S\in\mathcal{T}(T)$ with $A$ valid in $S$,
define
\begin{equation}
[A,S]=\{U\le S: r_m(U)=A\},
\end{equation}
and define
\begin{equation}
r_{m+1}[A,S]=\{B\in \mathcal{A}_{m+1}:
r_m(B)=A\mathrm{\ and \ }B\mathrm{\ is \ valid\ in\ }S\}.
\end{equation}
\end{defn}

Techniques  for building valid subtrees of a given strong coding tree  are now developed.
The next lemma provides a means  for
extending a
particular maximal  node $s$ in a finite
subtree $A$ of a strong coding tree $T$
to a particular extension $t$ in $T$,
and extending the rest of the maximal nodes in $A$ to the length of $t$, without introducing new sets of parallel $1$'s.
Let $\{s_i:i<\tilde{i}\}$ be some level set of nodes in a strong coding tree $T$.
We say that a level set of extensions $\{t_i:i<\tilde{i}\}$,
where each $t_i\contains s_i$,
 {\em  adds no new sets of parallel $1$'s over $\{s_i:i<\tilde{i}\}$}
if  whenever $l<|t_0|$ and the
 set
 $I_l:=\{i<\tilde{i}:t_i(l)=1\}$ has  cardinality at least $2$,
then there is an  $l'<|s_0|$ such that
$\{i<\tilde{i}:s_i(l')=0\}=I_l$.

\begin{lem}\label{lem.poc}
Suppose $T$ is a strong coding tree and $\{s_i:i<\tilde{i}\}$   is a set of two or more   nodes in $\widehat{T}$ of length $l_k+1$.
Let
 $n_*>k$,  let $l_*$ denote $l_{n_*}$,
and
 let $t_0$ be any extension of $s_0$ in $\widehat{T}$ of length $l_*+1$.
 For each $0<i<\tilde{i}$, let $t_i$ denote the leftmost extension of $s_i$ in $\widehat{T}$ of length $l_*+1$.
Then
the set $\{t_i:i<\tilde{i}\}$ adds no new sets of parallel $1$'s over $\{s_i:i<\tilde{i}\}$.
\end{lem}

\begin{proof}
Assume the hypotheses, and
suppose  that  there is some
 $l<l_*$ such that
the set $I_l=\{i<\tilde{i}:t_i(l)=1\}$ has at least two members.
Then by  the \POC,
there is an $n\le n_*$ such that
$t_i(l_n)=1$ for all $i<\tilde{i}$.
Since for each $0<i<\tilde{i}$,  $t_i$ is the leftmost extension of $s_i$,
by (6) and (7) in the definition of strong coding tree,
the passing number of $t_i$ at $l_j$ is $0$,
for all $k<j\le n_*$.
It follows that  any $n$ such that $c_n$ witnesses the parallel $1$'s in $\{t_i:i\in I_l\}$ must be less than or equal to $k$.
\end{proof}

In fact, any sets of parallel $1$'s from the set $\{t_i:i<\tilde{i}\}$ constructed in the preceding lemma occur at a level below $l$.

Given a set of nodes $S$ in a strong coding tree,
the {\em tree induced by $S$}
is  the set of nodes
$\{s\re |v|:s\in S,\ v\in S^{\wedge}\}$.
For a finite tree $A$, we shall use the notation $\max(A)$ in a slightly non-standard way.

\begin{notation}\label{notn.max}
Given a finite tree $A$,  $\max(A)$ denotes
 the set of terminal nodes in $A$ which have the maximal length of any node in $A$.
Thus,
\begin{equation}
\max(A)=\{t\in A:t=l_A\},
\end{equation}
where $l_A=\max\{|s|:s\in A\}$.
Note in particular that $\max(A)$ is a level set.
\end{notation}

The following lemma is immediate from finitely many applications of  Lemma \ref{lem.poc}, using the fact that maximal nodes of valid subtrees can be extended leftmost to any length without adding any new sets of parallel $1$'s.

\begin{lem}\label{lem.factssplit}
Let $A$ be a finite  valid subtree of any strong coding tree  $T$ and let $l$ be the  length of the nodes in $\max(A)$.
Let $\Spl(u)$ be any nonempty level  subset of $ \max(A)$, and
 let $Z$ be any subset of
 $\max(A)\setminus\Spl(u)$.
Then given any enumeration
$\{z_i:i<\tilde{i}\}$ of $\Spl(u)$ and
$l'\ge l$,
 there is an $l_*>l'$
 and extensions
$s_i^0,s_i^1\supset z_i$ for all $i<\tilde{i}$,
and $s_z\supset z$ for all $z\in Z$, each of length $l_*$,
such that, letting
\begin{equation}
X=\{s_i^j:s\in\Spl(u),\ j\in\{0,1\}\}\cup\{s_z:z\in Z\},
\end{equation}
and $B$ be the tree induced by $A\cup X$,
 the following hold:
\begin{enumerate}
\item
The splitting in $B$ above $A$ occurs in the order of the enumeration of $\Spl(u)$.
Thus, for $i<i'<\tilde{i}$,
$|s_i^0\wedge s_i^1|<|s_{i'}^0\wedge s_{i'}^1|$.
\item
$B$ has no new sets of parallel $1$'s over $A$.
\end{enumerate}
\end{lem}

\begin{conv}\label{conv.POC}
When working within a fixed strong coding tree $T$,
the passing numbers at coding nodes $c^T_n$ are completely determined by $T$.
Thus,   for a finite subset $A$ of  $T$ such that
 $l_A$ equals $l_n^T$ for some $n<\om$,
then
 saying  that  $A$  {\em  satisfies the \POC} implies that  the extension
$A\cup\{s^+:s\in\max(A)\}$ satisfies the \POC.
\end{conv}

Lemma \ref{lem.pnc} shows that given a valid subtree of a strong coding tree $T$,
any of its maximal nodes can be extended to a coding node $c_k^T$  in $T$ while the rest of the maximal nodes can be extended to length $l_k^T$ so that their passing numbers are anything desired,  subject only to the Triangle-Free Criterion.
Recall that any node  $u$ in $T$ at the level of a coding node $c_k^T$ has a unique immediate extension $u^+$  of length $l_k^T+1$ in $\widehat{T}$;
 so there is no ambiguity to
consider $u^+(l_k^T)$ to be the  passing number of $u$ at $c_k$, even though technically $u$ is not defined on input $l_k^T$.

\begin{lem}[Passing Number Choice Extension Lemma]\label{lem.pnc}
Let $T$ be a strong coding tree and
$A$ be any finite valid subtree of  $T$.
Let $l_A$ denote the length of the members of $\max(A)$
and let $A^+$ denote  the set of all members of $\widehat{T}$ of length $l_A+1$  which extend some member of $\max(A)$.
 List the nodes of $A^+$ as $s_i$, $i<\tilde{i}$.
Fix any  $d<\tilde{i}$.
For each $i\ne d$,  if  $s_i$ and $s_d$  have no parallel $1$'s, fix any $\varepsilon_i\in\{0,1\}$;  if $s_i$ and $s_d$ have parallel $1$'s, let $\varepsilon_i=0$.
In particular, $\varepsilon_d=0$.

Then  for each $j<\om$,
there is a coding node $c_k$ with $k\ge j$
extending $s_d$
 and
 extensions $u_i\contains s_i$, $i\in\tilde{i}\setminus \{d\}$, of length $l_k$
such that
the passing number of $u_i$ at $c_k$ is $\varepsilon_i$
 for each $i\in \tilde{i}\setminus\{d\}$.
Furthermore,
the nodes $u_i$ can be chosen so that
any new parallel $1$'s among $\{u_i:i<\tilde{i}\}$ which were not witnessed in $A$ are
 witnessed by $c_k$, and their first instances take place in the $k$-th interval of $T$.
In particular, if $A\cup \{s_i:i<\tilde{i}\}$ satisfies the \POC, then  $A\cup\{u_i:i<\tilde{t}\}$ also satisfies the \POC, where $u_d=c_k$.
\end{lem}

\begin{proof}
Assume the hypotheses of the lemma.
Let $j'$ be such that the nodes $\{s_i:i<\tilde{i}\}$ are in the $j'$-th interval of $T$.
For each $i<\tilde{i}$, let $t_i$ be the leftmost extension of $s_i$ of length $l_{j'}+1$.
Since $A$ is a valid subtree of $T$,
no new sets of parallel $1$'s are acquired by $\{t_i:i< \tilde{i}\}$.
Let $j<\om$ be given and
take $k\ge \max(j,j'+1)$ minimal such that $c_k\contains t_d$, and let $u_d=c_k$.
Such a $k$ exists since the coding nodes are dense in $T$.
For each  $i\ne d$,
extend $t_i$ via its leftmost extension to the level of
$l_{k-1}+1$, and label it $t'_i$.
By Lemma \ref{lem.poc},
 for $i\ne d$,
no  new  sets of parallel $1$'s  are acquired
by
$\{t_i':i\in\tilde{i}\setminus\{d\}\}\cup\{u_d\re (l_{k-1}+1)\}$.
For each $i\ne d$ for which $\varepsilon_i=0$,
let $u_i$ be the leftmost extension of $t'_i$ of length $l_k+1$.
For $i<\tilde{i}$ such that
 $\varepsilon_i=1$,
let $u_i$ be the rightmost extension of $t'_i$ to length
 $l_k+1$.
Note that  for each $i<\tilde{i}$,
the passing number of
of $u_i$ at  $c_k$  is
 $\varepsilon_i$.

For any $I\sse\tilde{i}$ of size at least two,
if there is some $l$ such that $u_i(l)=1$ for all $i\in I$,
and
 the least $l$ for which this holds is greater than  $l_A$,
then it must be that $u_i(l_k)=1$ for each $i\in I$,
 since
no new sets of parallel $1$'s are acquired among $\{u_i:i<\tilde{i}\}$ below $l_{k-1}+1$.
Thus, the set $\{u_i:i<\tilde{i}\}$ satisfies the lemma.
If $A$ satisfies the \POC,
then it is clear that $A\cup\{u_i:i<\tilde{i}\}$ also satisfies the \POC, since all the new parallel $1$'s are witnessed by the coding node $u_d=c_k$.
\end{proof}

The final lemma of this section combines the previous two, to show that any finite valid subtree of a strong coding tree can be extended to another valid subtree with any prescribed strong similarity type.

\begin{lem}\label{lem.facts}
Let $A$ be a finite  valid subtree of any strong coding tree  $T$, and let $l_A$ be the length of the nodes in $\max(A)$.
Fix any member $u\in\max(A)^+$.
Let $\Spl(u)$ be any set of nodes  $s\in \max(A)^+$
which  have no parallel $1$'s with $u$,
and let $Z$ denote $\max(A)^+\setminus(\Spl(u)\cup\{u\})$.
Let $l\ge l_A$ be given.
Then there is an $l_*>l$
 and extensions $u_*\supset u$,
$s_*^0,s_*^1\supset s$ for all $s\in\Spl(u)$,
and $s_*\supset s$ for all $s\in Z$, each of length $l_*$,
such that, letting
\begin{equation}
X=\{u_*\}\cup\{s_*^i:s\in\Spl(u),\ i\in\{0,1\}\}\cup\{s_*:s\in Z\},
\end{equation}
and $B$ be the tree induced by $A\cup X$,
 the following hold:
\begin{enumerate}
\item
$u_*$ is a coding node.
\item
For each $s\in\Spl(t)$ and $i\in\{0,1\}$, the passing number of $s_*^i$ at $u_*$ is $i$.
\item
For each $s\in  Z$,
the passing number of
 $s_*$ at $u_*$ is $0$.
\item
Splitting among the extensions of the $s\in \Spl(u)$ occurs in reverse lexicographic order:
For $s$ and $t$ in $\Spl(u)$,
 $|s_*^0\wedge s_*^1|<|t^0_*\wedge t^1_*|$
if and only if $s_*>_{\mathrm{lex}}t_*$.
\item
There are no new sets of  parallel $1$'s
among  the  nodes  in $X$
until they pass the level of the longest splitting node in $B$  below $u_*$.
\end{enumerate}
In particular, if $A$ satisfies the \POC, then so does $B$.
\end{lem}

\begin{proof}
Since $A$ is valid in $T$,
apply Lemma \ref{lem.factssplit}
to extend $\max(A)$ to have splitting nodes in the desired order without adding any new sets of parallel $1$'s.
Then apply Lemma
\ref{lem.pnc} to extend to a level with a coding node and passing numbers as prescribed.
\end{proof}

It follows from Lemma \ref{lem.facts}
that whenever $A$ is a finite strong coding tree
 which is valid in some  strong coding tree $T$
and strongly similar to $r_m(T)$,
then $r_{m+1}[A,T]$ is infinite.
In particular, $A$ can be extended to a strong coding tree $S$ such that $S\le T$.

\begin{rem}
It is straightforward to check
that the
space $(\mathcal{T}(T),\le, r)$ of strong coding trees
satisfies Axioms \bf A.1\rm, \bf A.2\rm, and \bf A.3(1) \rm of Todorcevic's   axioms  in Chapter 5 of \cite{TodorcevicBK10} guaranteeing a  topological Ramsey space.
On the other hand,
\bf A.3(2) \rm does not hold, and \bf A.4\rm, the pigeonhole principle, holds in a  modified form where the finite  subtree being extended is a valid subtree of the strong coding tree,
as will follow from Theorem \ref{thm.MillikenIPOC}.
It remains open what sort of infinitary Ramsey theory in the vein of \cite{Milliken81}
 holds in $(\mathcal{T}(T),\le, r)$,  in terms of  its Ellentuck topology.
\end{rem}


\section{Halpern-Lauchli-style Theorems for strong coding trees}\label{sec.5}

The
Ramsey theory content for strong coding trees
begins in this section.
The  ultimate goal is to obtain a Ramsey theorem for colorings of strictly similar (Definition \ref{defn.ssimtype}) copies of any given finite antichain of coding nodes, as these are the structures which will code finite triangle-free graphs.
This is accomplished in Theorem \ref{thm.mainRamsey}.
As a mid-point, we will prove a Milliken-style theorem (Theorem  \ref{thm.MillikenIPOC}) for
finite trees satisfying some strong version of the \POC.
Just as
 the Halpern-\Lauchli\ Theorem forms
the core content
 of Milliken's Theorem in the setting of strong trees,
so too in the setting of strong coding trees,
 Halpern-\Lauchli-style theorems are proved first and then applied to obtain
 Milliken-style theorems in later sections.

The main and only theorem of this section is
Theorem  \ref{thm.matrixHL}.
This  general theorem  encompasses colorings of two different types  of level set  extensions of a fixed finite tree: The  level set  either contains a splitting node (Case (a)) or a coding node (Case (b)).
In Case (a), we obtain a direct analogue of the
Halpern-\Lauchli\ Theorem.
In Case (b), we obtain a weaker version of the
Halpern-\Lauchli\ Theorem, which is later strengthened to
the direct analogue in Lemma \ref{lem.Case(c)}.

The structure of the proof follows the basic outline of Harrington's  proof of the
Halpern-\Lauchli\ Theorem, as outlined to the author by  Laver.
The reader wishing to read that proof as a warm-up is referred to  Section 2 of
 \cite{DobrinenRIMS17}.
In the setting of strong coding trees,
new considerations arise, and
 new forcings have to be established to achieve the result.
The main reasons that new
 forcings are needed are firstly, that
there are two types of nodes, coding and splitting nodes,
  and  secondly,  that the extensions achieving homogeneity must be  extendible to a strong coding tree valid inside the  ambient  tree.
This second property necessitates that  the extensions be valid and satisfy the \POC,
 and  is responsible for the strong definition of the partial ordering on the forcing.
The former is responsible for there being Cases (a) and (b).
The forcings will consist of conditions which
are finite functions with images which are
 certain level sets  of a given tree strong coding tree $T$, but the partial ordering will be stronger than the partial ordering of  subtree as branches added will have some dependence between them, so these are not simply Cohen forcings.

\begin{rem}
Although the
  proof  uses the set-theoretic technique of forcing,
the whole  construction takes place  in the original model of ZFC, not in  some generic extension.
The forcing should be thought of as conducting an unbounded search for a finite object, namely the finite set of nodes of  a  prescribed  form where homogeneity is attained.
Thus, the result and its proof
  hold  using only the standard axioms of mathematics.
\end{rem}

The following terminology and notation will be used throughout.
Let $T$ be a strong coding  tree.
Given finite subtrees $U,V$ of $T$,
we write $U\sqsubseteq V$ to mean that
there is some $k$ such that $U=\bigcup_{m<k}\Lev_U(m)=\bigcup_{m<k}\Lev_V(m)$,
and we say that $V$ {\em extends} $U$, or that $U$ is an {\em initial subtree of} $V$.
We write $U\sqsubset V$ if $U$ is a proper initial subtree of $V$.
Recall
from Definition \ref{AT}
that
$S\le T$ means that $S$ is a valid subtree of $T$
which is  strongly similar to $T$, and hence also a strong coding tree.
Given a finite strong coding tree $B$,
 $[B,T]$ denotes the set of all   $S\le T$ such that  $S$ extends $B$.
A set $X\sse \widehat{T}$ is a {\em level set} if all nodes in $X$ have the same length.
For level sets $X,Y$ we shall also say that $Y$ {\em extends} $X$ if $X$ and $Y$ have the same number of nodes and each node in
$X$ is extended by a unique node in
$Y$.
For level sets $Y=\{y_i:i\le d\}$ and $X=\{x_i:i\le d\}$
with $y_i\contains x_i$ for each $i\le d$,
we say that {\em $Y$ has no new sets of parallel $1$'s over $X$}
if for each $I\sse d+1$ for which there is an $l$ such that $y_i(l)=1$ for each $i\in I$,
then there is an $l'$ such that $x_i(l')=1$ for each $i\in I$.
For
any tree $U\sse\widehat{T}$ and  any $l<\om$ , let
$U\re l$ denote the set of $s\in \widehat{U}$ such that $|s|=l$.
A set of two or more nodes $\{x_i:i\in I\}$
  in $\widehat{T}$ is said to have {\em first parallel $1$'s at level $l$} if $l$ is least such that $x_i(l)=1$ for all $i\in I$.

For each $s\in \widehat{T}$, if $i\in\{0,1\}$ and $s^{\frown}i$ is in $\widehat{T}$, then we say that $s^{\frown}i$ is an {\em immediate extension of $s$} in $T$.
Thus, splitting nodes in $T$ have two immediate extensions in $T$, and non-splitting nodes, including every node at the level of a coding node, have exactly one immediate extension in $T$.
For a non-splitting node $s$ in $T$, we let $s^+$ denote the immediate extension of $s$ in $T$.
Given a finite subtree $A$ of $T$, let  $l_A$ denote the maximum of
the lengths of members of $A$,
 and let $\max(A)$ denote the set of all nodes in $A$  with length $l_A$.
Let $A^+$ denote the set of   immediate extensions in $\widehat{T}$ of the members of $\max(A)$:
\begin{equation}
A^+=\{s^{\frown}i : s\in \max(A),\ i\in\{0,1\},\mathrm{\ and\ } s^{\frown}i\in \widehat{T}\}.
\end{equation}
Note that $A^+$ is a level set of nodes of length $l_A+1$.

We now provide the set-up for the two  cases before stating the theorem.
\vskip.1in

\noindent\underline{\bf{The Set-up for Theorem \ref{thm.matrixHL}.}}
Let $\bT$ be a fixed strong coding tree, and let
$T\le \bT$ be given.
Let
$A$ be a finite valid subtree of $T$ satisfying the Parallel $1$'s Criterion.
It is fine for
$A$ to  have terminal  nodes at different levels, indeed, we need to allow  this for the intended applications later.
Without loss of generality and to simplify the presentation of the proof,
assume that $0^{l_A}$ is in $A$.
Let  $A_e$ be a  subset of $A^+$
containing $0^{l_A+1}$ and
of size at least two.
Let $C$ be a  finite valid subtree of $T$ containing $A$  such that
$C$ satisfies the \POC\ and
the collection of all nodes in $C$ not in $A$, denoted $C\setminus A$,
 forms
 a level set extending  $A_e$.
Assume moreover that  $0^{l_C}$ is the node in $C$ extending $0^{l_A+1}$,
 where $l_C$ is the length of the nodes in $C\setminus A$.
The two cases   are the following:
\begin{enumerate}
\item[]
\begin{enumerate}
\item[\bf{Case (a).}]
  $C\setminus A$ contains a splitting node.
\end{enumerate}
\end{enumerate}

In Case (a), define
$\Ext_T(A,C)$
to be the collection  of all  level sets
$X\sse T$ extending $A_e$
such that
$A\cup X\ssim C$ and $A\cup X$  is valid in $T$.
We point out that $A\cup X$ being valid in $T$ is equivalent to $X$ having  no pre-determined new parallel $1$'s.
It will turn out to be necessary to require this of $X$,
and the extensions for which the coloring is relevant will have this property anyway.

\begin{enumerate}
\item[]
\begin{enumerate}
\item[\bf{Case (b).}]
 $C\setminus A$ contains a coding  node.
\end{enumerate}
\end{enumerate}

In Case (b),
define
$\Ext_T(A,C)$
to be the collection  of all  level sets
$X\sse T$ extending $A_e$
such that
$A\cup X\ssim C$.
Since $X$ contains a coding node, $A\cup X$ is automatically valid in $T$.
Recalling (7) of Definition \ref{def.3.1.likeSauer},
$A\cup X\ssim C$ implies that, letting $f:A\cup X\ra C$ be the  strong similarity map,
for each $x\in X$
the passing number of $x^+$ at the coding node in $X$
equals the passing number of $(f(x))^+$ at the coding node in $C\setminus A$.
 Given any $X\in\Ext_T(A,C)$,
let $\Ext_T(A,C;X)$ denote the set of $Y\in\Ext_T(A,C)$ such that $Y$ extends $X$.

In both cases, $A\cup X\ssim C$ implies that $A\cup X$ satisfies the \POC.

\begin{thm}\label{thm.matrixHL}
Let $T\le \bT$ be  any strong coding tree and let  $B$ be a finite strong coding tree  valid in  $T$.
Let
 $A\sqsubset C$   be finite valid subtrees of $T$ such that both $A$ and $C$ satisfy the \POC,
$A$ is a subtree of $B$,
$C\setminus A$ is a level set of size at least two,
and
$0^{l_C}\in C$.
Further, assume that
the nodes in $C\setminus A$ extend nodes in $\max(A)\cap\max(B)$.
Let $A_e$ denote the set of  nodes in $A^+$ which are extended to nodes in $C\setminus A$.

In Case (a),
given any coloring $h:\Ext_T(A,C)\ra 2$,
 there is a strong coding tree $S\in [B,T]$ such that
$h$ is monochromatic on $\Ext_S(A,C)$.

In Case (b),  suppose
 $X\in \Ext_T(A,C)$  and
$m_0$ are given
for which there is a $B'\in r_{m_0}[B,T]$ with
$X\sse\max(B')$.
Then for any
 coloring $h:\Ext_T(A,C)\ra 2$
there is a strong coding tree $S\in [r_{m_0-1}(B'),T]$ such that
 $h$ is monochromatic on $\Ext_S(A,C;X)$.
\end{thm}

\begin{proof}
Let $T,A,A_e,B,C$ be given satisfying the hypotheses  of  either  Case (a)  or (b),
and
let $h:\Ext_T(A,C)\ra 2$ be a given coloring.
Let $d+1$ equal the number of  nodes in $A_e$.
List the  nodes of $A_e$ as  $s_0,\dots, s_d$, letting $s_d$ denote the node of $A_e$ that is extended to
the critical node in $C\setminus A$:
a splitting node  in Case (a)
and
 a coding node in Case (b).
For each $i\le d$,
let $t_i$ denote the node in $\max(C)$ which extends $s_i$.
In particular,
$t_d$ denotes the splitting or coding node in $\max(C)$.
Let $i_0$ denote the integer such that
 $s_{i_0}$ is  the node of $A_e$ which is a sequence of  $0$'s.
Then $t_{i_0}$ is the sequence of all $0$'s  in $C\setminus A$.
Notice that  $i_0$ can equal $d$ only if we are  in  Case (a)  and moreover
the  splitting node in $C\setminus A$ is a sequence of $0$'s.
In Case (b), the following notation will be used:
For each $i\le d$,
 $t^+_i$ denotes the member in $\max(C)^+$ extending $t_i$.
Let $I_{0}$ denote  the set of all $i<d$ such that $t^+_i(|t_d|)=0$
and let
$I_{1}$ denote  the set of all $i<d$ such that $t^+_i(|t_d|)=1$.

Let $L$ denote the collection of all  $l<\om$ such that there is a member of
 $\Ext_T(A,C)$ with maximal  nodes of length $l$.
$L$ is infinite since $B$ is valid in $T$.
In Case (a),
$L$  is exactly the set of  all $l<\om$ for which
 there is a  splitting node  of length $l$ extending $s_d$,
 and in Case (b),
 $L$  is exactly the set of  all $l<\om$ for which
 there is a coding  node  of length $l$ extending $s_d$, as this follows from the validity of $B$ in $T$ and Lemma \ref{lem.pnc}.
For each
  $i\in (d+1)\setminus\{i_0\}$,   let  $T_i=\{t\in T:t\contains s_i\}$;
let $T_{i_0}=\{t\in T:t\contains s_{i_0}$ and $t\in 0^{<\om}\}$, the collection of all leftmost nodes in $T$ extending $s_{i_0}$.

Let $\kappa=\beth_{2d}$.
The following forcing notion $\bP$    adds $\kappa$ many paths through  $T_i$, for each  $i\in d\setminus\{i_0\}$,
and one path through $T_d$.
If $i_0\ne d$, then $\bP$ will add one
 path through $T_{i_0}$, though allowing  $\kappa$ many ordinals to  label this path in order to simplify notation.
\vskip.1in

\noindent  \underline{Case (a)}.
$\bP$ is the set of conditions $p$ such that
$p$ is a  function
of the form
$$
p:(d\times\vec{\delta}_p)\cup\{d\}\ra T\re l_p,
$$
where $\vec{\delta}_p\in[\kappa]^{<\om}$ and $l_p\in L$,
such that
\begin{enumerate}
\item[(i)]
$p(d)$ is {\em the} splitting  node extending $s_d$ of length  $l_p$;
\item [(ii)]
For each $i<d$,
 $\{p(i,\delta) : \delta\in  \vec{\delta}_p\}\sse  T_i\re l_p$;
and
\item[(iii)]
$\{p(i,\delta):(i,\delta)\in d\times\vec{\delta}_p\}\cup \{p(d)\}$ has no pre-determined new parallel $1$'s.
\end{enumerate}
\vskip.1in

\noindent  \underline{Case (b)}.
$\bP$ is the set of conditions $p$ such that
$p$ is a  function
of the form
$$
p:(d\times\vec{\delta}_p)\cup \{d\}\ra T\re l_p,
$$
where $\vec{\delta}_p\in[\kappa]^{<\om}$ and $l_p\in L$,
such that
\begin{enumerate}
\item[(i)]
$p(d)$ is {\em the} coding  node extending $s_d$ of length $l_p$;
\item [(ii)]
For each $i<d$,
 $\{p(i,\delta) : \delta\in  \vec{\delta}_p\}\sse  T_i\re l_p$.
\item[(iii)]
For each $\delta\in\vec{\delta}_p$,
 $j\in \{0,1\}$,
and  $i\in I_j$, the immediate extension of $p(i,\delta)$ in $T$ is $j$;
that is, the passing number of $(p(i,\delta))^+$ at $p(d)$ is $j$.
\end{enumerate}

 In both Cases (a) and (b), the partial
 ordering on $\bP$  is defined as follows:
$q\le p$ if and only if
$l_q\ge l_p$, $\vec{\delta}_q\contains \vec{\delta}_p$, and
\begin{enumerate}
\item[(i)]
$q(d)\contains p(d)$,
and
$q(i,\delta)\contains p(i,\delta)$ for each $(i,\delta)\in d\times \vec{\delta}_p$;

\item[(ii)]
The set $\{q(i,\delta):(i,\delta)\in d\times \vec{\delta}_p\}\cup \{q(d)\}$ has no new sets of parallel $1$'s over
 $\{p(i,\delta):(i,\delta)\in d\times \vec{\delta}_p\}\cup \{p(d)\}$.
\end{enumerate}
\vskip.1in

Given $p\in\bP$,
we shall use $\ran(p)$ to denote the {\em range of $p$},
$\{p(i,\delta):(i,\delta)\in d\times \vec\delta_p\}\cup \{p(d)\}$.
If $p,q$ are members of $\bP$,
we shall use the abbreviation {\em $q$ has no new parallel $1$'s over $p$}
to mean that  $\ran(q)$ has no new sets of parallel $1$'s over $\ran(p)$.

The proof of the theorem proceeds in three parts.
Part I proves that  $\bP$ is an atomless partial order.
Part II proves Lemma \ref{lem.compat} which is
the main tool for building fusion sequences while preserving homogeneity.
This is applied
 in Part III to build
 the tree $S$ which is  valid in $T$ and such that  $\Ext_S(A,C)$ is homogeneous for  $h$ in Case (a), and  $\Ext_S(A,C;X)$ is homogeneous for $h$  in Case (b).
For the first two parts of the proof,
we present a general proof, indicating the steps at which the two cases require different approaches.
Part III will require the cases to  be handled separately.
\vskip.1in

\noindent \underline{\bf {Part I.}} $\bP$ is an atomless partial ordering.

\begin{claim}\label{claim.atomless_separative}
 $(\bP,\le)$ is a partial ordering.
\end{claim}

\begin{proof}
The  order  $\le$ on $\bP$ is clearly  reflexive and antisymmetric.
Transitivity follows from the fact that the requirement (ii) is a transitive property.
If $p\ge q$ and $q\ge r$,
then $\vec{\delta}_p\sse\vec{\delta}_q\sse \vec{\delta}_r$ and $l_p\le l_q\le l_r$.
Since
$r$
has no new sets of parallel $1$'s over
 $q$
and
$q$
has no new sets of parallel $1$'s over
 $p$,
it follows that
$r$
has no new sets of parallel $1$'s over
$p$.
Thus, $p\ge r$.
\end{proof}

We show that $\bP$ is atomless by proving the  following stronger claim.

\begin{claim}\label{claim.densehigh}
For each $p\in\bP$ and   $l>l_p$, there
are  $q,r\in\bP$ with $l_q,l_r> l$  such that  $q,r<p$ and $q$ and $r$ are incompatible.
\end{claim}

\begin{proof}
Let $p\in \bP$ and $l>l_p$ be given, and
let $\vec{\delta}_r=\vec{\delta}_q=\vec{\delta}_p$.

In Case (a),
take $q(d)$ and $r(d)$ to be incomparable splitting nodes in $T$ extending $p(d)$ to some lengths greater than $l$.
Such splitting nodes exist since strong coding trees are perfect.
Let $l_q=|q(d)|$ and $l_r=|r(d)|$.
For each $(i,\delta)\in d\times\vec{\delta}_p$,
let $q(i,\delta)$ be the leftmost extension (in $T$) of $p(i,\delta)$ to length $l_q$,
and  let
$r(i,\delta)$ be the leftmost extension of $p(i,\delta)$ to length $l_r$.
Then $q$ and $r$ have no  pre-determined  new parallel $1$'s, since $\ran(p)$    has no pre-determined new parallel $1$'s and all nodes except $q(d)$ and $r(d)$ are leftmost extensions in $T$ of members of $\ran(p)$;
so $q$ and $r$ are members of $\bP$.
By Lemma \ref{lem.poc},
both $q$ and $r$
 have no new  parallel $1$'s over $p$,
so  $q,r\le p$.
Since neither of $q(d)$ and $r(d)$ extends the other,
$q$ and $r$ are incompatible.

In Case (b),
let $s$ be a splitting node in $T$ of length greater than $l$ extending $p(d)$.
Let $c^T_k$ be the least coding node in $T$ above $s$.
Let $s_0,s_1$ extend $s^{\frown}0,s^{\frown}1$ leftmost in $T$  to the level of $c^T_k$, respectively.
For each $(i,\delta)\in d\times\vec\delta_p$,
let $p'(i,\delta)$ be the leftmost extension  in $T$ of $p(i,\delta)$ of length $l^T_k$.
 By  Lemma \ref{lem.pnc},
there are  $q(d)\contains s_0$ and   $q(i,\delta)\contains p'(i,\delta)$, $(i,\delta)\in d\times \vec{\delta}_p$,
such that
\begin{enumerate}
\item
 $q(d)$ is a coding node;
\item
$q$ has no new  parallel $1$'s over  $p$;
\item
 For each $j<2$,
$i\in I_j$ if and only if the immediate extension of $q(i,\delta)$ is $j$.
\end{enumerate}
Then $q\in \bP$ and $q\le p$.
Likewise  by  Lemma \ref{lem.pnc},
we may extend $\{p'(i,\delta): (i,\delta)\in d\times \vec{\delta}_p\}\cup\{s_1\}$
to
$\{ r(i,\delta):(i,\delta)\in d\times \vec{\delta}_p\}\cup\{r(d)\}$ to form a condition $r\in\bP$ extending $p$.
Since the coding nodes $q(d)$ and $r(d)$ are incomparable, $q$ and $r$  are incompatible conditions in $\bP$.
\end{proof}

From now on, whenever
 ambiguity will  not arise by doing so,
we will  refer to the splitting node in Case (a) and the coding node in Case (b)   simply as the
{\em critical node}.

Let $\dot{b}_d$ be a $\bP$-name for the generic path through $T_d$;
that is, $\dot{b}_d=\{\lgl p(d),p\rgl:p\in\bP\}$.
Note that for each $p\in \bP$, $p$ forces that $\dot{b}_d\re l_p= p(d)$.
By Claim \ref{claim.densehigh}, it is dense to force a critical node in $\dot{b}_d$ above any given level in $T$,  so $\mathbf{1}_{\bP}$ forces that the set of levels of critical nodes in $\dot{b}_d$ is infinite.
Thus, given any generic  filter  $G$  for $\bP$, $\dot{b}_d^G=\{p(d):p\in G\}$ is a cofinal path of critical nodes in $T_d$.
Let $\dot{L}_d$ be a $\bP$-name for the set of lengths of critical  nodes in $\dot{b}_d$.
Note that $\mathbf{1}_{\bP}\forces \dot{L}_d\sse L$.
Let $\dot{\mathcal{U}}$ be a $\bP$-name for a non-principal ultrafilter on $\dot{L}_d$.
For each $i<d$ and $\al<\kappa$,  let $\dot{b}_{i,\al}$ be a $\bP$-name for the $\al$-th generic branch through $T_i$;
that is, $\dot{b}_{i,\al}=\{\lgl p(i,\al),p\rgl:p\in \bP$ and $\al\in\vec{\delta}_p\}$.
For $i<d$ and
 for any condition $p\in \bP$ and  $\al\in \vec\delta_p$,  $p$ forces that $\dot{b}_{i,\al}\re l_p= p(i,\al)$.

For ease of notation, we  shall write sets $\{\al_i:i< d\}$ in $[\kappa]^d$ as vectors $\vec{\al}=\lgl \al_0,\dots,\al_{d-1}\rgl$ in strictly increasing order.
For $\vec{\al}=\lgl\al_0,\dots,\al_{d-1}\rgl\in[\kappa]^d$,
rather than writing out
$\lgl \dot{b}_{0,\al_0},\dots, \dot{b}_{d-1,\al_{d-1}},\dot{b}_d\rgl$
 each time we wish to refer to these generic branches,
we shall simply
\begin{equation}
\mathrm{let\ \ }\dot{b}_{\vec{\al}}\mathrm{\  \  denote\ \ }
\lgl \dot{b}_{0,\al_0},\dots, \dot{b}_{d-1,\al_{d-1}},\dot{b}_d\rgl,
\end{equation}
since the branch $\dot{b}_d$ being  unique  causes no ambiguity.
For any $l<\om$,
\begin{equation}
\mathrm{\ let\ \ }\dot{b}_{\vec\al}\re l
\mathrm{\ \ denote \  \ }
\{\dot{b}_{i,\al_i}\re l:i<d\}\cup \{\dot{b}_d\re l\}.
\end{equation}
Using the abbreviations just defined,
$h$ is a coloring on sets of nodes of the form $\dot{b}_{\vec\al}\re l$
whenever this is
 forced to be a member of $\Ext_T(A,C)$.
\vskip.1in

\noindent\underline{\bf{Part II.}}
The goal  now is  to
prove
Claims
\ref{claim.onetypes} and \ref{claim.j=j'}
and
Lemma \ref{lem.compat}.
To sum up, they secure that
there are
 infinite  pairwise disjoint sets $K_i\sse \kappa$ for $i<d$,
and a set of conditions $\{p_{\vec\al}:\vec\al\in \prod_{i<d}K_i\}$ which are compatible,
have the same images in $T$,
and such that for some fixed $\varepsilon^*\in\{0,1\}$,
for each $\vec\al\in\prod_{i<d}K_i'$,
$p_{\vec\al}$ forces
$h(\dot{b}_{\vec\al}\re l)=\varepsilon^*$  for  ultrafilter many  $l\in\dot{L}_d$.
Moreover, we will find nodes  $t^*_i$, $i\le d$, such that for each $\vec\al\in\prod_{i<d}K_i$,
$p_{\vec\al}(i,\al_i)=t^*_i$.
Lemma   \ref{lem.compat} will enable  fusion processes for  constructing $S$  with one color on $\Ext_S(A,C)$  in Part III.
There are no differences between the arguments for Cases (a) and (b) in Part II.

For each $\vec\al\in[\kappa]^d$,
choose a condition $p_{\vec{\al}}\in\bP$ such that
\begin{enumerate}
\item
 $\vec{\al}\sse\vec{\delta}_{p_{\vec\al}}$.

\item
$\{p_{\vec\al}(i,\al_i):i< d\}\cup\{p(d)\}\in\Ext_T(A,C)$.
\item
$p_{\vec{\al}}\forces$ ``There is an $\varepsilon\in 2$  such that
$h(\dot{b}_{\vec{\al}}\re l)=\varepsilon$
for $\dot{\mathcal{U}}$ many $l$ in $\dot{L}_d$."

\item
$p_{\vec{\al}}$ decides a value for $\varepsilon$, call it  $\varepsilon_{\vec{\al}}$.
\item
$h(\{p_{\vec\al}(i,\al_i):i< d\}\cup\{p(d)\})=\varepsilon_{\vec{\al}}$.
\end{enumerate}

Properties (1) -  (5) can be guaranteed  as follows.
Recall that for $i\le d$, $t_i$ denotes the member of $C\setminus A$ extending $s_i$.
For each  $\vec{\al}\in[\kappa]^d$, let
$$
p^0_{\vec{\al}}=\{\lgl (i,\delta), t_i\rgl: i< d, \ \delta\in\vec{\al} \}\cup\{\lgl d,t_d\rgl\}.
$$
Then $p^0_{\vec{\al}}$ is a condition in $\bP$,
$\vec\delta_{p_{\vec\al}^0}= \vec\al$, so (1) holds.
Further,
$\{p^0_{\vec\al}(i,\al_i):i<d\}\cup\{p^0_{\vec{\al}}(d)\}$ is a member of $\Ext_T(A,C)$ since it is exactly $C\setminus A$.
It is important to note that
for any $p\le p_{\vec\al}^0$,
 $\{p(i,\al_i):i< d\}\cup\{p(d)\}$
is also  a member of $\Ext_T(A,C)$,
as this follows from the fact that
$\{p(i,\delta):(i,\delta)\in d\times \vec{\delta}_{p_{\vec\al}^0}\}\cup\{p(d)\}$ has no new sets of parallel $1$'s over the image of $p_{\vec\al}^0$.
Thus (2) holds for any $p\le p_{\vec\al}^0$.
Take  an extension $p^1_{\vec{\al}}\le p^0_{\vec{\al}}$ which
forces  $h(\dot{b}_{\vec{\al}}\re l)$ to be the same value for
$\dot{\mathcal{U}}$  many  $l\in \dot{L}_d$, giving (3).
For Property (4),
since $\bP$ is a forcing notion, there is a $p^2_{\vec{\al}}\le p_{\vec{\al}}^1$ deciding a value $\varepsilon_{\vec{\al}}$ for which $p^2_{\vec{\al}}$ forces that $h(\dot{b}_{\vec{\al}}\re l)=\varepsilon_{\vec{\al}}$
for $\dot{\mathcal{U}}$ many $l$ in $\dot{L}_d$.
If $p^2_{\vec\al}$ forces 
$h(\dot{b}_{\vec\al}\re l_{p^2_{\vec\al}})=\varepsilon_{\vec\al}$, then let $p_{\vec\al}=p^2_{\vec\al}$.

If not,
take  some  $p^3_{\vec\al}\le p^2_{\vec\al}$
which decides
some $l\in\dot{L}$
such that
$l_{p^2_{\vec\al}}< l_n^T< l\le l_{p^3_{\vec\al}}$,
 for some $n$,
and  $p^3_{\vec\al}$ forces
$h(\dot{b}_{\vec\al}\re l)=\varepsilon_{\vec\al}$.
Since $p^3_{\vec\al}$ forces ``$\dot{b}_{\vec\al}\re l=
\{p^3_{\vec\al}(i,\al_i)\re l:i<d\}   \cup\{p^3_{\vec\al}(d)\re l\}$'' and $h$ is defined in the ground model,
 this means that  $p^3_{\vec\al}(d)\re l$ is a splitting node in Case (a) and a coding node in Case (b), and
\begin{equation}\label{eq.hrest}
h(X(p^3_{\vec\al})\re l)
=\varepsilon_{\vec\al},
\end{equation}
where
$X(p^3_{\vec\al})\re l$ denotes
$\{p^3_{\vec\al}(i,\al_i)\re l:i<d\}   \cup\{p^3_{\vec\al}(d)\re l\}$.
If $l=l_{p^3_{\vec\al}}$, let $p_{\vec\al}=p_{\vec\al}^3$, and note that $p_{\vec\al}$ satisfies (1) - (5).

Otherwise,  $l<l_{p^3_{\vec\al}}$.
In Case (a),  let
$p_{\vec\al}$ be  defined as follows:
Let  $\vec\delta_{\vec\al}=\vec\delta_{p_{\vec\al}^2}$ and
 \begin{equation}
\forall (i,\delta)\in d\times\vec\delta_{\vec\al}, \mathrm{\ let\ }
p_{\vec\al}(i,\delta)=p^3_{\vec\al}(i,\delta)\re l\mathrm{\ \ and\  let\  }
p_{\vec\al}(d)=p^3_{\vec\al}(d)\re l.
\end{equation}
Since $p^3_{\vec\al}$ is a condition in $\bP$,
$\ran(p^3_{\vec\al})$ is free in $T$.
Furthermore, $p^3_{\vec\al}\le p^2_{\vec\al}$
implies that $\ran(p_{\vec\al}^3\re \vec\delta_{p^2_{\vec\al}})$ has no new sets of parallel $1$'s over $\ran(p^2_{\vec\al})$.
Therefore,
$p_{\vec\al}$ is a condition in $\bP$ with   $p_{\vec\al}\le p_{\vec\al}^2$ satisfying  (1) - (5).

In Case (b),
  construct $p_{\vec\al}$ as follows:
Again, let
 $\vec{\delta}_{\vec\al}=\vec\delta_{p^2_{\vec\al}}$.
For each $i<d$, define
$p_{\vec\al}(i,\al_i)=p^3_{\vec\al}(i,\al_i)\re l$,
and  let
$p_{\vec\al}(d)=p^3_{\vec\al}(d)\re l$.
Letting $X=\{p^3_{\vec\al}(i,\al_i)\re l:i<d\}\cup\{p^3_{\vec\al}(d)\re l\}$,   then 
  $h(X)=\varepsilon_{\vec\al}$.
Let $U$ denote 
$\{p^2_{\vec\al}(i,\al_i)\re l:i<d\}\cup\{p^2_{\vec\al}(d)\re l\}$
and
let $U'=\ran (p_{\vec\al}^2)\setminus U$.
Note that $X$ end-extends $U$,  and $X$ is valid  in $T$ and has no new sets of parallel $1$'s over $U$.
By Lemma
\ref {lem.pnc},
there is an $X'$ end-extending $U'$ to nodes in $T\re l$ so that  
$X\cup X'$ 
has no new sets of parallel $1$'s  over $U\cup U'$, and 
each node in $X'$ has the same passing  number at $l$ as it does at $l_{p_{\vec\al}^2}$.
Let $\ran(p_{\vec\al})= X\cup X'$,
where
 for each $i<d$
and $(i,\delta)\in d\times\vec{\delta}_{p_{\vec\al}^3}$ with $\delta\ne\al_i$,
we
let $p_{\vec\al}(i,\delta)$ be the node in $Y'$ extending
 $p_{\vec\al}^3(i,\delta)$.
 This defines a condition $p_{\vec\al}\le p_{\vec\al}^2$
  satisfying  (1) - (5).

Since $\{p_{\vec\al}(i,\al_i):i<d\}\cup\{p_{\vec\al}(d)\}$ is what $p_{\vec\al}$ forces $\dot{b}_{\vec\al}\re l$ to be,
it follows that
$p_{\vec\al}$ forces
 $h(\{p_{\vec\al}(i,\al_i):i<d\}\cup\{p_{\vec\al}(d)\})=\varepsilon_{\vec\al}$, so (5) holds.

Now we prepare  for an application of the \Erdos-Rado Theorem (recall Theorem \ref{thm.ER}).
We are assuming $\kappa=\beth_{2d}$, which is at least  $\beth_{2d-1}(\aleph_0)^+$,  so  that  $\kappa\ra(\aleph_1)^{2d}_{\aleph_0}$.
Given two sets of ordinals $J,K$ we shall write $J<K$ if every member of $J$ is less than every member of $K$.
Let $D_e=\{0,2,\dots,2d-2\}$ and  $D_o=\{1,3,\dots,2d-1\}$, the sets of  even and odd integers less than $2d$, respectively.
Let $\mathcal{I}$ denote the collection of all functions $\iota: 2d\ra 2d$ such that
$\iota\re D_e$
and $\iota\re D_o$ are strictly  increasing sequences
and $\{\iota(0),\iota(1)\}<\{\iota(2),\iota(3)\}<\dots<\{\iota(2d-2),\iota(2d-1)\}$.
Thus, each $\iota$ codes two strictly increasing sequences $\iota\re D_e$ and $\iota\re D_o$, each of length $d$.
For $\vec{\theta}\in[\kappa]^{2d}$,
$\iota(\vec{\theta}\,)$ determines the pair of sequences of ordinals $(\theta_{\iota(0)},\theta_{\iota(2)},\dots,\theta_{\iota(2d-2))}), (\theta_{\iota(1)},\theta_{\iota(3)},\dots,\theta_{\iota(2d-1)})$,
both of which are members of $[\kappa]^d$.
Denote these as $\iota_e(\vec\theta\,)$ and $\iota_o(\vec\theta\,)$, respectively.
To ease notation, let $\vec{\delta}_{\vec\al}$ denote $\vec\delta_{p_{\vec\al}}$,
 $k_{\vec{\al}}$ denote $|\vec{\delta}_{\vec\al}|$,
and let $l_{\vec{\al}}$ denote  $l_{p_{\vec\al}}$.
Let $\lgl \delta_{\vec{\al}}(j):j<k_{\vec{\al}}\rgl$
denote the enumeration of $\vec{\delta}_{\vec\al}$
in increasing order.

Define a coloring  $f$ on $[\kappa]^{2d}$ into countably many colors as follows:
Given  $\vec\theta\in[\kappa]^{2d}$ and
 $\iota\in\mathcal{I}$, to reduce the number of subscripts,  letting
$\vec\al$ denote $\iota_e(\vec\theta\,)$ and $\vec\beta$ denote $\iota_o(\vec\theta\,)$,
define
\begin{align}\label{eq.fiotatheta}
f(\iota,\vec\theta\,)&= \,
\lgl \iota, \varepsilon_{\vec{\al}}, k_{\vec{\al}}, p_{\vec{\al}}(d),
\lgl \lgl p_{\vec{\al}}(i,\delta_{\vec{\al}}(j)):j<k_{\vec{\al}}\rgl:i< d\rgl,\cr
& \lgl  \lgl i,j \rgl: i< d,\ j<k_{\vec{\al}},\ \mathrm{and\ } \delta_{\vec{\al}}(j)=\al_i \rgl,
\lgl \lgl j,k\rgl:j<k_{\vec{\al}},\ k<k_{\vec{\beta}},\ \delta_{\vec{\al}}(j)=\delta_{\vec{\beta}}(k)\rgl\rgl.
\end{align}
Let $f(\vec{\theta}\,)$ be the sequence $\lgl f(\iota,\vec\theta\,):\iota\in\mathcal{I}\rgl$, where $\mathcal{I}$ is given some fixed ordering.
Since the range of $f$ is countable,
apply the \Erdos-Rado Theorem
to  obtain a subset $K\sse\kappa$ of cardinality $\aleph_1$
which is homogeneous for $f$.
Take $K'\sse K$ such that between each two members of $K'$ there is a member of $K$ and $\min(K')>\min(K)$.
Take subsets $K_i\sse K'$ such that  $K_0<\dots<K_{d-1}$
and   each $|K_i|=\aleph_0$.

\begin{claim}\label{claim.onetypes}
There are $\varepsilon^*\in 2$, $k^*\in\om$, $t_d$,
and $ \lgl t_{i,j}: j<k^*\rgl$, $i< d$,
 such that
for all $\vec{\al}\in \prod_{i<d}K_i$ and  each $i< d$,
 $\varepsilon_{\vec{\al}}=\varepsilon^*$,
$k_{\vec\al}=k^*$,  $p_{\vec{\al}}(d)=t_d$, and
$\lgl p_{\vec\al}(i,\delta_{\vec\al}(j)):j<k_{\vec\al}\rgl
=
 \lgl t_{i,j}: j<k^*\rgl$.
\end{claim}

\begin{proof}
Let  $\iota$ be the member in $\mathcal{I}$
which is the identity function on $2d$.
For any pair $\vec{\al},\vec{\beta}\in \prod_{i<d}K_i$, there are $\vec\theta,\vec\theta'\in [K]^{2d}$
such that
$\vec\al=\iota_e(\vec\theta\,)$ and $\vec\beta=\iota_e(\vec\theta'\,)$.
Since $f(\iota,\vec\theta\,)=f(\iota,\vec\theta'\,)$,
it follows that $\varepsilon_{\vec\al}=\varepsilon_{\vec\beta}$, $k_{\vec{\al}}=k_{\vec{\beta}}$, $p_{\vec{\al}}(d)=p_{\vec{\beta}}(d)$,
and $\lgl \lgl p_{\vec{\al}}(i,\delta_{\vec{\al}}(j)):j<k_{\vec{\al}}\rgl:i< d\rgl
=
\lgl \lgl p_{\vec{\beta}}(i,\delta_{\vec{\beta}}(j)):j<k_{\vec{\beta}}\rgl:i< d\rgl$.
Thus, define  $\varepsilon^*$, $k^*$, $t_d$, $\lgl \lgl t_{i,j}:j<k^*\rgl:i<d\rgl$ to be
$\varepsilon_{\vec\al}$, $k_{\vec\al}$,
$p_{\vec\al}(d)$,
$\lgl \lgl p_{\vec{\al}}(i,\delta_{\vec{\al}}(j)):j<k_{\vec{\al}}\rgl:i< d\rgl$
 for any $\vec\al\in \prod_{i<d}K_i$.
\end{proof}

Let $l^*$ denote the length of $t_d$.
Then all the nodes  $t_{i,j}$,  $i< d$, $j<k^*$,   also  have  length $l^*$.

\begin{claim}\label{claim.j=j'}
Given any $\vec\al,\vec\beta\in \prod_{i<d}K_i$,
if $j,k<k^*$ and $\delta_{\vec\al}(j)=\delta_{\vec\beta}(k)$,
 then $j=k$.
\end{claim}

\begin{proof}
Let $\vec\al,\vec\beta$ be members of $\prod_{i<d}K_i$   and suppose that
 $\delta_{\vec\al}(j)=\delta_{\vec\beta}(k)$ for some $j,k<k^*$.
For each $i<d$, let  $\rho_i$ be the relation from among $\{<,=,>\}$ such that
 $\al_i\,\rho_i\,\beta_i$.
Let   $\iota$ be the member of  $\mathcal{I}$  such that for each $\vec\gamma\in[K]^{d}$ and each $i<d$,
$\theta_{\iota(2i)}\ \rho_i \ \theta_{\iota(2i+1)}$.
Then there is a
$\vec\theta\in[K']^{2d}$ such that
$\iota_e(\vec\theta)=\vec\al$ and $\iota_o(\vec\theta)= \vec\beta$.
Since between any two members of $K'$ there is a member of $K$, there is a
 $\vec\gamma\in[K]^{d}$ such that  for each $i< d$,
 $\al_i\,\rho_i\,\gamma_i$ and $\gamma_i\,\rho_i\, \beta_i$,
and furthermore, for each $i<d-1$,
$\{\al_i,\beta_i,\gamma_i\}<\{\al_{i+1},\beta_{i+1},\gamma_{i+1}\}$.
Given that  $\al_i\,\rho_i\,\gamma_i$ and $\gamma_i\,\rho_i\, \beta_i$ for each $i<d$,
there are  $\vec\mu,\vec\nu\in[K]^{2d}$ such that $\iota_e(\vec\mu)=\vec\al$,
$\iota_o(\vec\mu)=\vec\gamma$,
$\iota_e(\vec\nu)=\vec\gamma$, and $\iota_o(\vec\nu)=\vec\beta$.
Since $\delta_{\vec\al}(j)=\delta_{\vec\beta}(k)$,
the pair $\lgl j,k\rgl$ is in the last sequence in  $f(\iota,\vec\theta)$.
Since $f(\iota,\vec\mu)=f(\iota,\vec\nu)=f(\iota,\vec\theta)$,
also $\lgl j,k\rgl$ is in the last  sequence in  $f(\iota,\vec\mu)$ and $f(\iota,\vec\nu)$.
It follows that $\delta_{\vec\al}(j)=\delta_{\vec\gamma}(k)$ and $\delta_{\vec\gamma}(j)=\delta_{\vec\beta}(k)$.
Hence, $\delta_{\vec\gamma}(j)=\delta_{\vec\gamma}(k)$,
and therefore $j$ must equal $k$.
\end{proof}

For any $\vec\al\in \prod_{i<d}K_i$ and any $\iota\in\mathcal{I}$, there is a $\vec\theta\in[K]^{2d}$ such that $\vec\al=\iota_o(\vec\theta)$.
By homogeneity of $f$ and  by the first sequence in the second line of equation  (\ref{eq.fiotatheta}), there is a strictly increasing sequence
$\lgl j_i:i< d\rgl$  of members of $k^*$ such that for each $\vec\al\in \prod_{i<d}K_i$,
$\delta_{\vec\al}(j_i)=\al_i$.
For each $i< d$, let $t^*_i$ denote $t_{i,j_i}$.
Then  for each $i<d$ and each $\vec\al\in \prod_{i<d}K_i$,
\begin{equation}
p_{\vec\al}(i,\al_i)=p_{\vec{\al}}(i, \delta_{\vec\al}(j_i))=t_{i,j_i}=t^*_i.
\end{equation}
Let $t_d^*$ denote $t_d$.

\begin{lem}\label{lem.compat}
For any finite subset $\vec{J}\sse \prod_{i<d}K_i$,
the set of conditions $\{p_{\vec{\al}}:\vec{\al}\in \vec{J}\,\}$ is  compatible.
Moreover,
$p_{\vec{J}}:=\bigcup\{p_{\vec{\al}}:\vec{\al}\in \vec{J}\,\}$
is a member of $\bP$ which is below each
$p_{\vec{\al}}$, $\vec\al\in\vec{J}$.
\end{lem}

\begin{proof}
For any $\vec\al,\vec\beta\in \prod_{i<d}K_i$,
whenver
 $j,k<k^*$ and
 $\delta_{\vec\al}(j)=\delta_{\vec\beta}(k)$, then $j=k$, by Claim  \ref{claim.j=j'}.
It then follows from Claim \ref{claim.onetypes}
that for each $i<d$,
\begin{equation}
p_{\vec\al}(i,\delta_{\vec\al}(j))=t_{i,j}=p_{\vec\beta}(i,\delta_{\vec\beta}(j))
=p_{\vec\beta}(i,\delta_{\vec\beta}(k)).
\end{equation}
Thus, for each $\vec\al,\vec\beta\in\vec{J}$ and each
$\delta\in\vec{\delta}_{\vec\al}\cap
\vec{\delta}_{\vec\beta}$,
for all $i<d$,
\begin{equation}
p_{\vec\al}(i,\delta)=p_{\vec\beta}(i,\delta).
\end{equation}
Thus,
$p_{\vec{J}}:
\bigcup \{p_{\vec{\al}}:\vec\al\in\vec{J}\}$
is a  function.
Let $\vec\delta_{\vec{J}}=
\bigcup\{
\vec{\delta}_{\vec\al}:
\vec\al\in\vec{J}\,\}$.
For each $\delta\in
\vec{\delta}_{\vec{J}}$ and   $i<d$,
$p_{\vec{J}}(i,\delta)$ is defined,
and it is exactly  $p_{\vec\al}(i,\delta)$, for any $\vec\al\in\vec{J}$ such that $\delta\in \vec\delta_{\vec\al}$.
Thus, $p_{\vec{J}}$ is a member of $\bP$, and $p_{\vec{J}}\le p_{\vec\al}$ for each $\vec\al\in\vec{J}$.
\end{proof}

We conclude this section with a general claim which will be useful in Part III.

\begin{claim}\label{subclaimA}
If $\beta\in \bigcup_{i<d}K_i$,
$\vec{\al}\in\prod_{i<d}K_i$,
and $\beta\not\in\vec\al$,
 then
$\beta$ is not  a member of   $\vec{\delta}_{\vec{\al}}$.
\end{claim}

\begin{proof}
Suppose toward a contradiction that $\beta\in\vec{\delta}_{\vec{\al}}$.
Then there is a $j<k^*$ such that $\beta=\delta_{\vec{\al}}(j)$.
Let $i$ be such that $\beta\in K_i$.
Since $\beta\ne\al_i=\delta_{\vec{\al}}(j_i)$, it must be that $j\ne j_i$.
However,
letting $\vec\beta$ be any member of $\prod_{i<d}K_i$ with $\beta_i=\beta$,
then
$\beta=\delta_{\vec{\beta}}(j_i)=\delta_{\vec{\al}}(j)$, so Claim \ref{claim.j=j'}
implies that $j_i=j$, a contradiction.
\end{proof}
\vskip.1in

\noindent \underline{\bf{Part III.}}
In this last part of the proof,
we build a strong coding tree $S$ valid in $T$ on which the coloring $h$ is homogeneous.
Cases (a) and  (b) are now handled separately.
\vskip.1in

\noindent\underline{\bf{Part III Case (a).}}
Recall that $\{s_i:i\le d\}$ enumerates the members of $A_e$, which is a subset of $B^+$.
Let $s_d^-$ denote $s_d\re l_A$,
and let $i_d\in\{0,1\}$ be such that $s_d={s_d^-}^{\frown}i_d$.
Let $m'$ be the   integer such that $B\in\mathcal{A}_{m'}(T)$.
Let $\sigma$
denote the strong similarity map from $B$ onto $r_{m'}(\bT)$, and
let $M=\{ m_j:j<\om\}$ be the strictly increasing enumeration of those $m> m'$
 such that the splitting node in  $\max(r_m(\bT))$ extends ${\sigma(s_d^-)}^{\frown}i_d$.
We will find $U_{m_0}\in r_{m_0}[B,T]$
and in general, $U_{m_{j+1}}\in r_{m_{j+1}}[U_{m_j},T]$
 so that  for each $j<\om$,
  $h$ takes color $\varepsilon^*$ on $\Ext_{U_{m_j}}(A,C)$.
Then setting $S=\bigcup_{j<\om} U_{m_j}$ will yield $S$ to be a member of $[B,T]$ for which  $\Ext_S(A,C)$ is homogeneous for $h$, with color $\varepsilon^*$.

First extend each node in $B^+$  to level $l^*$ as follows.
Recall that for each $i\le d$,
$t_i^*\contains t_i$,
so the set $\{t^*_i:i\le d\}$ extends
 $A_e$.
For  each node $u$ in $B^+\setminus A_e$, let $u^*$ denote
 its  leftmost extension in $T\re l^*$.
Then the set
\begin{equation}
U^*=\{t^*_i:i\le d\}\cup\{u^*:u\in B^+\setminus A_e\}
\end{equation}
extends each member of  $B^+$ to a unique node.
Furthermore, by the choice of $p^0_{\vec\al}$ for each $\al\in [K]^d$ and the definition of the partial ordering on $\bP$,
it follows that
the set $\{t_i^*:i\le d\}$ has no new sets of parallel $1$'s over $A_e$.
Since the nodes $u^*$ are leftmost extensions of  members of $B^+\setminus A_e$ and $B$ is valid in $T$,
it follows from Lemma \ref{lem.poc} that $U^*$ has no new sets of parallel $1$'s over $B$.
Furthermore, $U^*$ has no pre-determined new sets of parallel $1$'s, by (iii) in  the definition of the partial ordering $\bP$ for Case (a).
Thus,
 $B\cup U^*$  satisfies the \POC\ and   is valid in $T$.
If $m_0=m'+1$,
then  let $U_{m'+1}=B\cup U^*$ and
 extend $U_{m'+1}$ to a member $U_{m_1-1}\in r_{m_1-1}[U_{m'+1},T]$.
If $m_0>m'+1$,
apply Lemma \ref{lem.facts} to
extend above $U^*$  to construct
a member $U_{m_0-1}\in r_{m_0-1}[B,T]$.
In this case, $\max(r_{m'+1}(U_{m_0}))$ is not
$ U^*$, but rather  $\max(r_{m'+1}(U_{m_0}))$
extends $U^*$.

Assume  $j<\om$ and
 we have constructed $U_{m_j-1}$ so that every member of $\Ext_{U_{m_j-1}}(A,C)$ is colored $\varepsilon^*$ by $h$.
Fix some  $Y_j\in r_{m_j}[U_{m_j -1} ,T]$ and let $V_j$ denote $\max(Y_j)$.
The nodes in $V_j$ will not be in the tree $S$ we are constructing;
rather,
we will extend the nodes in $V_j$ to construct
$U_{m_j}\in r_{m_j}[U_{m_j-1},T]$.

We now start to construct a condition $q$ which will satisfy Claim \ref{claim.qbelowpal}.
Let $q(d)$ denote  the splitting node in $V_j$ and let $l_q=|q(d)|$.
For each $i<d$
for which $s_i$ and $s_d$ do not have parallel $1$'s,
 let
$Z_i$ denote the set of all
$v\in T_i\cap V_j$
such that $v$ and $q(d)$ have no parallel $1$'s.
For each $i<d$ for which $s_i$ and $s_d$ do have parallel $1$'s,
let $Z_i= T_i\cap V_j$.
For each $i<d$,
take  a set $J_i\sse K_i$ of cardinality $|Z_i|$
and label the members of $Z_i$ as
$\{z_{\al}:\al\in J_i\}$.
Notice that each member of $\Ext_T(A,C)$ above $V_j$ extends some set $\{z_{\al_i}:i<d\}\cup\{q(d)\}$, where each $\al_i\in J_i$.
Let $\vec{J}$ denote the set of those $\lgl \al_0,\dots,\al_{d-1}\rgl\in \prod_{i< d}J_i$ such that  the set $\{z_{\al_i}:i< d\}\cup\{q(d)\}$ is in $\Ext_T(A,C)$.
Notice that for each $i<d$,
$J_i=\{\al_i:\vec\al\in\vec{J}\}$, since
each node in $Z_i$ is
 in some member of $\Ext_T(A,C)$:
 Extending all the other $t_j^*$ ($j\ne i$) via their leftmost extensions in $T$ to length $l_q$, along with $q(d)$, constructs a member of $\Ext_T(A,C)$.
By Lemma \ref{lem.compat},
the set $\{p_{\vec\al}:\vec\al\in\vec{J}\}$ is compatible.
The fact that
$p_{\vec{J}}$ is a condition in $\bP$ will be used
to make  the construction of $q$ very precise.

Let
 $\vec{\delta}_q=\bigcup\{\vec{\delta}_{\vec\al}:\vec\al\in \vec{J}\}$.
For each $i<d$ and $\al\in J_i$,
define $q(i,\al)=z_{\al}$.
Notice that for each
$\vec\al\in \vec{J}$ and $i<d$,
\begin{equation}
q(i,\al_i)\contains t^*_i=p_{\vec\al}(i,\al_i)=p_{\vec{J}}(i,\al_i),
\end{equation}
and
\begin{equation}
q(d)\contains t^*_d=p_{\vec\al}(d)=p_{\vec{J}}(d).
\end{equation}
For each  $i<d$ and $\gamma\in\vec{\delta}_q\setminus
J_i$,
there is at least one $\vec{\al}\in\vec{J}$ and some $k<k^*$ such that $\delta_{\vec\al}(k)=\gamma$.
Let $q(i,\gamma)$ be the leftmost extension
 of $p_{\vec{J}}(i,\gamma)$ in $T$ of length $l_q$.
Define
\begin{equation}
q=\{q(d)\}\cup \{\lgl (i,\delta),q(i,\delta)\rgl: i<d,\  \delta\in \vec{\delta}_q\}.
\end{equation}

\begin{claim}\label{claim.qbelowpal}
For all $\vec\al\in\vec{J}$,
$q\le p_{\vec{\al}}$.
\end{claim}

\begin{proof}
Given  $\vec\al\in\vec{J}$,
it follows from the  definition of $q$ that
$\vec{\delta}_q\contains \vec{\delta}_{\vec{\al}}$,
$q(d)\contains p_{\vec{\al}}(d)$,
and
for each pair $(i,\gamma)\in d\times \vec{\delta}_{\vec\al}$,
$q(i,\gamma)\contains p_{\vec{\al}}(i,\gamma)$.
So it only remains to show that  $q$
has no new sets of parallel $1$'s over $p_{\vec{\al}}$.
It follows from  Claim \ref{subclaimA}
that
 $\vec{\delta}_{\vec\al}\cap
\bigcup_{i<d}K_i=\vec\al$.
Hence,
for  each $i<d$ and $\gamma\in\vec{\delta}_{\vec\al}\setminus \{\al_i\}$,
 $q(i,\gamma)$ is the leftmost extension of $p_{\vec\al}(i,\gamma)$.
Since $\vec\al$ is in $\vec{J}$,
$\{q(i,\al_i):i<d\}\cup\{q(d)\}$ is in $\Ext_T(A,C)$
by definition of $\vec{J}$.
This  implies that $\{q(i,\al_i):i<d\}\cup\{q(d)\}$
has no new parallel $1$'s over $A$, as this set union $A$ must be strongly similar to $C$ which satisfies the \POC, and since the critical node in $C\setminus A$ is a splitting node, $C\setminus A$ has no new parallel $1$'s over $A$.
It follows that
$\{q(i,\delta):(i,\delta)\in d\times\delta\in\vec{\delta}_{\vec\al}\}\cup\{q(d)\}$
has no new  parallel $1$'s  over
$\{p_{\vec\al}(i,\delta):(i,\delta)\in d\times\delta\in\vec{\delta}_{\vec\al}\}\cup\{p_{\vec\al}(d)\}$.
Therefore, $q\le p_{\vec\al}$.
 \end{proof}

\begin{rem}
Notice that we did not prove that $q\le p_{\vec{J}}$.
That  will be blatantly false for  all  large enough $j$,
as   the union of the sets $Z_i$, $i<d$,  composed from $V_j$  will have   many new sets of parallel $1$'s over $p_{\vec{J}}$.
This is one fundamental difference between the forcings being used for this theorem
and the  forcings adding $\kappa$ many Cohen reals used in Harrington's  proof of the Halpern-\Lauchli\ Theorem.
\end{rem}

To construct $U_{m_j}$,
take an $r\le q$ in  $\bP$ which  decides some $l_j$ in $\dot{L}_d$ for which   $h(\dot{b}_{\vec\al}\re l_j)=\varepsilon^*$, for all $\vec\al\in\vec{J}$.
This is possible since for all $\vec\al\in\vec{J}$,
$p_{\vec\al}$ forces $h(\dot{b}_{\vec\al}\re l)=\varepsilon^*$ for $\dot{\mathcal{U}}$ many $l\in \dot{L}_d$.
Without loss of generality, we may assume that
 the nodes in the image of $r$ have length  $l_j$.
Notice that since
$r$ forces $\dot{b}_{\vec{\al}}\re l_j=\{r(i,\al_i):i<d\}\cup\{r(d)\}$ for each $\vec\al\in \vec{J}$,
and since the coloring $h$ is defined in the ground model,
it is simply true in the ground model that
$h(\{r(i,\al_i):i<d\}\cup\{r(d)\})=\varepsilon^*$ for each $\vec\al\in \vec{J}$.
Extend the splitting node $q(d)$ in $V_j$
to $r(d)$.
For each $i<d$ and $\al_i\in J_i$,
extend $q(i,\al_i)$ to $r(i,\al_i)$.
Let $V_j^-$ denote  $V_j\setminus(\{q(i,\al_i):i<d$, $\al_i\in J_i\}\cup \{q(d)\})$.
For each node $v$ in $V_j^-$,
let $v^*$ be the  leftmost extension of $v$ in $T\re l_j$.
Let
\begin{equation}
U_{m_j}= U_{m_j-1}\cup\{r(d)\}\cup\{r(i,\al_i):i<d,\ \al_i\in J_i\}\cup\{v^*:v\in V_j^-\}.
\end{equation}

\begin{claim}\label{claim.correct}
$U_{m_j}\in r_{m_j}[U_{m_j-1},T]$ and
every $X\in\Ext_{U_{m_j}}(A,C)$ with $\max(X)\sse\max(U_{m_j})$
 satisfies $h(X)=\varepsilon^*$.
\end{claim}

\begin{proof}
Recall that $U_{m_j-1}\sqsubset Y_j$ are both valid in $T$.
Since $r\le q$, it follows that
$\{r(i,\delta):(i,\delta)\in d\times\vec{\delta}_q\}\cup\{r(d)\}$
has no new  sets of parallel $1$'s  over
$\{q(i,\delta):(i,\delta)\in d\times\vec{\delta}_q\}\cup\{q(d)\}$, which is a subset of $V_j$.
All other nodes in $\max(U_{m_j})$ are leftmost extensions of nodes in $V_j$.
Thus,
$\max(U_{m_j})$  extends $V_j$ and has
no new sets of parallel $1$'s over  $V_j$,
so
 $U_{m_j}\ssim r_{m_j}(\bT)$.
Further, $\max(U_{m_j})$
has no pre-determined new parallel $1$'s since $r\in \bP$.
It follows that $U_{m_j}\in r_{m_j}[U_{m_j-1},T]$.

For each $X\in\Ext_{U_{m_j}}(A,C)$ with $X\sse\max(U_{m_j})$,
the truncation  $A\cup\{x\re l_q:x\in X\}$
is a member of $\Ext_{Y_j}(A,C)$.
Thus, there corresponds a sequence $\vec\al\in\vec{J}$ such that
$\{x\re l_q:x\in X\}=\{q(i,\al_i):i<d\}\cup\{q(d)\}$.
Then
 $X=\{r(i,\al_i):i<d\}\cup \{r(d)\}$,
 which  has $h$-color $\varepsilon^*$.
\end{proof}

Let $S=\bigcup_{j<\om}U_{m_j}$.
For each $X\in\Ext_{S}(A,C)$, there corresponds a $j<\om$ such that $X\in\Ext_{U_{m_j}}(A,C)$ and $X\sse\max(U_{m_j})$.
By Claim \ref{claim.correct}, $h(X)=\varepsilon^*$.
Thus, $S\in [B,T]$ and satisfies the theorem.
This concludes the proof of the theorem for Case (a).
\vskip.1in

\noindent\underline{\bf{Part III Case (b).}}
Let $X\in  \Ext_T(A,C)$ and $m_0$ be given such that there is a $B'\in r_{m_0}[B,T]$ with
$X\sse \max(B')$.
Let $U_{m_0-1}$ denote $r_{m_0-1}(B')$.
We will  build an $S\in [U_{m_0-1},T]$ such that every member of $\Ext_S(A,C;X)$  has the same $h$-color.
Let $n_{B'}$ be the index such that $c^T_{n_{B'}}$ is the  coding node in $\max(B')$.
Label the members of $X$ as $x_i$, $i\le d$, so that each $x_i\contains s_i$.
For Case (b),  back in Part II, when choosing the $p_{\vec\al}$, $\vec\al\in[\kappa]^d$,
first
define
\begin{equation}
p_{\vec\al}^0=\{\lgl (i,\delta),x_i\rgl:i<d,\ \delta\in\vec\al\}\cup\{\lgl d,x_d\rgl\},
\end{equation}
so that each node
 $t_i^*$ will  extend $ x_i$, for  $i\le d$.
Then choose  $p_{\vec\al}^k$, $1\le k\le 3$, as before,
 with the additional requirement that
$p_{\vec\al}(d)=c^T_{n}$ for some $n\ge n_{B'}+3$.
Everything else in Part II remains the same.

We will build $U_{m_0}\in r_{m_0}[U_{m_0-1},T]$ so that its maximal members extend $\max(B')$, and hence each member of $X$ is extended uniquely in $\max(U_{m_0})$.
Let $V_0$ denote $\max(B')$.
Let $V_0^l$ and $V_0^r$ denote those members $v$  of  $V_0$ such that
the immediate extension of
$v$   is $0$ or $1$, respectively.
 For each $v\in V^r_0\setminus X$,
$v$ has no parallel $1$'s with $x_d$,
so
the Passing Number Choice Lemma \ref{lem.pnc}
guarantees that
there is a member $v^*$
extending $v$  to length $l^*:=|t^*_d|\ge l^T_{n+3}$ such that $v^*$ has immediate
successor $1$ in $T$.
For each
$v\in V^l_0\setminus X$,
take
$v^*$ to be  the leftmost extension of $v$ of length $l^*$.
Let
\begin{equation}
V^*=\{t^*_i:i\le d\}\cup\{v^*:v\in  V_0\setminus X\}.
\end{equation}

\begin{claim}\label{claim.begin}
$U_{m_0-1}\cup V^*$ is a member of $r_{m_0}[U_{m_0-1}, T]$.
\end{claim}

\begin{proof}
By the construction,  $V^*$ extends $V_0$,
and
 for each  $z\in V^*$,
the passing
number of $z$ at $t^*_d$ is equal to the passing number of
$z\re l_{B'}$  at $c^T_{n}$.
Thus,
it will follow that
$U_{m_0-1}\cup V^*\ssim B'$
once we prove that
$U_{m_0-1}\cup V^*$ satisfies the \POC.

Let $Y$ be  any subset of $ V^*$ for which there is an $l$ such that $y(l)=1$ for all $y\in Y$.
Since for each $\vec\al\in [K]^d$, $p_{\vec\al}\le p_{\vec\al}^0$,
it follows that
$\{t^*_i:i\le d\}$ has no new sets of parallel $1$'s over $X$.
It follows that
 if $Y\sse \{t^*_i:i\le d\}$, then the parallel $1$'s of $Y$ are either witnessed in $U_{m_0-1}$ or else are witnessed by the coding node in $X$, and hence by $t^*_d$.
In particular, the parallel $1$'s of $Y$ are witnessed in $U_{m_0-1}\cup V^*$.

If $Y$  contains  $v^*$ for some $v\in V^l_0\setminus X$,
then there must be an $l'< |x_d|$
where this set of parallel $1$'s is first witnessed,
as $v^*$ is the leftmost extension of $v$ in $T\re l^*$ and therefore
any coding node of $T$ where $v^*$ has passing number $1$ must have length less than $|x_d|$.
Since $U_{m_0-1}$ satisfies the \POC,
the set of parallel $1$'s in $Y$ is witnessed by a coding node  in $U_{m_0-1}$.

Now suppose that $Y\sse \{v^*:v\in V^r_0\setminus X\}\cup\{t^*_i:i\le d\}$.
If $Y\cap \{t^*_i:i\le d\}$ is contained in $\{t^*_i:i\in I_1\}$,
then $t^*_d$ witnesses the parallel $1$'s in $Y$.
Otherwise, there is some $t^*_i\in Y$ with $i\in I_0$.
Note that $t^*_i$ has immediate extension $0$ at $t^*_d$, and so in the interval in $T$ with $t^*_d$,
$t^*_i$ takes the leftmost path;
 also $t^*_i(|x_d|)=0$.
By the construction in the proof of Lemma \ref{lem.pnc},
all $v^*$ for $v\in V^r_0$
extend $v$ leftmost until the interval of $T$ containing the coding node  $t^*_d$.
Hence, any parallel $1$'s between such $v^*$ and $t^*_i$ must occur at a level below $|x_d|$.
Thus, the parallel $1$'s in $Y$ must first appear in $U_{m_0-1}$, and hence be witnessed by
 some coding node in $U_{m_0-1}$.

Therefore, $U_{m_0-1}\cup V^*$ satisfies the \POC,
and hence
  $U_{m_0-1}\cup V^*\in r_{m_0}[U_{m_0-1},T]$.
\end{proof}

Define $U_{m_0}=U_{m_0-1}\cup V^*$.
Let $M=\{m_j:j<\om\}$ enumerate  the set of $m\ge m_0$
such that the coding node $c^{\bT}_{m}\contains c^{\bT}_{m_0}$.
By strong similarity of $T$ with $\bT$, for any $S\in[U_{m_0}, T]$,
the  coding node $c^S_m$ will  extend $t^*_d$ if and only if $m\in M$.
Take any $U_{m_1-1}\in r_{m_1-1}[U_{m_0},T]$.
Notice that $\{t^*_i:i\le d\}$ is the only member of $\Ext_{U_{m_1-1}}(A,C;X)$,
and it has $h$-color $\varepsilon^*$.

Assume  now that  $1\le j<\om$ and
 we have constructed $U_{m_j-1}$ so that every member of $\Ext_{U_{m_j-1}}(A,C;X)$ is colored $\varepsilon^*$ by $h$.
Fix some  $Y_j\in r_{m_j}[U_{m_j-1} ,T]$.
Let $V_j$ denote $\max(Y_j)$.
The nodes in $V_j$ will not be in the tree $S$ we are constructing;
rather,
we will construct
$U_{m_j}\in r_{m_j}[U_{m_j-1},T]$
so that
$\max(U_{m_j})$ extends
 $V_j$.
Let $q(d)$ denote the coding   node in $V_j$ and let $l_q=|q(d)|$.
Recall that for $k\in\{0,1\}$, $I_k$ denotes the set of
 $i<d$ for which $t^*_i$ has passing number $k$ at $t^*_d$.
For each $k\in\{0,1\}$ and
each
$i\in I_k$, let
$Z_i$ be the set
of nodes $z$ in $T_i\cap V_j$ such that $z$ has passing number $k$ at  $q(d)$.

We now construct a condition $q$ similarly, but not exactly, as in   Case (a).
For each $i<d$,
let $J_i$ be a subset of $K_i$ with the same size as $Z_i$.
For each $i< d$, label the nodes in $Z_i$ as
$\{z_{\al}:\al\in J_i\}$.
Let $\vec{J}$ denote the set of those $\lgl \al_0,\dots,\al_{d-1}\rgl\in \prod_{i< d}J_i$ such that  the set
$\{z_{\al_i}:i< d\}\cup\{q(d)\}$ is in $\Ext_T(A,C)$.
Notice that for each $i<d$ and
$\vec\al\in \vec{J}$, $z_{\al_i}\contains t^*_i=p_{\vec{\al}}(i,\al_i)$, and $q(d)\contains t^*_d=p_{\vec{\al}}(d)$.
Furthermore, for each $i<d$ and   $\delta\in J_i$,
there is an $\vec\al\in\vec{J}$ such that $\al_i=\delta$.
Let
 $\vec{\delta}_q=\bigcup\{\vec{\delta}_{\vec\al}:\vec\al\in \vec{J}\,\}$.
For each pair $(i,\gamma)\in d\times\vec{\delta}_q$ with  $\gamma\in J_i$,
define $q(i,\gamma)=z_{\gamma}$.
For  each pair $(i,\gamma)\in d\times\vec{\delta}_q$
with  $\gamma\in\vec{\delta}_q\setminus
J_i$,
there is at least one $\vec{\al}\in\vec{J}$ and some $k<k^*$ such that $\delta_{\vec\al}(k)=\gamma$.
By Lemma \ref{lem.compat},
$p_{\vec\beta}(i,\gamma)=p_{\vec{\al}}(i,\gamma)=t^*_{i,k}$,
for any  $\vec\beta\in\vec{J}$ for which $\gamma\in\vec{\delta}_{\vec\beta}$.
For $i\in I_0$,
let $q(i,\gamma)$ be the leftmost extension
 of $t_{i,k}^*$ in $T$ to length $l_q$.
This will have passing number $0$ at $q(d)$, and any parallel $1$'s between this node and any other nodes in $V_j$ must be witnessed at or below $t^*_d$.
For $i\in I_1$, let $q(i,\gamma)$ be the extension of $t_{i,k}^*$
as in Lemma \ref{lem.pnc}:
extend $t_{i,k}^*$  leftmost  in $T$ until the  interval  of $T$  containing $q(d)$; in that interval,  extend to the next splitting node
 and take the right branch  of length $l_q$.
Let this node be $q(i,\gamma)$.
This has  passing number $1$ at $q(d)$, and
any parallel $1$'s between  $q(i,\gamma)$ and another node must be either witnessed by $q(d)$ or else at or  below $t^*_d$.
Define
\begin{equation}
q=\{q(d)\}\cup \{\lgl (i,\delta),q(i,\delta)\rgl: i<d,\  \delta\in \vec{\delta}_q\}.
\end{equation}
By the construction, $q$ is a member of $\bP$.

\begin{claim}\label{claim.qbelowpal}
For each $\vec\al\in \vec{J}$,
$q\le p_{\vec\al}$.
\end{claim}

\begin{proof}
Let $n$ denote  the  index   such that $c^T_n=q(d)$.
It suffices to show that for each $\vec\al\in\vec{J}$,
$q$ has no new sets of parallel $1$'s over $p_{\vec\al}$,
 since by construction, we have that
$q(i,\delta)\contains p_{\vec{\al}}(i,\delta)$ for all $(i,\delta)\in d\times \vec{\delta}_{\vec\al}$.

Let $\vec\al\in\vec{J}$ be given,
 and let $Y$ be any subset of $\{q(i,\delta):(i,\delta)\in d\times \vec{\delta}_{\vec\al}\}$  of size at least $2$ for which for some $l$, $y(l)=1$ for all $y\in Y$.
If $Y\sse  \{q(i,\al_i):i<d\}\cup\{q(d)\}$,
then $Y$ has no new parallel $1$'s over $X$, since
$\vec\al\in \vec{J}$ implies that  $\{q(i,\al_i):i<d\}\cup\{q(d)\}$ is in $\Ext_T(A,C;X)$.
Since $\{p(i,\al_i):i<d\}\cup\{p(d)\}$ extends $X$ and
 $Y$ consists of extensions of members of $\{p(i,\al_i):i<d\}\cup\{p(d)\}$,
it follows that
$Y$ has no new parallel $1$'s over $\{p(i,\al_i):i<d\}\cup\{p(d)\}$.

Now suppose $Y$ contains some $q(i,\delta)$, where $\delta\in \vec{\delta}_{\vec\al}\setminus \{\al_i\}$.
Recall that by Claim \ref{subclaimA},
$\vec{\delta}_{\vec\al}\cap (\bigcup_{i<d}K_i)=\vec\al$;
so in particular,  $\delta\not\in \bigcup_{i<d}J_i$.
By construction of $q$,
if  $i\in I_0$, then
$q(i,\delta)$ has no new parallel $1$'s above $l^*$ with any other $q(j,\gamma)$,
$(j,\gamma)\in d\times
 \vec{\delta}_{\vec\al}$.
If $i\in I_1$,
it follows from the construction of $q$ that
 any parallel $1$'s $q(i,\delta)$ has with another member of $\ran(q)$ below $l_{n-1}^T$
is witnessed below $l^*$.
Further, any parallel $1$'s $q(i,\delta)$ has in the interval $(l^T_{n-1},l^T_n]$  are witnessed by the coding node $q(d)$.
Thus, any
 new sets of parallel $1$'s in $Y$ occurring above length $l^*$ must be witnessed by $q(d)$.
Therefore, $q$ has no new parallel $1$'s over $p_{\vec\al}$, and hence,
$q\le p_{\vec\al}$.
\end{proof}

To construct $U_{m_j}$,
we will  extend each node in $V_j$ uniquely in such a manner   so that these extensions along with $U_{m_j-1}$ form a member of  $r_{m_j}[U_{m_j-1},T]$.
It suffices to find some $V^*$ extending $V_j$
such that the coding node in $V^*$ extends the coding node in $V_j$,
the passing number of each $v^*$ in $V^*$ extending some $v$ in $V_j$ is the same as the passing number of $v$ in $V_j$,
and
 no new sets of parallel $1$'s occur in $V^*$ over $V_j$.
Then $U_{m_j-1}\cup V^*$ will be  strongly similar to $r_{m_j}(\bT)$ and hence a member of $r_{m_j}[U_{m_j-1},T]$.

Take an $r\le q$ in  $\bP$ which  decides some $l_j$ in $\dot{L}_d$ such that  $h(\dot{b}_{\vec\al}\re l_j)=\varepsilon^*$ for all $\vec\al\in\vec{J}$, and such that there are at least two coding nodes in $T$ of lengths between $l_q$ and $l_r$.
Without loss of generality, we may assume that
 the nodes in the image of $r$ have length  $l_j$.
Extend the coding  node $q(d)$ in $V_j$
to $r(d)$.
For each $i<d$ and $\delta\in J_i$,
extend $q(i,\delta)$ to $r(i,\delta)$.
Let $V_j^l$ and $V_j^r$  denote the set of those
$v\in V_j$ with passing number $0$ and $1$, respectively, at $q(d)$.
Extend these nodes  according to the construction of Lemma \ref{lem.pnc} as follows:
For each node $v$ in $V_j^l\setminus(\{q(i,\delta):i<d$, $\delta\in J_i\}\cup \{q(d)\})$,
let $v^*$ be the  leftmost extension of $v$ in $T\re l_j$.
For each node $v$ in $V_j^r\setminus(\{q(i,\delta):i<d$, $\delta\in J_i\}\cup \{q(d)\})$,
extend $v$ leftmost to $v'$
of length $l^T_{n(r)-1}$, and
then let $v^*$ be the right extension of $\splitpred_T(v')$ to  length $l_r$,
where $n(r)$ is the index such that $c^T_{n(r)}=r(d)$.
Then each member of $V_j^l$ has passing number $0$ at $r(d)$ and each member of $V_j^r$ has passing number $1$ at $r(d)$.
Let
 $V_j^-$ denote $V_j\setminus(\{q(i,\delta):i<d$, $\delta\in J_i\}\cup \{q(d)\})$,
and define
\begin{equation}
V^*=\{r(d)\}\cup\{r(i,\al_i):i<d,\ \al_i\in J_i\}\cup\{v^*:v\in V_j^-\}
\end{equation}
and
\begin{equation}
U_{m_j}=U_{m_j-1}\cup V^*.
\end{equation}

\begin{claim}\label{claim.correct}
$U_{m_j}$ is a member of  $r_{m_j}[U_{m_j-1},T]$, and
$h(Y)=\varepsilon^*$
for each  $Y\in\Ext_{U_{m_j}}(A,C;X)$.
\end{claim}

\begin{proof}
By the construction of   $V^*$,
for each $v\in V_j$,
its extension $v^*$ in $V^*$ has the same passing number at $r(d)$ as $v$ does at $q(d)$.
Since $r\le q$,
all parallel $1$'s in $\{r(i,\delta):i<d,\ \delta\in J_i\}\cup\{r(d)\}$ are already witnessed in $V_j$.
Each $v$ in
$V_j^l\setminus(\{q(i,\delta):i<d$, $\delta\in J_i\}\cup \{q(d)\})$
has extension $v^*$ which has no new parallel $1$'s with any other member of $V^*$ above $l_q$.
Any set $Y\sse V_j^r\cup \{q(i,\delta):i<d$, $\delta\in J_i\}\cup \{q(d)\}$
cannot have new parallel $1$'s in the interval
$(l^*, l_{n(r)-1}]$, since for each
$v\in V_j^r\setminus (\{q(i,\delta):i<d$, $\delta\in J_i\}\cup \{q(d)\})$,
$v^*\re l_{n(r)-1}$ is the leftmost extension of $v$ in $T$ of  length  $l_{n(r)-1}$.
In the interval  $(l^*, l_{n(r)-1}]$,
Lemma \ref{lem.poc} implies the only new sets of  of parallel $1$'s  in $Y$  must be witnessed by $r(d)$.

Thus,
any sets of parallel $1$'s among $V^*$
are already witnessed in $V_j$.
Therefore,  $U_{m_j-1}\cup V^*$
satisfies the \POC\ and is strongly similar to $Y_j$,
and hence is in $r_{m_j}[U_{m_j-1},T]$.

Now suppose  $Z\sse V^*$ is a member of
$\Ext_{U_{m_j}}(A,C;X)$.
Then $Z\re l_q$ is in $\Ext_{T}(A,C;X)$,
so $Z$
extends $\{q(i,\al_i):i<d\}\cup \{q(d)\}$ for some $\vec\al\in \vec{J}$.
Thus, $Z=\{r(i,\al_i):i<d\}\cup \{r(d)\}$ for  that  $\vec\al$,
and  $r$ forces that $h(Z)=\varepsilon^*$.
Since $h$ and $Z$ are finite, they are  in the ground model, so
$h(Z)$ simply equals $\varepsilon^*$.
\end{proof}

To finish the proof of the theorem for Case (b),
Define  $S=\bigcup_{j<\om}U_{m_j}$.
Then $S\in [B',T]$, and
for each $Z\in\Ext_{S}(A,C;X)$, there is a $j<\om$ such that $Z\in\Ext_{U_{m_j}}(A,C)$ and
each member of
 $\max(U_{m_j})$ extending $X$  has $h$-color
 $\varepsilon^*$.

This concludes the proof of the theorem.
\end{proof}


\section{Ramsey Theorem for  finite  trees satisfying the \SPOC}\label{sec.1SPOC}

Our first   Ramsey theorem for  colorings of   finite subtrees  of a strong coding tree appears in this section.
Theorem \ref{thm.MillikenIPOC}, proves
that for any finite coloring of  the copies of a given finite tree satisfying the  \SPOC\  (Definition \ref{defn.strPOC})  in a strong coding tree $T$,
there is a strong coding tree $S\le T$ in which all strictly similar (Definition \ref{defn.strictly.similar}) copies have the same color.

Let $A$ be a   subtree of  a strong coding tree $T$.
Given  $l<\om$, define
\begin{equation}
A_{l,1}=\{t\re (l+1):t\in A,\ |t|\ge l+1,\mathrm{\  and\ }t(l)=1\}.
\end{equation}
We say that $l$ is a {\em minimal level of a new set of  parallel $1$'s in $A$} if
the set $A_{l,1}$ has at least two distinct members,
and for each $l'<l$,
the set $\{s\in A_{l,1}:  s(l')=1\}$ has cardinality strictly smaller than $|A_{l,1}|$.

\begin{defn}[\SPOC]\label{defn.strPOC}
A  subtree $A$ of a strong coding tree satisfies the {\em \SPOC}
if $A$ satisfies the \POC\ and additionally, the following hold:
For each  $l$ which is  the minimal length of a set of new parallel $1$'s in $A$,
\begin{enumerate}
\item
The critical node in $A$ with minimal length greater than or equal to  $l$ is a coding node in $A$, say $c$;
\item
There are no terminal nodes in $A$ in the interval $[l,|c|)$  ($c$ can be terminal in $A$);
\item
$A_{l,1}=\{t\re (l+1):t\in A_{|c|,1}\}$.
\end{enumerate}
\end{defn}

Thus a tree $A$ satisfies the \SPOC\ if it satisfies the \POC\ and moreover,
each  new set of parallel $1$'s in $A$ is witnessed by a coding node in $A$ before any other new set of parallel $1$'s,  critical node, or terminal node in $A$ appears.

\begin{defn}[Strictly Similar]\label{defn.strictly.similar}
Given $A,B$ subtrees of a strong coding tree,
we say that $A$ and $B$ are {\em strictly similar}
if $A$ and $B$ are strongly similar and both satisfy the \SPOC.
\end{defn}

\begin{thm}\label{thm.MillikenIPOC}
Let $T$
 be a strong coding tree
and
let $A$ be a finite subtree  of $T$ satisfying the \SPOC.
Then for any coloring of all strictly similar copies of $A$ in $T$ into finitely many colors,
 there is a strong coding subtree $S\le T$ such that all strictly similar copies of $A$ in $S$ have the same color.
\end{thm}

 Theorem \ref{thm.MillikenIPOC} will be proved via four  lemmas and then doing an induction argument.
Recall that Case (b) of Theorem \ref{thm.matrixHL}  only showed that, when  $C\setminus A$ contains a coding node and $X\in\Ext_T(A,C)$,
there is some $S\le T$ which is homogeneous for all members of $\Ext_S(A,C;X)$.
This  is weaker than the direct analogue of
the statement proved for  Case (a) in
 Theorem \ref{thm.matrixHL},
and this  disparity is addressed by
the following.
 Lemma \ref{lem.endhomog} will build a fusion sequence to obtain an $S\le T$ which is end-homogeneous on $\Ext_S(A,C)$, using
 Case (b) of Theorem \ref{thm.matrixHL}.
Lemma \ref{lem.Case(c)} will use a new forcing and many arguments from the proof of Theorem \ref{thm.matrixHL}
obtain an analogue of Case (a) when $C\setminus  A$ contains a coding node.
The only difference is that this analogue holds for  $\Ext_S^{SP}(A,C)$, rather than $\Ext_S(A,C)$, which is why Theorem \ref{thm.MillikenIPOC} requires the \SPOC.
The last two lemmas involve fusion to construct subtrees which have one color on $\Ext^{SP}_S(A',C)$, for each $A'$ strictly similar to $A$, for the two cases: $C\setminus A$  contains a coding node, and $C\setminus A$ contains a splitting node.
The theorem then follows  by induction and an application of Ramsey's Theorem.

The following basic assumption, similar to  Case (b) of Theorem \ref{thm.matrixHL},  will be  used in much of this section.

\begin{assumption}\label{assumption.6}
Let $A\sse C$ be fixed non-empty finite subtrees of a strong coding tree $T$ such that $A$ and $C$ satisfy the \SPOC.
Let $A_e$ be a subset of $A^+$, and assume that $A_e$ and  $C\setminus A$ are level sets,
and that $C\setminus A$  extends $A_e$, contains a coding node, and contains the sequence $0^{l_C}$.
Let $d+1=|A_e|$ and list the nodes of $A_e$ as $\lgl s_i:i\le d\rgl$, and the nodes of $C\setminus A$ as $\lgl t_i:i\le d\rgl$ so that each $t_i$ extends $s_i$ and  $t_d$ is the coding node in $C\setminus A$.
For $k\in\{0,1\}$,
let
 $I_k$ denote the set of $i\le d$ such that the immediate extension  of $t_i$ in $T$ is $k$.
Since $C\setminus A$ contains a coding node, the immediate successors of the $t_i$ are well-defined in $T$.
\end{assumption}

As usual,
when we talk about the parallel $1$'s of $C\setminus A$, we are  taking into account  the passing numbers of the members of $(C\setminus A)^+$ at the coding node $t_d$.
Recall that values of
the immediate successors of the $t_i$, $i\le d$,
 are considered when determining whether or not a level set $Y$ is in $\Ext_T(A,C)$,
this being defined as in Case (b) of the previous section.
We hold to the convention that for $Y\in \Ext_T(A,C)$,
the nodes in $Y$ are labeled $y_i$, $i\le d$, where  $y_i\contains s_i$ for each $i$.
In particular, $y_d$ is the coding node in $Y$.
Define
\begin{equation}
\Ext_T^{SP}(A,C)=\{Y\in\Ext_T(A,C):A\cup Y\mathrm{\
  satisfies\ the\ Strict\ Parallel\ } 1\mathrm{'s\ Criterion}\}.
\end{equation}
Recall  the
definition of $\splitpred_T(x)$ from
 Subsection
 \ref{subsection.4.1}.
We point out that if the parallel $1$'s in $C\setminus A$ are already witnessed in $A$,
then $\Ext_T^{SP}(A,C)$ is equal to $\Ext_T(A,C)$.
If there are parallel $1$'s in $C\setminus A$  not witnessed in $A$,
then $Y\in  \Ext_T^{SP}(A,C)$  if and only if $Y\in \Ext_T(A,C)$ and  additionally
 for the minimal $l$ such that $\{i<d:y_i(l)=1\}=I_1$,
$A\cup \{\splitpred_T(y_i\re l):i\in I_1\}\cup\{y_i\re l:i\in I_0\}$ satisfies the \POC.
Now we define the notion of minimal pre-extension of $A$ to a copy of $C$.
This  will be used in the next lemma to obtain a strong form of end-homogeneity for the case when $\max(C)$ has a coding node.

\begin{defn}[Minimal pre-extension of $A$ to a copy of $C$]
Let
 $X=\{x_i:i\le d\}$  be any   level set extending $A_e$ such that $x_i\contains s_i$ for each $i\le d$
and such that
 the length $l$ of the nodes in $X$
is the length of some coding node in $T$.
We say that  $X$ is  a
{\em minimal pre-extension in $T$ of $A$ to a copy of $C$} if
\begin{enumerate}
\item[(i)]
$\{i\le d:x_i^+(l)=1\}=I_1$,
where
$x_i^+$ denotes the immediate extension of $x_i$ in $\widehat{T}$; and
\item[(ii)]
$A\cup\{\splitpred_T(x_i):i\in I_1\}\cup\{x_i:i\in I_0\}$
satisfies the \POC.
\end{enumerate}
\end{defn}
We will simply call such an $X$ a
 {\em minimal pre-extension} when $T$, $A$, and $C$ are clear.
Minimal pre-extensions
are exactly the level sets  in $T$ which
 can be extended to a member of $\Ext_T^{SP}(A,C)$.
For $X$ any minimal pre-extension,
define $\Ext_T(A,C;X)$ to be
 the set of all $Y\in \Ext_T(A,C)$
such that $Y$ extends $X$.
Then
\begin{equation}
\Ext_T^{SP}(A,C)=\bigcup\{\Ext_T(A,C;X): X\mathrm{\ is\ a\ minimal\ {pre{-}extension}}\},
\end{equation}

\begin{defn}\label{defn.endhomog}
A coloring on $\Ext^{SP}_T(A,C)$ is {\em end-homogeneous} if  for each minimal pre-extension $X$ of $A$ to a copy of  $C$,
 every member  of $\Ext_T(A,C;X)$ has the same color.
\end{defn}

\begin{lem}[End-homogeneity]\label{lem.endhomog}
Assume \ref{assumption.6}, and
let $k$ be minimal such that $\max(A)\sse  r_k(T)$.
Then for any coloring $h$
 of $\Ext_T(A,C)$ into  two  colors,
 there is a $T'\in[r_k(T),T]$ such that  $h$ is
 end-homogeneous  on $\Ext_{T'}^{SP}(A,C)$.
\end{lem}

\begin{proof}
Let $(n_i)_{i<\om}$ enumerate those integers greater than $k$ such that there is a minimal pre-extension of $A$ to a copy of  $C$ from among  the maximal nodes in
$r_{n_i}(T)$.
Each  of these  $r_{n_i}(T)$ contains a coding node in its maximal level,
though there may be  minimal pre-extensions  contained in $\max(r_{n_i}(T))$ not containing that coding node.

Let $T_{-1}$ denote $T$.
Suppose that $j<\om$ and $T_{j-1}$ are given  so that
the coloring $h$ is homogeneous on
$\Ext_{T_{j-1}}(A,C;X)$ for each minimal pre-extension $X$ in $r_{n_j-1}(T_{j-1})$.
Let  $U_{j-1}$ denote $r_{n_j-1}(T_{j-1})$.
Enumerate the collection of all minimal pre-extensions of $A$ to $C$ from among  $\max(r_{n_j}(T_{j-1}))$ as $X_0,\dots, X_q$.
We will  do an inductive argument over $p\le q$  to obtain a $T_j\in [U_{j-1},T_{j-1}]$ such that $\max(r_{n_j}(T_j))$ extends $\max(r_{n_j}(T_{j-1}))$
and   $\Ext_{T_j}(A,C;Z)$ is homogeneous
for each minimal pre-extension $Z$ in $\max(r_{n_j}(T_{j-1}))$.

Suppose $p\le q$ and for all $i<p$,
there are strong coding trees $S_i$ such that
 $S_0\in [U_{j-1},T_{j-1}]$,
and
 for all $i'<i<p$,
$S_{i}\in [U_{j-1},S_{i'}]$
and
 $h$ is homogeneous on $\Ext_{S_{i}}(A,C;X_{i})$.
Let $l$ denote the length of the  nodes in $\max(r_{n_j}(T_{j-1}))$.
Note that
$X_p$ is contained in $r_{n_j}(S_{p-1})\re l$,
though $l$ does not have to  be  the length of any node in $S_{p-1}$.
The point is that the set of nodes $Y_p$ in $\max(r_{n_j}(S_{p-1}))$ extending $X_p$
is again a minimal pre-extension.
Extend the nodes in $Y_p$ to some $Z_p\in \Ext_{ S_{p-1}}(A,C;Y_p)$,
and let $l'$ denote the length of the nodes in $Z_p$.
Note that $Z_p$ has no new sets of parallel $1$'s over $A\cup Y_p$.
Let $W_p$ consist of the nodes in $Z_p$ along with the leftmost extensions   of
the nodes in $\max(r_{n_j}(S_{p-1}))\setminus Y_p$
 to the length  $l'$ in $S_{p-1}$.

Let $S'_{p-1}$ be a strong coding tree in $[U_{j-1},S_{p-1}]$ such that $\max(r_{n_j}(S'_{p-1}))$ extends $W_p$.
Such an $S'_{p-1}$ exists by Lemma \ref{lem.facts},
since
 $W_p$
has exactly the same set of new  parallel $1$'s
over  $r_{n_{j-1}}(S_{p-1})$ as does
$\max(r_{n_j}(S_{p-1}))$.
Apply
 Case (b)  of Theorem \ref{thm.matrixHL}
to obtain  a strong coding tree
 $S_p\in [U_{j-1},S'_{p-1}]$ such that the coloring on $\Ext_{S_p}(A,C;Z_p)$ is homogeneous.
At the end of this process, let $T_j=S_q$.
Note that for each minimal pre-extension $Z\sse\max(r_{n_j}(T_j))$,
there is a unique $p\le q$ such that
$Z$ extends $X_p$,
since each node in  $\max(r_{n_j}(T_j))$ is a unique  extension of one node in $\max(r_{n_j}(T_{j-1}))$,
and hence
$\Ext_{T_j}(A,C;Z)$ is homogeneous.

Having chosen each $T_j$ as above,
let $T'=\bigcup_{j<\om}r_{n_j}(T_j)$.
Then $T'$ is a strong coding tree which is
in $[r_k(T),T]$,
and for each minimal pre-extension $Z$ in $T'$,
$\Ext_{T'}(A,C;Z)$ is homogeneous for $h$.
Therefore,  $h$ is end-homogeneous on $\Ext^{SP}_{T'}(A,C)$.
\end{proof}

The next lemma provides a means for uniformizing the  end-homogeneity from the previous lemma
to obtain one color for all
members of  $\Ext_S^{SP}(A,C)$.
This will yield almost the full analogue of Case (a) of Theorem \ref{thm.matrixHL} for Case (b),
when the level sets being colored contain a coding node, the difference being the restriction to strictly similar extensions rather than just strongly similar extensions.
The arguments are often similar to those of
Case (a) of
Theorem \ref{thm.matrixHL}, but sufficiently different to warrant a proof.

\begin{lem}\label{lem.Case(c)}
Assume \ref{assumption.6},  and suppose that $B$ is a finite strong coding tree valid in $T$ and  $A$ is a subtree of $B$
such that $\max(A)\sse\max(B)$.
Suppose that $h$ is end-homogeneous on $\Ext_{T}^{SP}(A,C)$.
Then there is an $S\in[B,T]$ such that $h$ is homogeneous on
 $\Ext_S^{SP}(A,C)$.
\end{lem}

\begin{proof}
Given any $U\in[B,T]$,
let $\MPE_U(A,C)$ denote  the set of all minimal pre-extensions of $A$ to a copy of $C$ in $U$.
Without loss of generality, we may assume that the nodes in $C\setminus A$ occur in an interval of $T$ strictly above the interval of $T$ containing $B$.
This presents no obstacle to the application, as the goal is to find some $S\in[B,T]$ for which $h$ takes the same value on
every extension in $\Ext_U(A,C)$ extending some  member of $\MPE_S(A,C)$,
and we can  take the first level of $S$ above $B$ to  be in the interval of $T$ strictly above $B$ since $B$ is valid in $T$.

Enumerate the nodes of $A_e$ as $\{s_i:i\le d\}$,
letting $i_0$ be the index such that $s_{i_0}$ is a sequence of all $0$'s.
In the notation of Assumption  \ref{assumption.6},
$i_0$ is a member of  $I_0$.
Each member $Y$ of $\MPE_T(A,C)$
will be enumerated as $\{y_i:i\le d\}$ so that $y_i\contains s_i$ for each $i\le d$.
Given   $Y\in\MPE_T(A,C)$,
define the notation
\begin{equation}
\splitpred_T(Y)=\{y_i:i\in I_0\}\cup\{\splitpred_T(y_i):i\in I_1\}.
\end{equation}
Since $C$ satisfies the \SPOC,
$C\setminus A$ is in $\MPE_T(A,C)$.
Let  $C^-$ denote $\splitpred_T(C\setminus A)$.
Since we are assuming that $C\setminus A$ is contained in an interval of $T$ above the interval containing $\max(A)$,
each node of $C^-$ extends one node of $A_e$.
For any $U\in[B,T]$, define $X\in\Ext_{U}(A,C^-)$ if and only if
$X=\splitpred_U(Y)$ for some
$Y\in\MPE_U(A,C)$.
Equivalently,
 $X\in\Ext_{U}(A,C^-)$ if and only if
 the following three conditions hold:
\begin{enumerate}
\item
$X$ extends $A_e$; label the nodes in $X$ as $\{x_i:i\le d\}$ so that  $x_i\contains s_i$.
\item
There is a coding node $c$  in $U$ such that
$|c|=|x_{i_0}|$;
for each $i\in I_0$,
the
passing number of $x_i$ at $c$ is
 $0$;
and for each $i\in I_1$,
$x_i=\splitpred_{U}(y_i)$ for some $y_i\contains s_i$ in $U$ of length $|c|$ such that the
passing number  of $y_i$  at $c$ is $1$.
\item
The set $A\cup  X$ satisfies the \POC.
\end{enumerate}
Thus, $X$ is a member of $\Ext_{U}(A,C^-)$ if and only if
$\{x_i:i\in I_0\}$ along with the rightmost paths extending $\{x_i:i\in I_1\}$ to  length $|x_{i_0}|$ forms a minimal pre-extension of $A$ to a copy of $C$ in $U$.
Note that condition (3) implies that $X$ has no new sets of parallel $1$'s over $A$,
since $X$ contains no coding node.

By assumption,
the coloring $h$ on $\Ext_{T}^{SP}(A,C)$  is end-homogeneous.
Thus, it induces a coloring on $\MPE_T(A,C)$,
by giving $Y\in\MPE_T(A,C)$ the $h$-color that all members of $\Ext_T(A,C;Y)$ have.
This further
induces a coloring $h'$  on $\Ext_{T}(A,C^-)$,
since a  set of nodes $X$ in  $T$ is in
$ \Ext_T(A,C^-)$
if and only if $X=\splitpred_T(Y)$ for some   $Y\in\MPE_T(A,C)$.
Define $h'(\splitpred_T(Y))$ to be the color of $h$ on  $\Ext_{T}(A,C;Y)$.

Let $L$ denote the collection of all  $l<\om$ such that there is a member of
 $\Ext_{T}(A,C^-)$ with maximal  nodes of length $l$.
For each
  $i\in (d+1)\setminus\{i_0\}$,   let  $T_i=\{t\in T:t\contains s_i\}$.
Let $T_{i_0}=\{t\in T\cap 0^{<\om}:t\contains s_{i_0}\}$, the collection of all leftmost nodes in $T$ extending $s_{i_0}$.
Let $\kappa=\beth_{2d+2}$.
The following forcing notion $\bQ$  will  add $\kappa$ many paths through each $T_i$, $i\in (d+1)\setminus\{i_0\}$ and
 one path through $T_{i_0}$.
The present case is handled similarly to Case (a) of Theorem \ref{thm.matrixHL}, so much of the  current proof refers back to the proof of Theorem \ref{thm.matrixHL}.

We now define a new forcing.
Let
$\bQ$ be  the set of conditions $p$ such that
$p$ is a  function
of the form
$$
p:(d+1)\times\vec{\delta}_p\ra  T,
$$
where $\vec{\delta}_p\in[\kappa]^{<\om}$,
$l_p\in L$,
 and
there  is  some
some coding node
$c^{T}_{n(p)}$ in $T$
 such that $l^{T}_{n(p)}=l_p$,
and
\begin{enumerate}
\item[(i)]
 For each $(i,\delta)\in (d+1)\times
\vec{\delta}_p$,
$p(i,\delta)\in T_i$ and
 $l^{T}_{n(p)-1}< |p(i,\delta)|\le l_p$; and

\item [(ii)]
\begin{enumerate}
\item[($\al$)]
If  $i\in I_1$,
 then
 $p(i,\delta)=\splitpred_{T}(y)$ for some $y\in T_i\re l_p$
which has immediate extension $1$ in $T$.
\item[$(\beta)$]
If $i\in I_0$,  then
$p(i,\delta)\in T_i\re l_p$ and has immediate extension $0$ in $T$.
\end{enumerate}
\end{enumerate}
It follows from the definition  that for $p\in \bQ$,
the range  of $p$, $\ran(p):=\{p(i,\delta):(i,\delta)\in (d+1)\times\vec{\delta}_p\}$, has no pre-determined new sets of parallel $1$'s.
Furthermore, all nodes in $\ran(p)$ are contained in the $n(p)$-th interval of $T$.
We point out that  $\ran(p)$ may or may not contain a coding node.
If it does, then that coding node must appear as $p(i,\delta)$ for some $i\in I_0$.

The partial ordering on $\bQ$ is defined as follows:
$q\le p$ if and only if
$l_q\ge l_p$, $\vec{\delta}_q\contains \vec{\delta}_p$,
\begin{enumerate}
\item[(i)]
$q(i,\delta)\contains p(i,\delta)$  for each $(i,\delta)\in (d+1)\times\vec{\delta}_p$; and

\item[(ii)]
$\{q(i,\delta):(i,\delta)\in (d+1)\times\vec{\delta}_p\}$
has no new sets of parallel $1$'s over
 $\ran(p)$.
\end{enumerate}

It is routine to show that
Claims
\ref{claim.atomless_separative}
 and \ref{claim.densehigh} in the proof of Theorem \ref{thm.matrixHL}
also hold for $(\bQ,\le)$.
That is, $(\bQ,\le)$ is an atomless partial order, and any condition in $\bQ$ can be extended by two incompatible conditions of length greater than any given $l<\om$.

Let $\dot{\mathcal{U}}$ be a $\bQ$-name for a non-principal ultrafilter on $L$.
For each $i\le d$ and $\al<\kappa$,  let $\dot{b}_{i,\al}$ be a $\bQ$-name for the $\al$-th generic branch through $T_i$;
that is, $\dot{b}_{i,\al}=\{\lgl p(i,\al),p\rgl:p\in \bQ$ and $\al\in\vec{\delta}_p\}$.
For
any condition $p\in \bQ$, for
$(i,\al)\in
 I_0\times \vec\delta_p$,  $p$ forces that $\dot{b}_{i,\al}\re l_p= p(i,\al)$.
For $(i,\al)\in I_1\times\vec\delta_p$,  $p$ forces that $\splitpred_{T}(\dot{b}_{i,\al}\re l_p)= p(i,\al)$.
For $\vec{\al}=\lgl\al_0,\dots,\al_{d}\rgl\in[\kappa]^{d+1}$,
\begin{equation}
\mathrm{let\ \ }\dot{b}_{\vec{\al}}\mathrm{\  \  denote\ \ }
\lgl \dot{b}_{0,\al_0},\dots,\dot{b}_{d,\al_d}\rgl.
\end{equation}
For $l\in L$, we shall  use the abbreviation
\begin{equation}
\dot{b}_{\vec\al}\re l
\mathrm{\ \ to\ denote \ \ }
\splitpred_T(\dot{b}_{\vec\al}\re l),
\end{equation}
which is exactly
$\{\dot{b}_{i,\al_i}\re l :i\in I_0\}\cup \{\splitpred_T(\dot{b}_{i,\al_i}\re l):i\in I_1\}$.

Similarly to  Part II  of  the proof of Theorem \ref{thm.matrixHL},
we will find  infinite  pairwise disjoint sets $K_i\sse \kappa$, $i\le d$, such that $K_0<K_1<\dots K_d$,
and conditions $p_{\vec\al}$, $\vec\al\in \prod_{i\le d}K_i$,
such that these conditions are pairwise compatible,
have the same images in $T$, and force the same color $\varepsilon^*$  for $h'(\dot{b}_{\vec\al}\re l)$  for  $\dot{\mathcal{U}}$ many levels  $l$ in $L$.
Moreover, the nodes $\{t^*_i:i\le d\}$ obtained from the application of the \Erdos-Rado Theorem for this setting
will  extend
 $\{s_i:i\le d\}$
and  form a member of $\Ext_{T}(A,C^-)$.
The arguments are mostly similar to those in Part II of Theorem \ref{thm.matrixHL}, so we only fill in the details for arguments which are necessarily different.
\vskip.1in

\noindent{\bf \underline{Part II}.}
For each $\vec\al\in[\kappa]^{d+1}$,
choose a condition $p_{\vec{\al}}\in\bQ$ such that
\begin{enumerate}
\item
 $\vec{\al}\sse\vec{\delta}_{p_{\vec\al}}$.

\item
$\{p_{\vec\al}(i,\al_i):i\le  d\}\in\Ext_T(A,C^-)$.
\item
$p_{\vec{\al}}\forces$ ``There is an $\varepsilon\in 2$  such that
$h(\dot{b}_{\vec{\al}}\re l)=\varepsilon$
for $\dot{\mathcal{U}}$ many $l$ in $\dot{L}_d$."

\item
$p_{\vec{\al}}$ decides a value for $\varepsilon$, call it  $\varepsilon_{\vec{\al}}$.
\item
$h(\{p_{\vec\al}(i,\al_i):i\le  d\})=\varepsilon_{\vec{\al}}$.
\end{enumerate}

Properties (1) -  (5) can be guaranteed  as follows.
For each $i\le d$, let $t_i$ denote the member of $C^-$ which extends $s_i$.
For each  $\vec{\al}\in[\kappa]^{d+1}$, let
$$
p^0_{\vec{\al}}=\{\lgl (i,\delta), t_i\rgl: i\le d, \ \delta\in\vec{\al} \}.
$$
Then $p^0_{\vec{\al}}$ is a condition in $\bP$ and
$\vec\delta_{p_{\vec\al}^0}= \vec\al$, so (1) holds.
Further,
$\ran(p^0_{\vec\al})$ is a member of $\Ext_T(A,C^-)$ since it is exactly $C^-$.
Note that
for any $p\le p_{\vec\al}^0$,
 $\{p(i,\al_i):i\le d\}$
is also  a member of $\Ext_T(A,C^-)$,
so (2) holds for any $p\le p_{\vec\al}^0$.
Take  an extension $p^1_{\vec{\al}}\le p^0_{\vec{\al}}$ which
forces  $h'(\dot{b}_{\vec{\al}}\re l)$ to be the same value for
$\dot{\mathcal{U}}$  many  $l\in \dot{L}_d$,
and then take
 $p^2_{\vec{\al}}\le p_{\vec{\al}}^1$ deciding a value $\varepsilon_{\vec{\al}}$ for which $p^2_{\vec{\al}}$ forces that $h'(\dot{b}_{\vec{\al}}\re l)=\varepsilon_{\vec{\al}}$
for $\dot{\mathcal{U}}$ many $l$ in $\dot{L}_d$.
This satisfies
 (3) and (4).
Take $p_{\vec\al}\le p^2_{\vec\al}$  which  decides  $h'(\dot{b}_{\vec\al}\re l_{p_{\vec\al}})=\varepsilon_{\vec\al}$.
Then $p_{\vec\al}$ satisfies (1) through (5),
since
$p_{\vec\al}$ forces
 $h'(\{p_{\vec\al}(i,\al_i):i\le d\})=\varepsilon_{\vec\al}$.

We are assuming $\kappa=\beth_{2d+2}$.
Let $D_e=\{0,2,\dots,2d\}$ and  $D_o=\{1,3,\dots,2d+1\}$, the sets of  even and odd integers less than $2d+2$, respectively.
Let $\mathcal{I}$ denote the collection of all functions $\iota: (2d+2)\ra (2d+2)$ such that
$\iota\re D_e$
and $\iota\re D_o$ are strictly  increasing sequences
and $\{\iota(0),\iota(1)\}<\{\iota(2),\iota(3)\}<\dots<\{\iota(2d),\iota(2d+1)\}$.
For $\vec{\theta}\in[\kappa]^{2d+2}$,
$\iota(\vec{\theta}\,)$ determines the pair of sequences of ordinals $(\theta_{\iota(0)},\theta_{\iota(2)},\dots,\theta_{\iota(2d))}), (\theta_{\iota(1)},\theta_{\iota(3)},\dots,\theta_{\iota(2d+1)})$,
both of which are members of $[\kappa]^{d+1}$.
Denote these as $\iota_e(\vec\theta\,)$ and $\iota_o(\vec\theta\,)$, respectively.
Let $\vec{\delta}_{\vec\al}$ denote $\vec\delta_{p_{\vec\al}}$,
 $k_{\vec{\al}}$ denote $|\vec{\delta}_{\vec\al}|$,
and let $l_{\vec{\al}}$ denote  $l_{p_{\vec\al}}$.
Let $\lgl \delta_{\vec{\al}}(j):j<k_{\vec{\al}}\rgl$
denote the enumeration of $\vec{\delta}_{\vec\al}$
in increasing order.
Define a coloring  $f$ on $[\kappa]^{2d+2}$ into countably many colors as follows:
Given  $\vec\theta\in[\kappa]^{2d+2}$ and
 $\iota\in\mathcal{I}$, to reduce the number of subscripts,  letting
$\vec\al$ denote $\iota_e(\vec\theta\,)$ and $\vec\beta$ denote $\iota_o(\vec\theta\,)$,
define
\begin{align}\label{eq.fiotatheta(c)}
f(\iota,\vec\theta\,)&= \,
\lgl \iota, \varepsilon_{\vec{\al}}, k_{\vec{\al}},
\lgl \lgl p_{\vec{\al}}(i,\delta_{\vec{\al}}(j)):j<k_{\vec{\al}}\rgl:i\le d\rgl,\cr
& \lgl  \lgl i,j \rgl: i\le d,\ j<k_{\vec{\al}},\ \mathrm{and\ } \delta_{\vec{\al}}(j)=\al_i \rgl,
\lgl \lgl j,k\rgl:j<k_{\vec{\al}},\ k<k_{\vec{\beta}},\ \delta_{\vec{\al}}(j)=\delta_{\vec{\beta}}(k)\rgl\rgl.
\end{align}
Let $f(\vec{\theta}\,)$ be the sequence $\lgl f(\iota,\vec\theta\,):\iota\in\mathcal{I}\rgl$, where $\mathcal{I}$ is given some fixed ordering.
By the \Erdos-Rado Theorem,
there is a subset $K\sse\kappa$ of cardinality $\aleph_1$
which is homogeneous for $f$.
Take $K'\sse K$ such that between each two members of $K'$ there is a member of $K$ and $\min(K')>\min(K)$.
Then take subsets $K_i\sse K'$ such that  $K_0<\dots<K_{d}$
and   each $|K_i|=\aleph_0$.
The following three claims and lemma are direct analogues of
Claims \ref{claim.onetypes}, \ref{claim.j=j'},  and \ref{subclaimA}, and Lemma \ref{lem.compat}.
Their proofs follow by simply making the correct notational substitutions, and so are omitted.

\begin{claim}\label{claim.onetypes(c)}
There are $\varepsilon^*\in 2$, $k^*\in\om$,
and $ \lgl t_{i,j}: j<k^*\rgl$, $i\le d$,
 such that
for all $\vec{\al}\in \prod_{i\le d}K_i$ and  each $i\le d$,
 $\varepsilon_{\vec{\al}}=\varepsilon^*$,
$k_{\vec\al}=k^*$,  and
$\lgl p_{\vec\al}(i,\delta_{\vec\al}(j)):j<k_{\vec\al}\rgl
=
 \lgl t_{i,j}: j<k^*\rgl$.
\end{claim}

Let $l^*=|t_{i_0}|$.
Then for each $i\in I_0$,
 the nodes  $t_{i,j}$,  $j<k^*$,    have  length $l^*$;
and for each $i\in I_1$,
the nodes $t_{i,j}$,  $j<k^*$,  have length in the interval $(l^T_{n-1},l^T_n)$,
where $n$ is the index of the coding node in $T$ of length $l^*$.

\begin{claim}\label{claim.j=j'(c)}
Given any $\vec\al,\vec\beta\in \prod_{i\le d}K_i$,
if $j,k<k^*$ and $\delta_{\vec\al}(j)=\delta_{\vec\beta}(k)$,
 then $j=k$.
\end{claim}

For any $\vec\al\in \prod_{i\le d}K_i$ and any $\iota\in\mathcal{I}$, there is a $\vec\theta\in[K]^{2d+2}$ such that $\vec\al=\iota_o(\vec\theta)$.
By homogeneity of $f$, there is a strictly increasing sequence
$\lgl j_i:i\le d\rgl$  of members of $k^*$ such that for each $\vec\al\in \prod_{i\le d}K_i$,
$\delta_{\vec\al}(j_i)=\al_i$.
For each $i\le d$, let $t^*_i$ denote $t_{i,j_i}$.
Then  for each $i\le d$ and each $\vec\al\in \prod_{i\le d}K_i$,
\begin{equation}
p_{\vec\al}(i,\al_i)=p_{\vec{\al}}(i, \delta_{\vec\al}(j_i))=t_{i,j_i}=t^*_i.
\end{equation}

\begin{lem}\label{lem.compat(c)}
For any finite subset $\vec{J}\sse \prod_{i\le d}K_i$,
the set of conditions $\{p_{\vec{\al}}:\vec{\al}\in \vec{J}\,\}$ is  compatible.
Moreover,
$p_{\vec{J}}:=\bigcup\{p_{\vec{\al}}:\vec{\al}\in \vec{J}\,\}$
is a member of $\bP$ which is below each
$p_{\vec{\al}}$, $\vec\al\in\vec{J}$.
\end{lem}

\begin{claim}\label{subclaimA(c)}
If $\beta\in \bigcup_{i\le d}K_i$,
$\vec{\al}\in\prod_{i\le d}K_i$,
and $\beta\not\in\vec\al$,
 then
$\beta$ is not  a member of   $\vec{\delta}_{\vec{\al}}$.
\end{claim}

\noindent{\bf \underline{Part III}.}
Let $(n_j)_{j<\om}$ denote the set of  indices for which  there is an
$X\in \MPE_T(A,C)$
with $X=\max(V)$ for some
$V$ of $r_{n_j}[B,T]$.
 For  $i\in I_0$,
let $u^*_i=t^*_i$.
For $i\in I_1$,
let $u_i^*$ be the  leftmost extension of $t^*_i$ in $T\re l^*$.
Note that $\{u_i^*:i\le d\}$ has no new sets of parallel $1$'s over $A_e$.
Extend each node $u$ in $\max(B)\setminus A_e$   to its  leftmost extension in $T\re l^*$ and label that extension $u^*$.
Let
\begin{equation}
U^*=\{u^*_i:i\le d\}\cup\{u^*:u\in \max(r_k(T))\setminus A_e\}.
\end{equation}
Thus, $U^*$ extends $\max(B)$,
all sets of parallel $1$'s in $U^*$ are already witnessed in $B$ since $B$ is valid in $T$,
and $U^*$ has no new pre-determined parallel $1$'s.

Suppose that $j<\om$ and for all $i<j$, there have been chosen $S_i\in r_{n_i}[B,T]$ such that
$h'$ is constant of value $\varepsilon^*$ on $\Ext_{S_{i}}(A,C^-)$,
and
 for $i<i'<j$,
$S_i\sqsubset S_{i'}$.
Let $k_B$ be the integer such that $B=r_{k_B}(B)$, and
let $e$ be the index such that
$l^T_{e-1}$ is greater than the length of the maximal nodes in $B$.
For $j=0$,
take
  $V_0$ to be any member of $r_{n_0}[B,T]$
such that
the nodes in $\max(r_{k_B+1}(V_0))$ extend the nodes in $U^*$
and have length greater than $l^T_e$.
This is possible by Lemma \ref{lem.facts}.
For $j\ge 1$,
take $V_j\in r_{n_j}[B,T]$
such that $V_j\sqsupset S_{j-1}$.
Let $X$ denote $\max(V_j)$.
Then the
nodes in $\splitpred_T(X)$
extend the nodes in $U^*$, and moreover,
extend the nodes in $\max(S_{j-1})$ if $j\ge 1$.
By the definition of $n_j$, the set of nodes $X$ contains a coding node.
For each $i\in I_0$,
let $Y_i$ denote the set of all $t\in T_i\cap X$
which have immediate extension $0$ in $T$.
For each $i\in I_1$,
let $Y_i$ denote the set of all
splitting nodes in
$T_i\cap \splitpred_T(X)$.
For each $i\le d$,
let $J_i$ be a subset of $K_i$ of size $|Y_i|$,
and enumerate the members of $Y_i$ as $q(i,\delta)$, $\delta\in J_i$.
Let $\vec{J}$ denote the set of $\vec\al\in\prod_{i\le d}J_i$ such that the set $\{q(i,\al_i):i\le d\}$
has no new sets of parallel $1$'s over $A$.
Thus, the set of $\{q(i,\al_i):i\le d\}$, $\vec\al\in \vec{J}$, is exactly the collection of sets of nodes in $\splitpred_T(X)$ which are members of
 $\Ext_{T}(A,C^-)$.
Moreover, for each
$\vec\al\in \vec{J}$
and  all $i\le d$,
\begin{equation}
q(i,\al_i)\contains t^*_i=p_{\vec{\al}}(i,\al_i).
\end{equation}

To complete the construction of the desired $q\in \bQ$ for which $q\le p_{\vec\al}$ for all $\vec\al\in \vec{J}$,
let  $\vec{\delta}_q=\bigcup\{\vec{\delta}_{\vec\al}:\vec\al\in \vec{J}\}$.
For each pair $(i,\gamma)$  with  $\gamma\in\vec{\delta}_q\setminus
J_i$,
there is at least one $\vec{\al}\in\vec{J}$ and some $j<k^*$ such that $\gamma=\delta_{\vec\al}(j)$.
As in Case (a) of Theorem \ref{thm.matrixHL},
for any other $\vec\beta\in\vec{J}$ for which $\gamma\in\vec{\delta}_{\vec\beta}$,
it follows
 that
$p_{\vec\beta}(i,\gamma)=p_{\vec{\al}}(i,\gamma)=t^*_{i,j}$ and  $\delta_{\vec\beta}(j)=\gamma$.
If $i\in I_0$,
let $q(i,\gamma)$ be the leftmost extension
 of $t_{i,j}^*$ in  $T\re l^{V_j}_{n_j}$.
If $i\in I_1$, let $q(i,\gamma)$ be the leftmost extension of $t^*_{i,j}$ to a splitting node in $T$
in the interval
 $(l^{V_j}_{n_j-1}, l^{V_j}_{n_j}]$.
Such a splitting node must exist because of the construction of $U^*$.
Precisely,
let $c^X$ denote the coding node in $X$.
Note that $c^X\re l_B$ must have no parallel $1$'s with any $s_{i'}$, $i'\in I_1$, since $X$ contains a member of $\MPE_T(A,C)$.
If $c^X$ does not extend $t^*_{i'}$ for any $i'\le d$,
then $c^X\re l^*$ is the leftmost extension in $T$ of
$c^X\re l_B$,
which implies that
 $c^X\re l^*$ has no  parallel $1$'s  with  $t^*_{i,j}$.
Thus,  $q(i,\gamma)$, being the leftmost extension of $t^*_{i,j}$, has no parallel $1$'s with $c^X$.
If  $c^X$ extends some $t^*_{i',j'}$,
then $c^X\re l_B=s_{i'}$.
For $c^X$ to be a node in a member of $\MPE_T(A,C)$, $c^X\re l_B$ must not have parallel $1$'s with any $s_i$, $i\in I_1$.
In particular,  $i'$ must be in $ I_0$, and $t^*_{i,j}$ has no parallel $1$'s with $t^*_{i',j'}$,
because $s_i$ and $s_{i'}$ have no parallel $1$'s
and by
 the definition of the partial ordering on $\bQ$, since $t^*_{i,j}$ and $t^*_{i',j'}$  are in $\ran(p_{\vec\al})$ for  any $\vec\al\in [K']^{d+1}$,
and $p_{\vec\al}\le p^0_{\vec\al}$.
Thus, the leftmost extension $q(i,\gamma)$ of $t^*_{i,j}$ has no parallel $1$'s with $c^X$.
Therefore, $q(i,\gamma)$ is well-defined.
Define
\begin{equation}
q=\bigcup_{i\le d}\{\lgl (i,\al),q(i,\al)\rgl: \al\in \vec{\delta}_q\}.
\end{equation}
By a  proof similar to that of
Claim \ref{claim.qbelowpal},
it follows that
$q\le p_{\vec\al}$,
for each $\vec\al\in \vec{J}$.

Take an $r\le q$ in  $\bP$ which  decides some $l_j$ in $L$  which is strictly greater than the length of the next coding node above the  coding node  $c^X$ in $X$,
and such that
for all $\vec\al\in\vec{J}$,  $h'(\dot{b}_{\vec\al}\re l_j)=\varepsilon^*$.
Without loss of generality, we may assume that the maximal nodes in $r$ have length $l_j$.
If $c^X=q(i',\al')$ for some $i'\in I_0$ and $\al'\in J_{i'}$,
then let $c_r$ denote $r(i',\al')$;
otherwise, let $c_r$ denote the leftmost extension of $c^X$ in $T$ of length $l_j$.
Let $Z_0$ denote those nodes in $\splitpred_T(X)\setminus Y_0$  which have length equal to $c^X$; in particular, $Z_0$ is the set of nodes in $X$ which are not splitting nodes in $\splitpred_T(X)$ and are also not in $Y_0$.
For each $z\in Z_0$,
let $s_z$ denote the leftmost extension of $z$ in $T$ to length $l_j$.
Let $Z_1$ denote the set of all splitting nodes  in $\splitpred_T(X)\setminus Y_1$.
For each $z\in Z_1$, let $s_z$ denote the  splitting predecessor  in $T$ of the leftmost extension of $z$ in $T$ to length $l_j$.
This splitting predecessor exists in $T$ for the following reason:
If $z$ is a  splitting node  in $\splitpred_T(X)$,
then $z$ has no parallel $1$'s with $c^X$, and so the leftmost extension of $z$ to any length has no parallel $1$'s with any extension of $c^X$.
In particular, the set $\{s_z:z\in Z_0\cup Z_1\}$ has no new sets of parallel $1$'s over $\splitpred_T(X)$.

Let
\begin{equation}
Z^-=\{q(i,\al):i\le d,\ \al\in J_i\}\cup\{s_z:z\in Z_0\cup Z_1\}.
\end{equation}
Let $Z^*$ denote the extensions in $T$ of all members of $Z^-$ to length $l_j$.
Let $j{^-}$ denote the index such that the maximal coding node in $V_j$ below $c^X$ is $c^{V_j}_{n_{j^-}}$.
Note that $Z^*$ has no new sets of parallel $1$'s over $\splitpred_T(X)$;
furthermore,  the tree induced by
$r_{n_{j^-}}(V_j)\cup Z^*$ is strongly similar to $V_j$, except possibly for the coding node being in the wrong place.
Using Lemma \ref{lem.facts},
extend the nodes in $Z^*$
to obtain some $S_j\in r_{n_j}[r_{n_{j^-}}(V_j),T]$ where  $\max(S_j)$ extends $Z^*$.
Then every member of $\Ext_{S_j}(A,C^-)$ has the same $h'$ color $\varepsilon^*$,
by the choice of $r$,
since each minimal pre-extension in $\MPE_{S_j}(A,C)$ extends some member of $\Ext_{S_j}(A,C-^)$
which extends members in $\ran(r)$ and so have $h'$-color $\varepsilon^*$.

Let $S=\bigcup_{j<\om} S_j$.
Then $S$ is a strong coding tree in $[B,T]$.
Let $Y\in\Ext^{SP}_S(A,C)$.
Then there is some
 $X\in\MPE_S(A,C)$
such that $Y$ extends $X$.
Since
$\splitpred_S(X)$ is in $\Ext_{S_j}(A,C^-)$ for some $j<\om$,
$\splitpred_S(X)$ has
 $h'$ color $\varepsilon^*$.
Thus,
$Y$ has
 $h$-color $\varepsilon^*$.
\end{proof}

Recall that given a tree $A$, $\Sim^s_T(A)$ denote the set of all subtrees  $A'$ of $T$ which are strongly similar to $A$.

\begin{lem}\label{lem.Case(b)}
Assume \ref{assumption.6}.
Then there is a strong coding subtree $S\le T$ such that for each $A'\in\Sim^s_S(A)$,
$h$ is homogeneous on $\Ext^{SP}_S(A',C)$.
\end{lem}

\begin{proof}
Let $(k_i)_{i<\om}$ be the sequence of integers such that $r_{k_i}(T)$ contains a strictly similar copy  of $A$
which is valid in
$r_{k_i}(T)$ and such that $\max(A)\sse\max(r_{k_i}(T))$.
Let $k_{-1}=0$, $T_{-1}=T$,
and
 $U_{-1}=r_{0}(T)$.

Suppose $i<\om$, and $U_{i-1}\ssim r_{k_{i-1}}(T)$
 and $T_{i-1}$ are given satisfying that for each $A'\in\Sim^s_{U_{i-1}}(A)$ valid in $U_{i-1}$ with $\max(A)\sse \max(U_{i-1})$,
$h$ is homogeneous on $\Ext_{U_{i-1}}^{SP}(A',C)$.
Let $U_i$ be in $r_{k_i}[U_{i-1},T_{i-1}]$.
Enumerate the set of all $A'\in\Sim^s_{U_i}(A)$ which are valid in $U_i$ and have $\max(A')\sse \max(U_i)$ as $\lgl A_0,\dots,A_n\rgl$.
Apply Lemma \ref{lem.endhomog} to obtain  $R_{0}\in [U_i,T_{i-1}]$ which is end-homogeneous for $\Ext^{SP}_{R_0}(A_0,C)$.
Then
apply Lemma \ref{lem.Case(c)} to
obtain $R'_{0}\in [U_i,R_0]$
such that $\Ext^{SP}_{R'_{0}}(A_0,C)$ is homogeneous for $h$.
Given $R'_{j}$  for $j<n$,
apply Lemma \ref{lem.endhomog} to obtain a $R_{j+1}\in [U_i,R'_{j}]$ which is end-homogeneous for $\Ext^{SP}_{R_{j+1}}(A_{j+1},C)$.
Then
apply Lemma \ref{lem.Case(c)} to
obtain $R'_{j+1}\in [U_i,R_{j+1}]$
such that $\Ext^{SP}_{R'_{j+1}}(A_{j+1},C)$ is homogeneous for $c$.
Let $T_i=R'_n$.

Let $U=\bigcup_{i<\om}U_i$.
Then $U\le T$ and $h$ has the same color on $\Ext_U^{SP}(A,C)$
for   each
$A'\in\Sim^s_U(A)$ which is valid in $U$.
Finally, take $S\le U$.
Then for each $k<\om$, $r_k(S)$ is valid in $U$,
so in particular,
 each $A'\in\Sim^s_S(A)$ is valid in $U$.
Hence,
$h$ is homogeneous on  $\Ext_S^{SP}(A',C)$.
\end{proof}

A similar lemma holds for the setting of Case (a) in Theorem \ref{thm.matrixHL}.
Since the critical node is a splitting node in this case,
 we do not need to restrict to \SPOC\ copies of $A$ in $T$.

\begin{lem}\label{lem.fusionsplit}
Let $T$ be a strong coding tree and let $A,C,h$ be as in Case (a) of Theorem \ref{thm.matrixHL}.
Then there is a strong coding tree $S\le T$ such that for each  $A'\in\Sim^s_S(A)$, $\Ext_S(A',C)$ is homogeneous for $h$.
\end{lem}

\begin{proof}
Similarly to  the  fusion argument in proof of  Lemma \ref{lem.Case(b)} but
applying Case (a) of Theorem \ref{thm.matrixHL}
in place of Lemmas \ref{lem.endhomog} and \ref{lem.Case(c)},
one  builds a strong coding tree $S\le T$ such that for each copy  $A'$  of $A$ in $S$,
$\Ext_S(A',C)$  is homogeneous for $h$.
\end{proof}

\noindent {\bf Proof of Theorem \ref{thm.MillikenIPOC}}.
The proof is by induction on the number of critical nodes.
Suppose first that $A$ consists of a single node.
Then such a node must be a splitting node in $0^{<\om}\cap T$,
so
$\Sim^s_{T}(A)$ is the infinite set of all splitting nodes in $0^{<\om}\cap T$.
Let $h$ be any finite coloring on $\Sim^s_{T}(A)$.
By Ramsey's Theorem,
 infinitely many members of $\Sim^s_{T}(A)$  must have the same $h$ color,
so there is a  subtree $S\le T$ for which
all its
nodes in $S\cap 0^{<\om}$
have the same $h$ color.
Such an $S\le T$ exists
 by the definition of strong coding tree, since $T$ is strongly skew,  perfect, and the coding nodes are dense in $T$.

Now assume that $n\ge 1$ and  the theorem holds
for each finite tree $B$  with  $n$ or less critical nodes
such that $B$ satisfies the \SPOC\ and  $\max(B)$ contains a node which is a sequence of all $0$'s.
Let $C$ be a finite tree with $n+1$ critical nodes
containing a maximal node in $0^{<\om}$, and suppose
 $h$  maps $\Sim^s_{T}(C)$ into finitely many colors.
Let  $d$ denote the maximal critical node in $C$ and let
  $B=\{t\in C: |t|<|d|\}$.
Apply
Lemma \ref{lem.Case(b)} or \ref{lem.fusionsplit}, depending on whether $d$ is a coding or splitting node,
to obtain $T'\le T$ so that for each $V\in \Sim^s_{T'}(B)$,  the set $\Ext_{T'}^{SP}(V,C)$ is homogeneous for $h$.
Define $g$ on $\Sim_{T'}^s(B)$ by  letting $g(V)$ be the value of  $h$ on $V\cup X$ for any $X\in\Ext_{T'}^{SP}(V,C)$.
By the induction hypothesis,
there is an $S\le T'$ such that $g$ is homogeneous on $\Sim_S^{SP}(B)$.
It follows that $h$ is homogeneous on $\Sim_S^{SP}(C)$.

To finish, let $A$ be any tree satisfying the \SPOC\ where $\max(A)$ does not contain a member of $0^{<\om}$,
and let $g$ be a  finite coloring of $\Sim^s_T(A)$.
Let $l_A$ denote the longest length of nodes in $A$,
and let $C$ be the tree induced by $A\cup\{0^{l_A}\}$.
Then there is a one-to-one correspondence between members of $\Sim_T^s(A)$ and $\Sim^s_T(C)$;
say $\varphi: \Sim_T^s(A)\ra \Sim^s_T(C)$ by
definining $\varphi (A')$ to be the member of $\Sim^s_T(C)$  which is the tree induced  by adding the node $0^{l_{A'}}$ to $A'$.
For $C'\in \Sim_T^s(C)$, define $h(C')=g(\varphi^{-1}(C'))$.
Take $S\le T$ homogeneous for $h$.
Then $S$ is homogeneous for $g$ on $\Sim^s_S(A)$.
\hfill $\square$
\vskip.1in


\section{Incremental strong coding trees}\label{sec.squiggle}

This section develops the notion of incremental new sets of parallel $1$'s,
and the related concepts of
\IPOC,
incremental
 strong coding subtrees, and  sets of witnessing coding nodes.
The main lemma, Lemma \ref{lem.squiggletree},
will be instrumental in attaining
the Ramsey theorem in the next section.
This will be a Ramsey theorem for finite  colorings of
strictly similar copies of any given
 finite subtree of a strong coding tree.
The work in this section sets the stage  for the  removal of the requirement of
 any form of \POC\ on the finite tree whose copies are being colored.

\begin{defn}[Incremental  parallel $1$'s]\label{defn.incrementalpo}
Let $Z$ be a  finite
 subtree of a strong coding tree $T$,
 and let $\lgl l_j:j<\tilde{j}\rgl$
list in increasing order the minimal lengths
of new  parallel $1$'s in $Z$.
We say that
$Z$ has {\em incremental  new sets of  parallel $1$'s},
or simply
{\em incremental parallel $1$'s},
 if
the following holds.
For each $j<\tilde{j}$ for which
\begin{equation}
Z_{l_j,1}:=\{z\re (l_{j+1}):z\in Z, \ |z|>l_j,\mathrm{\  and \ } z(l_j)=1\}
\end{equation}
has size at least three,
letting $m$  denote the length of the longest critical node in $Z$ below $l_j$,
for each
proper subset $Y\subsetneq Z_{l_j,1}$  of cardinality at least two,
there is a $j'<j$ such that  $l_{j'}>m$,
$Y_{l_{j'},1}:=\{y\re(l_{j'}+1):y\in Y$ and $y(l_{j'})=1\}$ has the same size as $Y$,
and
$Y_{l_{j'},1}=Z_{l_{j'},1}$.

We shall say that an infinite  tree $S$ has {\em incremental new parallel $1$'s} if for  each
$l<\om$, the
initial subtree $S\re l$ of $S$
 has incremental new parallel $1$'s.
\end{defn}

\begin{defn}[\IPOC]\label{defn.IPOC}
Let $Z$ be a
 subtree of a strong coding tree $T$.
We say that $Z$ satisfies the {\em \IPOC} if
$Z$ has incremental new parallel $1$'s and  satisfies the \POC.
\end{defn}

Thus, to satisfy the \IPOC, a tree must have a coding node witnessing {\em each} of its new sets of parallel $1$'s,  and these are occuring incrementally.
Note that any strong coding tree does not satisfy the \IPOC.
In the next section,
we will be  interested in extending  finite trees  $A$
to trees $E$ which
 satisfy the \IPOC,
for such $E$ automatically satisfy the \SPOC, so the Ramsey theorems from the previous section can be applied.

The next definition of an incremental strong coding tree will be vital to finding bounds for the big Ramsey degrees in $\mathcal{H}_3$.

\begin{defn}[Incremental Strong Coding Tree]\label{defn.incrementalsct}
A strong coding tree $T$ is called {\em incremental}
if  it satisfies the following.
Let  $n$ be any integer  for which there are at least three distinct nodes  in $T\re (|c^T_n|+1)$ which have passing number $1$ at $c^T_n$, and
list
the set of those nodes as $\lgl t_i: i<\tilde{i}\rgl$.
Let $m$ denote the  length of the maximal splitting node in $T$ below $c^T_n$.
Let $\mathcal{P}$ denote the collection of all proper subsets  $P\sse \tilde{i}$ of size at least two, and
let $\tilde{k}=|\mathcal{P}|$.
Then there is an ordering $\lgl P_k:k<\tilde{k}\rgl$
of $\mathcal{P}$
and
 a strictly increasing sequence $\lgl p_k:k<\tilde{k}\rgl$ such that
\begin{enumerate}
\item[(i)]
$m<p_0$ and $p_{\tilde{k}-1}<|c_n^T|$;
\item[(ii)]
$k<k'<\tilde{k}$ implies $P_k\not\contains P_{k'}$;
and
\item[(iii)]
For each $k<\tilde{k}$,
$p_k$ is minimal  such that $\{i<\tilde{i}:t_i(p_k)=1\}=P_k$.
\end{enumerate}
\end{defn}

\begin{defn}[Incrementally witnessed  parallel $1$'s]\label{defn.incremental}
Let $S\le T$  be an incremental strong coding tree.
We say that  the sets of  parallel $1$'s  in $S$  are  {\em incrementally witnessed in $T$}
if the following hold.
For each $n<\om$,
given $\mathcal{P}$, $\lgl P_k:k<\tilde{k}\rgl$, and
$\lgl p_k:k<\tilde{k}\rgl$ satisfying Definition \ref{defn.incrementalsct},
there is a coding node $w_{n,k}$ in $T$  satisfying
\begin{enumerate}
\item
$|d^S_{m_n-1}|<|w_{n,0}^{\wedge}|<p_0\le |w_{n,0}|<|w_{n,1}^{\wedge}|<p_1\le |w_{n,1}|<\dots  <|w_{n,\tilde{k}-1}^{\wedge}|<p_{\tilde{k}-1}\le|w_{n,\tilde{k}-1}|<|c^S_n|$.
\item
$w_{n,k}$ witnesses the parallel $1$'s in $S_{p_k,1}$;
that is,
for all $z\in S\re (p_k+1)$,
$z(|w_k|)=1$ if and only if $z(p_k)=1$.
\end{enumerate}
\end{defn}

The main lemma of this section shows  that given a strong coding tree $T$,
there is an incremental strong coding subtree $S\le T$ and
moreover,   a set $W\sse T$ of coding nodes disjoint from $S$ such that
each  new set of parallel $1$'s in $S$ is witnessed by a coding node in $W$.
This set-up is what will allow for the definition and use of envelopes in the next section,
as it will ensure that subtrees from $S$ can be enhanced with witnessing coding nodes from $W$ so that their union satisfies the \SPOC.
This will allow  application of  Theorem
\ref{thm.MillikenIPOC}
to
obtain upper bounds on the finite  big Ramsey degrees in the universal triangle-free graph.

\begin{lem}\label{lem.squiggletree}
Let $T$ be a strong coding tree.
Then there is an incremental strong coding tree $S\le T$ and a set of coding nodes $W\sse T$ such that  each
 new set of parallel $1$'s in $S$ is
 incrementally witnessed in $T$ by  a coding node in $W$.
\end{lem}

\begin{proof}
Let  $\lgl d^T_m:m<\om\rgl$ denote the critical nodes in $T$ in order of  increasing length.
Let
$\lgl m_n:n<\om\rgl$ denote the indices
such that $d^T_{m_n}=c^T_n$, so the $m_n$-th critical node in $T$ is the $n$-th coding node in $T$.
Let $S_0$ be a valid subtree of $T$ which is strongly similar to $r_{m_0+1}(T)$.
Since $r_{m_0+1}(T)$ has only one node with passing number $1$ at $c^T_0$,
there is nothing to do;
vacuously $S_0$ has  incremental new sets of parallel $1$'s and these are vacuously witnessed in $T$.

Suppose now that $n\ge 1$ and we have chosen $S_{n-1}\ssim r_{k_{n-1}+1}(T)$
 valid in $T$ so that $S_{n-1}$ is incremental and has its new sets of parallel $1$'s incrementally witnessed in $T$.
Take some $S'_{n-1}\in r_{k_n}[S_{n-1},T]$, so $S_{n-1}'$ is valid in $T$.
There   is a one-to-one correspondence between the nodes in  $\max(r_{k_n+1}(T))$ and $\max(r_{k_n}(T))^+$,
and hence also between   $\max(r_{k_n+1}(T))$ and
$\max(S_{n-1}')^+$.
Let $\varphi :\max(r_{k_n+1}(T))\ra \max(S_{n-1}')^+$ be the lexicographic order preserving bijection.
Let $\lgl t_i:i<\tilde{i}\rgl$ be the lexicographically increasing enumeration of those nodes in $\max(r_{k_n+1}(T))$ which have passing number $1$ at $c^T_n$.
Let $s_i=\varphi(t_i)$.
Then $\{s_i:i<\tilde{i}\}$ is the set of nodes which must extend to have passing number $1$ at  the next coding node in $S$, $c^S_n$.
If $\tilde{i}\le 2$, there is nothing to do; extend to some $S_n\in r_{k_n+1}[S_{n-1}',T]$.

Otherwise, $\tilde{i}\ge 3$.
List all subsets of $\tilde{i}$ of size at least two as $\lgl P_k:k<\tilde{k}\rgl$ in any manner so long as the following is satisfied:
For each $k<k'<\tilde{k}$, $P_k\not\contains P_{k'}$.
Let $X_0$ denote $\max(S_{n-1}')^+$.
Given $k<\tilde{k}$
 and $X_k$,
let  $w_{n,k}^{\wedge}$ be
some splitting node in $T$ in $0^{<\om}$ with length
above the lengths in $X_k$.
Extend all nodes in $X_k$ leftmost in $T$ to length $|w_{n,k}^{\wedge}|+1$,
and let $Y_k$ denote the level  set of these extensions.
Apply Lemma \ref{lem.facts} to
extend the nodes in  $Y_k\cup\{{w_{n,k}^{\wedge}}^{\frown}1\}$ to a  level set $Z_k$ in $T$ such that the following hold:
\begin{enumerate}
\item
The extension of ${w_{n,k}^{\wedge}}^{\frown}1$ is a coding node, label it $w_{n,k}$;
\item
Enumerating $Z_k\setminus\{w_{n,k}\}$ as $\{z_i:i<\tilde{i}\}$ so that for each  $z_i\contains s_i$,
then
for each $i<\tilde{i}$, the immediate extension of $z_i$ in $T$ is $1$ if and only if $i\in P_k$.
\item
The only possible set of new parallel $1$'s in $Z_k$ over $S'_{n-1}\cup X_k$ is $\{z_i:i\in P_k\}$.
\end{enumerate}
If $k<\tilde{k}-1$,
let $X_{k+1}=Z_k$ and continue the procedure.
Upon obtaining $Z_{\tilde{k}-1}$,
apply Lemma \ref{lem.facts}
to obtain an $S_n\in r_{m_n+1}[S_{n-1}',T]$
such that $\max(S_n)$ extends $Z_{\tilde{k}-1}$.

To finish, let $S=\bigcup_{n<\om} S_n$.
Then $S\le T$, $S$ is incremental, and
the sets of parallel $1$'s in $S$ are strongly incrementally witnessed in $T$.
Let $W=\{w_{n,k}:n<\om,\ k<\tilde{k}_n\}$,
where $\tilde{k}_n$ is the number of subsets of $S_{l_n,1}$ of size at least two.
\end{proof}


\section{Ramsey theorem for  strict similarity types}\label{sec.1color}

The strongest Ramsey theorem  proved  so far is
Theorem \ref{thm.MillikenIPOC},
 a Milliken-style theorem for colorings of  finite trees
 satisfying the \SPOC.
In this section we obtain a
 general Ramsey theorem for all {\em strictly similar} copies (Definition \ref{defn.ssimtype}) of any finite tree for which the maximal nodes are exactly the coding nodes forming an antichain.
This involves a new
notion of envelope for strongly diagonal subsets of  strong coding trees, the main property being that any envelope satisfies the \SPOC.
Then applying Theorem \ref{thm.MillikenIPOC},
Lemma \ref{lem.squiggletree},
and envelopes,
we obtain
 Theorem \ref{thm.mainRamsey},
the main Ramsey theorem for strong coding trees in this paper.

Recall from
Definition \ref{def.2.2.Sauer}
that a strongly diagonal subset of $2^{<\om}$  is an antichain $Z$ such that its meet closure forms a transversal with the property that for any splitting node $s\in Z^{\wedge}$, all   nodes in $Z^{\wedge}$ of length greater than $|s|$, except for those nodes extending $s$,
  have passing number $0$ at $s$.
It is a byproduct of the definition of  strong coding trees
 that any  subset of  a strong coding tree forming an antichain is in fact strongly diagonal.
Henceforth, we shall use the term {\em antichain of coding nodes}, or simply {\em antichain}, to refer to  strongly diagonal sets of coding nodes in a strong coding tree.
If $Z$ is an antichain,
then  by the {\em tree induced by $Z$}
we mean the set
\begin{equation}
\{z\re |u|:z\in Z\mathrm{\ and\ } u\in Z^{\wedge}\}.
\end{equation}
We say that an antichain satisfies the \POC\ (\SPOC) if and only if the tree it induces satisfies the \POC\ (\SPOC).

Let $Z$ be an antichain of coding nodes.
Enumerate the nodes in $Z$  in order of increasing length as $\lgl z_i:i< \tilde{i}\rgl$.
For each $l< |z_{\tilde{i}-1}|$,
let
\begin{equation}
I_l^Z=\{i<\tilde{i}:|z_i|>l\mathrm{\ and\ } z_i(l)=1 \},
\end{equation}
and  define
\begin{equation}
Z_{l,1}=\{z_i\re(l+1):i\in I^Z_l\}.
\end{equation}
Thus, $Z_{l,1}$ is the collection of all $z_i\re(l+1)$ which have passing number $1$ at level $l$.
Given $l$ such that  $|Z_{l,1}|\ge 2$,
we say that the set
of parallel $1$'s
at level $l$ is {\em witnessed by the coding node $z_j$   in $Z$}
if $z_i(|z_j|)=1$ for each $i\in I_l^Z$, and
 either $|z_j|\le l$
or else
both $|z_j|> l$ and
 $Z$ has no splitting nodes and no coding nodes of length in $[l,|z_j|]$.
A level $l$
is
{\em the minimal level of a new set  of parallel $1$'s in $Z$} if
$|I^Z_l|\ge 2$ and
whenever
 $l'<l$ and $I^Z_{l'}\sse I^Z_l$,
then $|I^Z_{l'}|<|I^Z_l|$.
It follows that if there are two or more members of $Z$
extending some ${0^l}^{\frown}1$,
then $l$ is the minimal level of a new set of parallel $1$'s,
namely of $I_l^Z$.

\begin{defn}\label{defn.admint}
Given $Z$ an antichain  of coding nodes,
if $l$ is the minimal level of a new set of parallel $1$'s in $Z$, {\em the  admissible interval for $I^Z_l$}
is the interval $[l,l^*]$, where $l^*> l$ is maximal satisfying the following:
\begin{enumerate}
\item
$Z^{\wedge}$ has no
 splitting node and no coding node  of length  in $(l,l^*)$.
\item
Each $l'\in (l,l^*]$
 is not the minimal level of a new set of parallel $1$'s in $Z$.
\end{enumerate}

If $l$ is the minimal level of a new set of parallel $1$'s in $Z$, we say that the set of parallel $1$'s indexed by
$I_l^Z$ is  {\em minimally witnessed in $Z$} if,
letting  $k<\tilde{i}$ be minimal such that
$|z_k|\ge l$,
$|z_k|$ is in the admissible interval $[l,l^*]$
and $z_k$ witnesses the parallel $1$'s in $I^Z_l$;
that is,
$\{i< \tilde{i}:z_i(|z_k|)=1\}=I^Z_l$.
Note that $z_k$ is in the interval $[l,l^*]$ if and only if either $|z_k|=l$ or $|z_k|=l^*$.
Otherwise, we say that $I^Z_l$ is {\em not minimally witnessed in $Z$}.
\end{defn}

The following fact  is immediate from the previous definition.

\begin{fact}\label{fact.thepoint}
If  all new sets of parallel $1$'s
are minimally witnessed in an antichain $Z$,
then the tree
induced by $Z$
 satisfies  the \SPOC.
\end{fact}

\begin{defn}[Strict similarity type]\label{defn.ssimtype}
Given $Z$ a finite   antichain   of coding nodes in some strong coding tree $T$, list the minimal levels of new sets of parallel $1$'s  in $Z$  which are not minimally witnessed in $Z$
 in increasing order as $\lgl l_j: j<\tilde{j}\rgl$.
Enumerate all nodes in $Z^{\wedge}$ as $\lgl u^Z_m:m< \tilde{m}\rgl$ in order of increasing length.
Thus, each $u^Z_m$ is either a splitting node in $Z^{\wedge}$ or else  a coding node $z_i$ for some $i<\tilde{i}$.
The sequence
\begin{equation}
\lgl  \lgl l_j:j<\tilde{j}\rgl,
\lgl I^Z_{l_j}:j< \tilde{j}\rgl,
\lgl |u^Z_m|:m<\tilde{m}\rgl\rgl
\end{equation}
is  {\em  the strict similarity  sequence of $Z$}.

Let $Y$ be another finite  antichain in  $T$, and
let
\begin{equation}
\lgl
\lgl p_j:j<\tilde{k}\rgl,
\lgl I^Y_{p_j}:j< \tilde{k}\rgl,
\lgl |u^Y_m|:m<\tilde{n}\rgl\rgl
\end{equation}
be its  strict similarity  sequence.
We say that $Y$ and $Z$ have the same {\em strict similarity type} or are {\em strictly similar}, written $Y\sssim Z$, if
\begin{enumerate}
\item
$Y^{\wedge}$ and $Z^{\wedge}$ are strongly similar;
\item
$\tilde{j}=\tilde{k}$ and $\tilde{m}=\tilde{n}$;
\item
For each $j<\tilde{j}$, $I^Y_{n_j}=I^Z_{l_j}$;
and
\item
The function $\varphi:\{p_{j}:j<\tilde{j}\}\cup\{|u^Y_m|:m<\tilde{m}\}\ra \{l_{j}:j<\tilde{j}\}\cup\{|u^Z_m|:m<\tilde{m}\}$,
defined by $\varphi(p_{j})=l_{j}$ and
$\varphi(u^Y_m)=u^Z_m$,
is an order preserving bijection between these two linearly ordered sets of natural numbers.
\end{enumerate}
Define
\begin{equation}
\Sim^{ss}_T(Z)=\{Y\sse T:Y\sssim Z\}.
\end{equation}
\end{defn}

Note that  for two  antichains  $Y\sssim Z$,
 the map $f:Y\ra Z$ by $f(y_i)=z_i$ for each $i<\tilde{i}$ induces a strong similarity map from $Y^{\wedge}$ onto $Z^{\wedge}$ by defining $f(y_i\wedge y_j)=z_i\wedge z_j$
for each pair  $i,j<\tilde{i}$.
Then  $f(u^Y_m)=u^Z_m$ for each $m<\tilde{m}$.
Further,
by (3) and  (4) of Definition \ref{defn.ssimtype},
 this map preserves the order in which minimal sets of parallel $1$'s appear, relative to all  other  minimal sets of parallel $1$'s and the nodes in $Y^{\wedge}$ and $Z^{\wedge}$.

The definition of strictly similar in
Definition \ref{defn.ssimtype}
extends Definition \ref{defn.strictly.similar}
to finite sets which do not necessarily satisfy the \POC.
When $Z$ is an antichain such that its induced tree satisfies the \IPOC,
then Definitions \ref{defn.strictly.similar} and \ref{defn.ssimtype} coincide, and further,  for such $Z$, these coincide with the notion of  strongly similar.

\begin{fact}\label{fact.Claim0}
Let $T$ be a strong coding tree, and $A$ and $B$ be subsets of $T$.
Suppose $A$ satisfies the \IPOC.
Then $B\ssim A$ if and only if $B\sssim A$.
\end{fact}

The following notion of envelope is defined in terms of structure without regard to an ambient strong coding tree.
In any given strong coding tree $T$, there will
 certainly be  finite subtrees  of  $T$ which have no envelope in $T$.
This poses no problem to our intended application,
as  by the work done in the previous section,
inside a given strong coding tree $T$, there will be
an incremental strong coding tree  $S$ along with a set of witnessing coding nodes $W\sse T$ so that each finite antichain in  $S$ has an   envelope consisting of nodes from $W$.
Thus, envelopes of antichains in $S$ will exist in $T$.

\begin{defn}[Envelopes]\label{defn.envelope}
Let $Z$ be a finite   antichain of coding nodes and
let
$\lgl
\lgl l_j:j<\tilde{j}\rgl,
\lgl I_{l_j}:j<\tilde{j}\rgl,
\lgl |u_m|:m<\tilde{m}\rgl\rgl$
be the strict similarity sequence of $Z$.
A finite set  $E(Z)$  is an {\em   envelope} of $Z$ if
$E(Z)=Z\cup W$ is an antichain of coding nodes,
where
$W=\{w_j:j <\tilde{j}\}$,
such that the following hold:
 For each $j<\tilde{j}$,
\begin{enumerate}
\item
$w_j$ is in the admissible interval of $l_j$; that is,
$l_j\le |w_j|\le l_j^*$;
\item
$I_{|w_j|}=I_{l_j}$;
\item
 $w_j$ has no parallel $1$'s with any member of $Z\cup (W\setminus\{w_j\})$;
and
\item
 $l^*_{j-1}<|w_j^{\wedge}|<l_j$ and there is no member of
 $(Z\cup W)^{\wedge}$ with length in $(|w_j^{\wedge}|,|w_j|)$.
\end{enumerate}
\end{defn}

The set $W$ is called the set of {\em witnessing coding nodes}, since they minimally witness all parallel $1$'s in $Z$  not minimally witnessed  by any coding node  in $Z$.
The next fact follows immediately from the definitions.

\begin{fact}\label{fact.Claim1}
Let $S$ be any  strongly incremental strong coding tree and
 $Z$ be any antichain  in  $S$.
Then any envelope $E$ of $Z$
satisfies the \IPOC,
and hence also the \SPOC.
\end{fact}

\begin{lem}\label{lem.envelopesbehave}
Let $Y$ and $Z$ be strictly similar  antichains.
Then  any envelope of $Y$ is strictly similar to any envelope of $Z$;
in particular, any two envelopes of $Y$ are strictly similar.
\end{lem}

\begin{proof}
Let $Y=\{y_i:i<\tilde{i}\}$ and $Z=\{z_i:i<\tilde{i}\}$ be the enumerations of $Y$ and $Z$, respectively, in order of increasing length.
Let
\begin{equation}
\lgl \lgl p_j:j<\tilde{j}\rgl,\lgl I^Y_{p_j}:j<\tilde{j}\rgl, \lgl |u^Y_m|:m<\tilde{m}\rgl\rgl
\end{equation}
and
\begin{equation}
\lgl \lgl l_j:j<\tilde{j}\rgl,\lgl I^Z_{l_j}:j<\tilde{j}\rgl, \lgl |u^Z_m|:m<\tilde{m}\rgl\rgl
\end{equation}
be their strict similarity sequences, respectively.
Let $E=Y\cup V$ and $F=Z\cup W$ be any envelopes of $Y$ and $Z$, respectively.
Enumerate the nodes in  $V$ and $W$ in order of increasing length as $\{v_j:j<\tilde{j}\}$ and $\{w_j:j<\tilde{j}\}$, respectively.
Note that $|E|=|F|=\tilde{i}+\tilde{j}$,
since exactly $\tilde{j}$ many coding nodes are added to make envelopes of $Y$ and $Z$.
Let $\tilde{k}=\tilde{i}+\tilde{j}$,
and let $\{e_k:k<\tilde{k}\}$ and $\{f_k:k<\tilde{k}\}$ be the enumerations of $E$ and $F$ in order of increasing length, respectively.
For each $j<\tilde{j}$,
let $k_j$ be the index in $\tilde{k}$ such that
$e_{k_j}=v_j$ and $f_{k_j}=w_j$.
For  $k<\tilde{k}$,
let $E(k)$ denote
the tree induced by $E$ restricted to those nodes of length less than or equal to $|e_k|$; precisely,
$E(k)=\{e\re |t|: e, t\in E^{\wedge}$ and $ |t|\le \min(|e|,|e_k|)\}$.
Likewise for $F$.

If $\tilde{j}=0$, then $E=Y$ and $F=Z$, so $E\ssim F$ follows from $E\sssim F$.
Suppose now that $\tilde{j}\ge 1$.
It must be the case that $p_0>|u^Y_0|$,
since $u^Y_0$ is the stem of the tree induced by $Y$,
and $Y$ does not have any sets of parallel $1$'s below its stem.
Likewise, $l_0>|u^Z_0|$.
Let $m_0$  be the least  integer below $\tilde{m}$  such that
 $|u^Y_{m_0}|>p_0$.
Then the admissible interval $[p_0,p^*_0]$ is contained in the interval $(|u^Y_{m_0-1}|,|u^Y_{m_0}|)$,
and  moreover,
\begin{equation}
|u^Y_{m_0-1}|<|v_0^{\wedge}|<p_0\le
|v_0|\le p_0^*,
\end{equation}
 by  the definition of envelope.
Since $Y\sssim Z$,
it follows that the admissible interval $[l_0,l_0^*]$ is contained in $(|u^Z_{m_0-1}|,|u^Z_{m_0}|)$
and
\begin{equation}
|u^Z_{m_0-1}|<|w_0^{\wedge}|<l_0\le
|w_0|\le l_0^*.
\end{equation}
Thus,
$E(k_0-1)$ is exactly the tree induced by  $Y$ restricted below $|u^Z_{m_0-1}|$,
which is strongly similar to
the tree induced by $Z$ restricted below
$|u^Z_{m_0-1}|$,
this being exactly
$F(k_0-1)$.

Now suppose that $j<\tilde{j}$
and $E(k_j-1)\ssim F(k_j-1)$.
Let $m_j$ be the least integer below $\tilde{m}$ such that
$|u^Y_{m_j}|>p_j$.
Then
the only  nodes in $E^{\wedge}$
in the interval $(|u^Y_{m_j-1}|,|u^Y_{m_j}|)$
are $v_j^{\wedge}$ and $v_j$.
Likewise, the only nodes in $F^{\wedge}$ in the interval
$(|u^Z_{m_j-1}|,|u^Z_{m_j}|)$
are $w_j^{\wedge}$ and $w_j$.
Extend the strong similarity map $g:E(k_j-1)\ra F (k_j-1)$
to the map
$g^*:E(k_j)\ra F(k_j)$
as follows:
Define
$g^*=g$ on $E(k_j-1)$,
 $g^*(v_j^{\wedge})=w_j^{\wedge}$, and
 $g^*(v_j)=(w_j)$.
If the sequence  of  $0$'s of length $|v_j|$ is in $E$,
 then define $g^*$ of that node to be the sequence of $0$'s of length $|w_j|$.
For each node $s$ in $E(k_j)$ of length $|v_j|$ besides $v_j$ itself,
$s$ extends a unique maximal node $s^-$  in $E (k_j-1)$;
define $g^*(s)$ to be the unique node in $F(k_j)$ of length $|w_j|$ extending $g(s^-)$.
Note that
 each node $t$ in $E(k_j)$ of length $|v_j^{\wedge}|$, besides $v_j^{\wedge}$ itself,
is equal to $s\re |v^{\wedge}_j|$ for some unique $s$ as above;
define $g^*(t)$ to be $g^*(s)\re|w^{\wedge}_j|$.
As the only new set of parallel $1$'s in $Y$ in this interval is $I_j^Y$, which is equal to $I_j^Z$,
and as
\begin{equation}
\max(l_{j-1}^*,|u^Y_{m_j-1}|)<|v^{\wedge}_j|<p_j\le |v_j|\le p_j^*,
\end{equation}
and similarly for $w_j$,
and $v_j,w_j$ witness the parallel $1$'s indexed by
$I^Y_j, I^Z_j$, respectively,
it follows that $g^*$ is a strong similarity map from $E(k_j)$ to $F(k_j)$.

If $j<\tilde{j}-1$,
noting that the only nodes in the tree induced by $E$ with
 length in
 the interval $(|v_j|, |v^{\wedge}_{j+1}|)$
are in the tree induced by $Y$,
and likewise, all nodes in the tree induced by $F$ in the interval
$(|w_j|, |w^{\wedge}_{j+1}|)$ are in the tree induced by $Z$,
it follows that $E(k_{j+1}-1)$ is strongly similar to $F(k_{j+1}-1)$.
Then the induction continues.

To finish, when $j=\tilde{j}-1$,
all nodes in the tree induced by
$E$ in the interval $(|v_{\tilde{j}-1}|,|y_{\tilde{i}-1}|]$
are in fact nodes in $Y^{\wedge}$.
Likewise, all nodes in  the tree induced by
 $F$ in the interval $(|w_{\tilde{j}-1}|,|z_{\tilde{i}-1}|]$ are
 in $Z^{\wedge}$.
Further, all sets of parallel $1$'s in $E$ and $F$ in these intervals are already witnessed at or below
$|v_{\tilde{j}-1}|$ and $|w_{\tilde{j}-1}|$, respectively.
Thus, the strict similarity between $Y$ and $Z$ induces
an extension of the strong similarity between
$E(k_{\tilde{j}-1})$
and
$E(k_{\tilde{j}-1})$
to a strong similarity between
$E^{\wedge}$ and $F^{\wedge}$.
\end{proof}

\begin{lem}\label{lem.lastpiece}
Let
$S$ be a strongly incremental strong coding tree, a subtree of $T$.
Let $Z$ be a finite antichain  of coding nodes in  $S$, and let $E$ be any envelope of $Z$ in $T$.
Enumerate the nodes
in $Z$ and $E$ in
order of increasing length as
$\lgl z_i:i<\tilde{i}\rgl$ and
 $\lgl e_k:k<\tilde{k}\rgl$,
respectively.
Then whenever  $F\ssim E$, the subset
$F\re Z:=\{f_{k_i}:i<\tilde{i}\}$ of $F$
is strictly similar to $Z$, where
$\lgl f_k:k<\tilde{k}\rgl$
enumerates the nodes in $F$ in order of increasing length
 and for each $i<\tilde{i}$,
$k_i$ is the index such that $e_{k_i}=z_i$.
\end{lem}

\begin{proof}
Recall that $F\ssim E$ implies $F\sssim E$ and that $E$ and hence $F$ satisfy the \IPOC,
since  $E$ is an envelope of  a diagonal subset of an incremental strong coding tree.
Let   $\iota_{Z,F}:Z\ra F$ be the injective map defined via $\iota_{Z,F}(z_i)=f_{k_i}$, for each $i<\tilde{i}$,
and let
 $F\re Z$ denote  $\{f_{k_i}:i<\tilde{i}\}$, the image of $\iota_{Z,F}$.
Then $F\re Z$ is a subset of $F$ which we claim is strictly similar to $Z$.

Since $F$ and $E$ satisfy the \IPOC,
 the strong similarity map $g:E\ra F$
satisfies that for each $j<\tilde{k}$, the sets of new parallel $1$'s at level of the $j$-th coding node are equal:
\begin{equation}
\{k<\tilde{k}:e_k(|e_j|)=1\}=
\{k<\tilde{k}: g(e_k)(|g(e_j)|)=1\}
=\{k<\tilde{k}:f_k(|f_j|)=1\}.
\end{equation}
Since $\iota_{Z,F}$ is the restriction of $g$ to $Z$,
$\iota_{Z,F}$ also
takes each
  new set of parallel $1$'s  in $Z$
 to the corresponding set of new parallel $1$'s in  $F\re Z$, with the same set of indices.
Thus, $\iota_{Z,F}$ witnesses that $F\re Z$ is
strictly similar to  $Z$.
\end{proof}




\begin{thm}[Ramsey Theorem for Strict Similarity Types]\label{thm.mainRamsey}
Let $Z$ be a finite antichain of coding nodes in   a strong coding tree $T$,
and let $h$  color of all subsets of $T$ which are strictly similar to $Z$ into finitely many colors.
Then there is an incremental strong coding tree $S\le T$ such that
all subsets of $S$ strictly similar to $Z$ have the same $h$ color.
\end{thm}

\begin{proof}
First, note that there is an envelope $E$
of a copy of $Z$ in $T$:
By  Lemma \ref{lem.squiggletree},
 there is a strongly incremental strong coding tree $U\le T$ and a set of coding nodes $V\sse T$
such that each  $Y\sse U$  which is strictly similar to $Z$
has an envelope  in $T$ by adding nodes from $V$.
Since $U$ is strongly similar to $T$,
there is  subset $Y$ of $U$ which is  strictly similar  to  $Z$.
Let
$E$ be any envelope of $Y$ in $T$, using witnessing coding nodes from $V$.

By Lemma \ref{lem.envelopesbehave}, all envelopes of copies of $Z$ are strictly similar.
Define a coloring $h^*$ on $\Sim^{ss}_T(E)$ as follows:
For each $F\in \Sim^{ss}_T(E)$,
define
$h^*(F)=h(F\re Z)$,
where
 $F\re Z$ is  the  subset of $F$ provided by Lemma \ref{lem.lastpiece}.
The set $F\re Z$
is strictly similar to $Z$,
so  the coloring $h^*$ is well-defined.
Since  envelopes satisfy the \SPOC,
Theorem \ref{thm.MillikenIPOC}
yields
a strong coding tree $T'\le T$ such that $\Sim^{ss}_{T'}(E)$  is homogeneous for $h^*$.
 Lemma \ref{lem.squiggletree} implies there is an incremental strong coding tree $S\le T'$ and a set of coding nodes $W\sse T'$
such that each  $Y\sse S$  which is strictly similar to $Z$
has an envelope $F$ in $T'$.
Thus, $h(Y)=h^*(F)$.
Therefore, $h$ takes only one color on the set of all $Y\sse S$ which are strictly similar to $Z$.
\end{proof}

\begin{rem}
If $Z$ is not incremental, then $S$ will have no strictly similar copies of $Z$, since every antichain in $S$ is strongly incremental.
Thus, non-incremental antichains will not contribute to the big Ramsey degrees.
\end{rem}

\begin{rem}
The definition of envelope can be  extended to handle  any finite subset of a strong coding tree,
where maximal nodes can be any nodes in a strong coding tree rather than just coding nodes.
This is accomplished using the same definition of strict similarity type, accounting for all minimal new sets of parallel $1$'s,
and then letting envelopes consist of adding new coding nodes as before to witness these sets of parallel $1$'s in their admissible intervals.
Then
Theorem \ref{thm.mainRamsey} extends to a
Ramsey theorem  for  strict similarity types of any finite subset of a strong coding tree.
However, as the
main result of this paper
only needs Theorem \ref{thm.mainRamsey},
in order to avoid unnecessary length, we do not present the full generality  here.
\end{rem}


\section{The  universal triangle-free graph has finite big Ramsey degrees}\label{sec.7}

The main theorem of this paper, Theorem \ref{finalthm}, will now be proved:
 The universal triangle-free homogeneous graph $\mathcal{H}_3$ has finite  big Ramsey degrees.
This result will  follow from
Theorem \ref{thm.mainRamsey}, which is
the
Ramsey Theorem for Strict Similarity Types,
  along with Lemma \ref{lem.bD},  which shows that any strong coding tree contains an infinite strongly diagonal set of coding nodes  which  code the universal triangle-free graph.

Recall from the discussion in the previous section that in a strong coding tree, a set of coding nodes is strongly diagonal if and only if it is an antichain.
Given an antichain $D$ of coding nodes from a strong coding tree,
its  meet closure,
 $D^{\wedge}$  has  at most one node of any given length.
Let $L_D$ denote the set of all
lengths of nodes $t\in D^{\wedge}$
such that
 $t$ is not the splitting predecessor of any coding node in $D$.
Define
\begin{equation}\label{eq.D^*}
D^*=\bigcup\{t  \re l:t\in D^{\wedge}\setminus D
\mathrm{\ and\ }l\in L_D\}.
\end{equation}
Then $(D^*,\sse)$ is a tree.

For a strong coding tree $T$,  let $(T,\sse)$
be the reduct of $(T,\om;\sse,<,c)$.
Then  $(T,\sse)$
is simply the tree structure of $T$, disregarding the difference between coding nodes and non-coding nodes.
We say that two trees $(T,\sse)$ and $(S,\sse)$ are {\em strongly similar trees}
if they
 satisfy
Definition 3.1 in  \cite{Sauer06}.
This is the  the same as the
modification of Definition \ref{def.3.1.likeSauer}
leaving out (6) and changing (7) to apply to passing numbers of  {\em all} nodes in the trees.
When we say that two finite trees are strongly similar trees, we will be implying that
when
 extending the two trees to include the  immediate extensions of their maximal nodes, the two extensions are still strongly similar.
Thus, strong similarity of finite trees implies passing numbers of  their immediate extensions are preserved.

\begin{lem}\label{lem.bD}
Let $T\le \bT$ be a strong coding tree.
Then there is an infinite  antichain of coding nodes  $D\sse T$  which code
 $\mathcal{H}_3$ in exactly the same way that $\bT$ does:
$c^{D}_n(l^{D}_i)=c^{\bT}_n(l^{\bT}_i)$,
for all $i<n<\om$.
Moreover,
$(D^*,\sse)$  and $(\bT,\sse)$ are strongly similar trees.
\end{lem}

\begin{proof}
To simplify the indexing of the construction,
we will  construct a subtree
 $\bD\sse\bT$
such that $\bD$
the set of coding nodes in $\bD$  form
 an antichain
satisfying the lemma.
Then, since $T$ is strongly similar to $\bT$,
letting $\varphi:\bT\ra T$ be the strong similarity map  between $\bT$ and $T$,
the image of $\varphi$ on the coding nodes of  $\bD$ will yield an antichain of coding nodes $D\sse T$ satisfying the lemma.

We will construct $\bD$ so that  for each $n$,
the node of length $l^{\bD}_{n}+1$ which  is going to be extended to the next  coding node $c^{\bD}_{n+1}$ will  split at a  level lower than
 any of the other nodes of length $l^{\bD}_{n+1}$  split in $\bD$.
Above that, the splitting will be regular in the interval until  the next coding node.
Recall that  for each $i<\om$,
 $\bT$
 has either a coding node or a splitting node of length $i$.
To avoid some superscripts, let
 $l_n=|c^{\bT}_n|$
and $k_n=|c^{\bD}_n|$.
Let
$j_n$  be the index such that $c^{\bD}_n=c^{\bT}_{j_n}$, so that
 $k_n$  equals $l_{j_n}$.
The set of  nodes in  $\bD\setminus\{c^{\bD}_n\}$ of length  $k_n$
 shall be indexed  as $\{d_t:t\in \bT\re l_n\}$.

Define  $d_{\lgl\rgl}=\lgl\rgl$ and
let $\Lev_{\bD}(0)=\{d_{\lgl\rgl}\}$.
As the node $\lgl\rgl$ splits in $\bT$,
so the node $d_{\lgl\rgl}$ will split in $\bD$.
Extend $\lgl 1 \rgl$ to a splitting node in $\bT$ and label this extension $v_{\lgl 1\rgl}$.
Let $a_{\lgl 0\rgl }$ be the leftmost node in $\bT$ of length $|v_{\lgl 1\rgl}|+1$,
let $a_{\lgl 1\rgl}={v_{\lgl 1\rgl}}^{\frown}0$,
and $u_{\lgl 1\rgl}={v_{\lgl 1\rgl}}^{\frown}1$.
Extend $a_{\lgl 0\rgl }$ to the shortest  splitting node  containing it in $\bT\cap 0^{<\om}$; label this $d_{\lgl 0\rgl}$.
Let $d_{\lgl 1\rgl}$ be the leftmost extension of $a_{\lgl 1\rgl}$ in $\bT$ of length $|d_{\lgl 0\rgl}|$,
and let $u'_{\lgl 1\rgl}$ be the leftmost
extension of $u_{\lgl 1\rgl}$ in $\bT$ of length $|d_{\lgl 0\rgl}|$.
Apply Lemma \ref{lem.facts}
to extend  ${d_{\lgl 0\rgl}}^{\frown}0$, ${d_{\lgl 0\rgl}}^{\frown}1$, ${d_{\lgl 1\rgl}}^{\frown}0$, and
${u'_{\lgl 1\rgl}}^\frown 0$
to nodes
$d_{\lgl 0,0\rgl}$, $d_{\lgl 0,1\rgl}$, $d_{\lgl 1,0\rgl}$, and $c^{\bD}_0$, respectively,
so that
the tree induced by these nodes satisfy the \POC,
$c^{\bD}_0$ is a coding node,
and the immediate extension of $d_{\lgl i_0,i_1\rgl}$
in $\bT$ is $i_1$,
for all $\lgl i_0,i_1\rgl$ in $\Lev_{\bT}(2)$.
Let $k_0=|c^{\bD}_0|$, and notice that we have constructed $\bD\re(\le k_0)$ satisfying the lemma.

For the induction step,
suppose $n\ge 1$ and we have constructed $\bD\re (\le k_{n-1})$
 satisfying the lemma.
Then by the induction hypothesis,
there is a strong similarity map  of the trees
$\varphi :\bT\re(\le l_{n-1})\ra \bD^*\re (\le k_{n-1})$,
where
for each $t\in \bT\re l_{n-1}$,
$d_t=\varphi(t)$.
Let $s$ denote the  node  in $\bT\re l_{n-1}$  which  extends to the coding node $c^{\bT}_n$.
Let $v_s$ be a splitting node in $\bT$ extending $d_s$.
Let $u_s={v_s}^{\frown}1$
and extend all nodes $d_t$,
$t\in (\bT\re l_{n-1})\setminus\{s\}$,
 leftmost to length $|u_s|$ and label these $d_t'$.
Extend ${v_s}^{\frown}0$ leftmost to length $|u_s|$ and label it $d'_s$.
Let $X=\{d'_t:t\in \bT\re l_{n-1}\}\cup\{u_s\}$ and let $\Spl(u_s)$ be the set of all nodes in $X$ which have no parallel $1$'s with $u_s$.
Apply Lemma \ref{lem.facts} to obtain a coding node $c^{\bD}_n$ extending $u_s$
and nodes $d_w$, $w\in \bT\re l_n$,
so that,
letting  $k_n=|c^{\bD}_n|$
and
\begin{equation}
\bD\re k_n=
\{d_m:m\in \bT\re l_n\}\cup\{c^{\bD}_n\},
\end{equation}
 the following hold.
$\bD\re(\le k_n)$ satisfies the \POC,
and
$\bD^*\re(\le k_n)$ is strongly similar  as  a tree
to
$\bT\re (\le l_n)$.
Thus, the coding nodes in $\bD\re (\le k_n)$
code exactly the same graph as the coding nodes in $\bT\re (\le l_n)$.

Let $\bD=\bigcup_{n<\om}\bD\re(\le k_n)$.
Then the set of coding nodes in  $\bD$  forms an antichain of maximal nodes in $\bD$.
Further,
the tree generated by the
 the meet closure of the set $\{c^{\bD}_n:n<\om\}$
is exactly  $\bD$,
and
 $\bD^*$  and  $\bT$
are strongly similar as trees.
By the construction,
for each pair $i<n<\om$, $c^{\bD}_n(k_i)=c^{\bT}_n(l_i)$;
 hence they code $\mathcal{H}_3$ in the same order.

To finish,  let $\psi$ be the strong similarity  map   from $\bT$ to $S$.
Letting $D$ be the $\psi$-image of $\{c^{\bD}_n:n<\om\}$,
we obtain an antichain  of coding nodes in $S$
such that $D^*$ and $\bD^*$ are strongly similar trees,
and hence $D^*$ is strongly similar as  a tree to $\bT$.
Thus, the antichain of coding nodes $D$
 codes $\mathcal{H}_3$ and satisfies  the lemma.
\end{proof}

\begin{figure}
\begin{tikzpicture}[scale=.3]

\foreach \x in {3}{
\foreach \y in {0}{
\node  at (\x,\y) {};
}}

\foreach \x in {-5,11}{
\foreach \y in {2}{
\node at (\x,\y) {};
}}
\draw (-5,2)--(3,0)--(11,2);

\foreach \x in {-6,10,12}{
\foreach \y in {3}{
\node at (\x,\y){};
}}
\draw (-6,3)--(-5,2);
\draw (10,3)--(11,2)--(12,3);

\foreach \x in {-8,8,10}{
\foreach \y in {5}{
\node at (\x,\y){};
}}
\draw (-8,5)--(-6,3);
\draw (8,5)--(10,3);
\draw (10,5)--(12,3);

\foreach \x in {-9,-7,7,9}{
\foreach \y in {6}{
\node at (\x,\y){};
}}
\draw (-9,6)--(-8,5)--(-7,6);
\draw (7,6)--(8,5);
\draw (9,6)--(10,5);

\foreach \x in {-12,-7,4,9}{
\foreach \y in {9}{
\node at (\x,\y){};
}}
\draw (-12,9)--(-9,6);
\draw (-7,6) to [out= 30, in=-110]  (-7,9);
\draw (4,9)--(7,6);
\draw (9,6) to [out= 120, in=-70]  (9,9);

\foreach \x in {-13,-6,3}{
\foreach \y in {10}{
\node at (\x,\y){};
}}
\draw (-13,10)--(-12,9);
\draw (-6,10)--(-7,9);
\draw (3,10)--(4,9);

\foreach \x in {-16,-6,0}{
\foreach \y in {13}{
\node at (\x,\y){};
}}
\draw (-16,13)--(-13,10);
\draw (-6,10) to [out= 45, in=-120] (-6,13);
\draw (0,13)--(3,10);

\foreach \x in {-17,-7,-5,-1}{
\foreach \y in {14}{
\node at (\x,\y){};
}}
\draw (-17,14)--(-16,13);
\draw (-7,14)--(-6,13)--(-5,14);
\draw (-1,14)--(0,13);

\foreach \x in {-19,-9,-7,-1}{
\foreach \y in {16}{
\node at (\x,\y){};
}}
\draw (-19,16)--(-17,14);
\draw(-9,16)--(-7,14);
\draw (-7,16)--(-5,14);
\draw (-1,14) to [out= 110, in=-70] (-1,16);

\foreach \x in {-20,-10,-8,-2,0}{
\foreach \y in {17}{
\node at (\x,\y){};
}}
\draw (-20,17)--(-19,16);
\draw (-10,17)--(-9,16);
\draw(-8,17)--(-7,16);
\draw (-2,17)--(-1,16)--(0,17);

\foreach \x in {-21,-11,-9,-3,-1}{
\foreach \y in {18}{
\node at (\x,\y){};
}}
\draw (-21,18)--(-20,17);
\draw(-11,18)--(-10,17);
\draw(-9,18)--(-8,17);
\draw(-3,18)--(-2,17);
\draw(-1,18)--(0,17);

\foreach \x in {-22,-20,-12,-10,-4,-2}{
\foreach \y in {19}{
\node at (\x,\y){};
}}
\draw (-22,19)--(-21,18)--(-20,19);
\draw(-12,19)--(-11,18);
\draw(-10,19)--(-9,18);
\draw(-4,19)--(-3,18);
\draw(-2,19)--(-1,18);

\foreach \x in {-25,-20,-15,-10,-7,-2}{
\foreach \y in {22}{
\node at (\x,\y){};
}}
\draw (-25,22)--(-22,19);
\draw(-20,19) to [out= 50, in=-120] (-20,22);
\draw(-15,22)--(-12,19);
\draw  (-10,19)  to [out= 120, in=-60] (-10,22);
\draw (-7,22)--(-4,19);
\draw(-2,19)  to [out= 150, in=-35](-2,22);

\foreach \x in {-26,-19,-16,-8,-1}{
\foreach \y in {23}{
\node at (\x,\y){};
}}
\draw (-26,23)--(-25,22);
\draw (-19,23)--(-20,22);
\draw(-16,23)--(-15,22);
\draw (-8,23)--(-7,22);
\draw(-1,23)--(-2,22);

\node[right] at (2,-1) {$d_{\lgl \rgl}$};

\node[right] at (10,1.2) {$v_{\lgl 1\rgl}$};
\node[right] at (12,2.8) {$u_{\lgl 1\rgl}$};
\node[left] at (-8,5) {$d_{\lgl 0\rgl}$};
\node[left] at (8,5)  {$d_{\lgl 1\rgl}$};
\node[right] at (10,5) {$u_{\lgl 1\rgl}'$};

\node[left] at (-12,9) {$d_{\lgl 0,0\rgl}$};
\node[right] at (-7,9) {$d_{\lgl 0,1\rgl}$};
\node[left] at (4,9) {$d_{\lgl 1,0\rgl}$};
\node[right] at (8,10){$c_0^{\bD}$};

\node[left] at (-6,13){$v_{\lgl 0,1\rgl}$};
\node[right] at (-5.2,13.8){$u_{\lgl 0,1\rgl}$};
\node[right] at (-5.2,13.8){$u_{\lgl 0,1\rgl}$};

\node[right] at (-1,16) {$d_{\lgl 1,0,0\rgl}$};

\node[right] at (-1,18) {$d_{\lgl 1,0,0,1\rgl}$};
\node[left] at (-2.8,18) {$d_{\lgl 1,0,0,0\rgl}$};

\node[left] at (-21,18) {$d_{\lgl 0,0,0,0\rgl}$};
\node[left] at  (-25,22) {$d_{\lgl 0,0,0,0,0\rgl}$};
\node[left] at (-9.1,22.8) {$c^{\bD}_1$};

\node[left] at (-19,16) {$d_{\lgl 0,0,0\rgl}$};
\node[left] at (-9,16) {$d_{\lgl 0,1,1\rgl}$};
\node[left] at (-10.8,18) {$d_{\lgl 0,1,1,0\rgl}$};
\node[left] at (-17.3,21.2) {$d_{\lgl 0,0,0,0,1\rgl}$};

\node[right] at (-2,22) {$d_{\lgl 1,0,0,1,0\rgl}$};

\node[left] at (-12,21.2) {$d_{\lgl 0,1,1,0,0\rgl}$};
\node[right] at (-7.8,21.2) {$d_{\lgl 1,0,0,1,0\rgl}$};

\node[circle, fill=black,inner sep=0pt, minimum size=5pt] at (3,0) {};

\node[circle, fill=black,inner sep=0pt, minimum size=5pt] at (-8,5) {};
\node[circle, fill=black,inner sep=0pt, minimum size=5pt] at (8,5) {};

\node[circle, fill=black,inner sep=0pt, minimum size=5pt] at (-12,9) {};
\node[circle, fill=black,inner sep=0pt, minimum size=5pt] at (-7,9) {};
\node[circle, fill=black,inner sep=0pt, minimum size=5pt] at (4,9) {};

\node[circle, draw, inner sep=0pt, minimum size=5pt] at (9,9) {};
\node[circle, draw, inner sep=0pt, minimum size=5pt] at (-10,22) {};

\node[circle, fill=black,inner sep=0pt, minimum size=5pt] at (-1,16) {};
\node[circle, fill=black,inner sep=0pt, minimum size=5pt] at (-9,16) {};
\node[circle, fill=black,inner sep=0pt, minimum size=5pt] at (-19,16) {};

\node[circle, fill=black,inner sep=0pt, minimum size=5pt] at (-21,18) {};
\node[circle, fill=black,inner sep=0pt, minimum size=5pt] at (-11,18) {};
\node[circle, fill=black,inner sep=0pt, minimum size=5pt] at (-3,18) {};
\node[circle, fill=black,inner sep=0pt, minimum size=5pt] at (-1,18) {};

\node[circle, fill=black,inner sep=0pt, minimum size=5pt] at (-25,22) {};
\node[circle, fill=black,inner sep=0pt, minimum size=5pt] at (-20,22) {};
\node[circle, fill=black,inner sep=0pt, minimum size=5pt] at (-15,22) {};
\node[circle, fill=black,inner sep=0pt, minimum size=5pt] at (-7,22) {};
\node[circle, fill=black,inner sep=0pt, minimum size=5pt] at (-2,22) {};

\end{tikzpicture}
\caption{The construction of $\bD$}
\end{figure}

The filled-in nodes in the graphic  form the tree $\bD^*$.
The coding nodes are exactly the maximal nodes of $\bD$ and form an antichain.
Notice that the collection of nodes $\{d_t:t\in \bT\re (\le 2)\}$, which are exactly the filled-in nodes in the figure,
forms a tree strongly similar to $\bT\re 2$.
The bent lines indicate that the next node was
chosen
 either  to be least such that it was a critical node or  according to Lemma \ref{lem.facts}.

\begin{mainthm}\label{finalthm}
The universal triangle-free graph has finite big Ramsey degrees.
\end{mainthm}

\begin{proof}
Let $\G$ be a finite triangle-free graph,
and let $f$ be a coloring of all the copies of $\G$ in $\mathcal{H}_3$ into finitely many colors.
By Theorem \ref{thm.cool}, there is a
 strong coding tree $\bT$ in which the coding nodes code $\mathcal{H}_3$.
Let $\mathcal{A}$ denote the set of all
antichains of coding nodes
 of $\bT$ which code a copy of $\G$.
For each $Y\in\mathcal{A}$,
let
$h(Y)=f(\G')$,
where
$\G'$ is the copy of $\G$
coded by the coding nodes in $Y$.
Then $h$  is a finite coloring on $\mathcal{A}$.

Let $n(\G)$ be the number of different strict similarity types
of incremental strongly diagonal subsets of $\bT$ coding $\G$,
and let $\{Z_i:i<n(\G)\}$ be a set of one representative from  each of  these  different strict similarity types.
Successively
apply Theorem \ref{thm.mainRamsey}
  to obtain incremental strong coding trees $\bT\ge T_0\ge\dots\ge T_{n(\G)-1}$ so that for each $i<n(\G)$,
$h$ is takes only one color on
$\Sim^s_{T_i}(Z_i)$.
Let $S=T_{n(\G)-1}$.

By Lemma \ref{lem.bD}
there is a strongly diagonal subtree $D\sse S$ which also codes $\mathcal{H}_3$.
Then every set of coding nodes in $D$ coding $\G$ is automatically strongly diagonal and incremental.
Therefore, every copy  of $\G$ in the copy of $\mathcal{H}_3$ coded by the coding nodes in $\D$
is coded by an incremental strongly diagonal set.
Thus, the number of
 strict similarity types of  incremental strongly  diagonal  subsets of $\bT$  coding $\G$
provides an upper bound for
the big Ramsey degree of $\G$ in $\mathcal{H}_3$.
\end{proof}


\section{Concluding Remarks}

The number   of strict similarity types of antichains of coding nodes in a strong coding tree
which code
 a given finite graph $\G$
is bounded by the number of subtrees of the binary tree of height $2(|\G|+1)$,
times the number of ways to choose incremental sets of new parallel $1$'s between any successive levels of the tree.
We leave it as an open problem to determine this recursive function precisely.

Although we have not yet proved the lower bounds
to obtain the precise
 big Ramsey degrees  $T(\G,\mathcal{K}_3)$ for finite triangle-free graphs inside the universal triangle-free graph, we conjecture that they will be equal to the number of strict similarity
 types of  strongly
incremental
antichains  coding $\G$.
We further conjecture that once found, the lower bounds will satisfy the conditions   needed for  Zucker's work in \cite{Zucker19} to apply.
If  so, then  $\mathcal{H}_3$ would admit a big Ramsey structure and  any big Ramsey flow will be a universal completion flow, and any two universal completion flows will be universal.
We refer the interested reader to
Theorem 1.6 in \cite{Zucker19} and surrounding comments.


The author is currently working to extend
the techniques developed here to prove that for each $k>3$,
the universal $k$-clique-free homogeneous graph $\mathcal{H}_k$ has finite big Ramsey degrees.
Preliminary analyses indicate that the
methodology
created in this paper is robust enough
 to  apply, with  modifications,  to a large class of \Fraisse\ limits of \Fraisse\ classes of relational structures omitting some irreducible substructure.




\bibliographystyle{amsplain}
\bibliography{references}

\end{document}